\documentclass[11pt,reqno]{amsart}

\usepackage{amsmath}
\usepackage{dsfont}  
\usepackage{amsthm}
\usepackage{graphicx}
\usepackage{mathtools}
\usepackage{amsfonts}
\usepackage{amssymb}
\usepackage{hyperref}
\usepackage{enumerate}
\usepackage{enumitem}
\usepackage{bm}
\usepackage{comment}
\usepackage[margin=1 in]{geometry}
\usepackage[normalem]{ulem}
\usepackage{tikz}
\usepackage{tikz-cd}
\usepackage{xcolor}
\usetikzlibrary{matrix,arrows,positioning,automata}
  \usepackage{cancel}

\usepackage[numbers]{natbib}

\usepackage{todonotes}

\numberwithin{equation}{section}
\newtheorem{theorem}{Theorem}[section]
\newtheorem{lemma}[theorem]{Lemma}
\newtheorem{proposition}[theorem]{Proposition}
\newtheorem{corollary}[theorem]{Corollary}
\newtheorem{assumption}[theorem]{Assumption}

\theoremstyle{definition}

\newtheorem{definition}[theorem]{Definition}
\newtheorem{remark}[theorem]{Remark}

\def\R{{\mathbb R}}

\def\P{{\mathcal P}}

\DeclareMathOperator*{\argmin}{arg\,min}

\newcommand\cE{\mathcal E}
\newcommand\cF{\mathcal F}

\newcommand\cP{\mathcal P}

\def \P{\mathbb{P}}

\def \R{\mathbb{R}}

\DeclareMathOperator*{\esssup}{ess\,sup}

\thanks{The first and third authors (in alphabetical order) acknowledge financial support by AFOSR grant FA9550-23-1-0324.\ The second and third authors acknowledge financial support by NSF CAREER award DMS-2143861.}
 
\title[Conditional McKean-Vlasov control]{Conditional McKean-Vlasov control}%{The weak formulation for McKean-Vlasov control: Conditioned interaction and the Pontryagin principle}%
\author{Ren\'e Carmona, Ludovic Tangpi and Kaiwen Zhang}

\address{Princeton University}
\email{rcarmona@princeton.edu}
\email{ludovic.tangpi@princeton.edu}
\email{kezhang@princeton.edu}

\begin{document}
\maketitle

\begin{abstract}

Conditional McKean--Vlasov control problems involve controlling McKean--Vlasov diffusions where the interaction occurs through the law of the state process conditionally on it staying in a domain. Introduced by Lions in his 2016 lectures at the Collège de France, these problems have notable applications, particularly in systemic risk.
We establish well-posedness and provide a general characterization of optimal controls using a new Pontryagin maximum principle in the \emph{probabilistic weak formulation}. Unlike the classical approach based on forward--backward systems, our results connect the control problem to a generalized McKean--Vlasov backward stochastic differential equation (BSDE).
We illustrate our framework with two applications: a version of the Schrödinger problem with killing, and a construction of equilibria in potential mean field games via McKean--Vlasov control.
\end{abstract}

\tableofcontents

\section{Introduction}
This paper aims to advance our understanding of the class of McKean–Vlasov stochastic optimal control problems in which the controlled particle is considered \emph{conditionally on remaining within a given domain}.
Before providing proper definitions below, let us sketch the problem here: Let the $\mathbb R^d$-valued state process $X$ be a weak solution to the controlled McKean–Vlasov stochastic differential equation (SDE)
\begin{equation}\label{eq: state}
    X_t = \xi + \int_0^t b(s, X_{\cdot \wedge s}, \alpha_s, \mathcal{L}_{\mathbb{P}^\alpha}(X_{\cdot \wedge s}\vert s < \tau), \mathbb{P}^\alpha[s<\tau])ds + \int_0^t \sigma(s, X_{\cdot \wedge s}) dW^\alpha_s.
\end{equation}
where $ \tau$ denotes the exit time of  $X$ from a domain $ D\subseteq \mathbb R^d$, defined by $\tau = \inf\{ t \geq 0 \mid X_t \notin D \}$, and $\mathbb{P}^\alpha$ is the probability measure, defined in \eqref{eq: p def}, under which $W^\alpha$ is a Wiener process. 
With $\mathcal{L}_{\mathbb{P}^\alpha}(\cdot \vert s < \tau)$, we denote the law of a random variable under $\mathbb{P}^\alpha$ conditioned on the event $\lbrace s < \tau \rbrace$.
By controlling the process \( \alpha \), the goal is to minimize the cost functional
\begin{equation}\label{eq: def J}
\small{
    J(\alpha) := \mathbb{E}^\alpha\bigg[\int_0^{T \wedge\tau} f(s, X_{\cdot \wedge s}, \alpha_s, \mathcal{L}_{\mathbb{P}^\alpha}(X_{\cdot \wedge s}\vert s < \tau), \mathbb{P}^\alpha[s < \tau]) ds + \mathbf{1}_{\lbrace T < \tau \rbrace}g(X, \mathcal{L}_{\mathbb{P}^\alpha}(X\vert T < \tau), \mathbb{P}^\alpha[T < \tau]) \bigg].}
\end{equation}
This problem can be understood as a generalization of the conditional exit control problem first proposed by P.-L. Lions in his 2016 lectures at the Collège de France \cite{LionsCours} which, in a simplified form, is given on a probability space $(\Omega, \cF, \P)$ with Brownian motion $W$ by:
\[
\min_\alpha \check{J}(\alpha), \quad \text{with} \quad dX^\alpha_t =\alpha_t dt + \sigma dW_t, \quad \check{J}(\alpha) = \int_0^T \mathbb{E}[\check{f}(X_s^\alpha, \alpha_s) \mid s < \tau] ds + \mathbb{E}[\check{g}(X_T^\alpha) \mid T < \tau].
\]
The problem with cost \eqref{eq: def J} generalizes the natural weak formulation of the conditional exit control problem.
To see this, it suffices to put $f(s, x, a, \mu, p)=\frac{\check{f}(x_s, a)}{p}$, and $g(x, \mu, p)=\frac{\check{g}(x_T)}{p}$.
% we see that the problem we described above contains a weakly formulated version of  as a particular case.
As we explain in more details below, this class of stochastic control problems has received a growing interest due to its connection with systemic risk and the control of Fleming–Viot processes.\par
We are not able to simply rely on established results on McKean-Vlasov control due to lack of regularity. Indeed, rewriting the interaction coefficients as functions of the (unconditioned) law amount to  $\bar{g}(x, \mu) = g(x, \frac{\mu(\cdot \cap \lbrace T<\tau \rbrace)}{\mu(T<\tau)}, \mu(T<\tau))$ e.g. for the terminal cost.
% as well as $\bar{f}$ and $\bar{b}$ accordingly, 
This introduces significant singularities that violate the regularity assumptions typically required to establish the stochastic maximum principle, see e.g.\ \cite{CarmonaDelarue}.\par
A common workaround in the literature is to adopt the so-called \emph{soft-killing} formulation, in which particles are not immediately removed from the system but instead are killed at an independent exponential clock for which particles accumulate intensity outside the domain. This relaxed formulation has been studied, for instance, in \cite{Carmona24, Daudin25, HamblyJettkant231, HamblyJettkant232, HamblyJettkant24}.\par
In contrast, our goal is to address the original and more natural version of the problem, namely the \emph{hard-killing} formulation, in which particles are instantaneously removed upon exiting the domain. Overcoming the resulting regularity issues is a core obstacle we face.
\subsection{Main contributions}
Beyond establishing conditions for existence and uniqueness of optimal controls, the main objective of this paper is to characterize optimal controls in the current setting. Such characterizations are essential for analyzing both the value function and the structure of optimal control strategies. 
As in classical 
%(non)McKean-Vlasov 
control problems, we do so by establishing the stochastic Pontryagin maximum principle: Consider the Hamiltonian $h(t, x, a, \mu, p, z) = f(t, x, a, \mu, p) + (\sigma^{-1} b)(t, x, a, \mu, p)^\top z$ and adjoint backward equation
\begin{eqnarray}\label{eq: adj eq intr}
    &&Y^\alpha_{t\wedge \tau} = \mathbf{1}_{\lbrace T < \tau \rbrace} \Big( g(\Theta^\alpha_T) + \tilde{\mathbb{E}}^\alpha\Big[\frac{\delta g}{\delta m}(\tilde{\Theta}^\alpha_T, X)\vert T < \tilde{\tau}\Big] + \mathbb{E}^\alpha[\mathbf{1}_{\lbrace T < \tau \rbrace}g_p(\Theta^\alpha_T)] \Big)\\
    &+& \int_{t\wedge \tau}^{T \wedge \tau} \Big(h(\Theta^\alpha_s, Z^\alpha_s) + \tilde{\mathbb{E}}^\alpha\Big[\frac{\delta h}{\delta m}(\tilde{\Theta}^\alpha_s, \tilde{Z}^\alpha_s, X_{\cdot \wedge s})\vert s < \tilde{\tau} \Big] + \mathbb{E}^\alpha[\mathbf{1}_{\lbrace s<\tau \rbrace} h_p(\Theta^\alpha_s, Z^\alpha_s)]\Big)ds -\int_{t\wedge\tau}^{T \wedge \tau} Z^\alpha_s dW_s \nonumber
\end{eqnarray}
where $\Theta^\alpha_t = (t, X_{\cdot \wedge t}, \alpha_t, \mathcal{L}_{\mathbb{P}^\alpha}(X_{\cdot \wedge t}\vert t < \tau), \mathbb{P}^\alpha[t < \tau])$ and $\delta g/\delta m$ and $\delta h/\delta m$ are the linear derivatives of $g$ and $h$ respectively, and $g_p$, $h_p$ are derivatives with respect to $p$. 
Our first main result states that a necessary condition for optimality of controls is the pointwise minimization of the Hamiltonian over the action space $A$:
\begin{equation}\label{eq: ptw cnd}
    \alpha_t \in \argmin_{a\in A} h(\Theta^\alpha_t, Z^\alpha_t)\quad dt\times \mathbb{P} \text{ a.e. on} \quad \lbrace t < \tau \rbrace.
\end{equation}
Under an additional convexity condition, \eqref{eq: ptw cnd} becomes sufficient for optimality as well.
This result builds on the maximum principle for (non McKean-Vlasov) control problems in the weak formulation proposed by Cvitanić and Zhang \cite{Cvitanic12}.
Although our results are formulated in the conditional framework, they still apply --and remain both new and interesting-- even in the unconditional case.\par
The adjoint equation \eqref{eq: adj eq intr} used in our necessary condition of optimality can be understood as a generalized McKean-Vlasov backward stochastic differential equation (BSDE) as introduced by Possama\"i and Tangpi \cite{PossamaiTangpi}. 
Although we do not discuss general well-posedness of such BSDEs, in the convex framework, we still establish existence and uniqueness of an optimal control and thus also of solutions of our adjoint equation. The convexity condition we impose can sometimes be difficult to verify, and relaxing it seems an interesting avenue for future research.
At the technical level, we do not impose regularity in the state variable, and need existence of derivatives in the measure argument only in the linear sense, not in the sense of Lions' derivative.
Moreover, we do not require the drift or the controls to be uniformly bounded as typically assumed in the literature.
Nevertheless, we can show that under suitable conditions, the optimal control admits very strong integrability properties. 
In particular, it belongs to some BMO space. 
These properties are derived using a truncation argument that would allow to naturally generalize results on games and control in the weak formulation, notably those in \cite{Carmona15, Z24}.

We finish the work with two applications: 
First, we consider a variant of the Schrödinger problem, which originally concerns the most likely evolution of a particle system given initial and terminal distributions (see \cite{Leonard14}). 
In our version, particles are killed if they exit a domain, and we are interested in the most likely evolution of such particles given an initial distribution on $\mathbb R^d$ and a terminal distribution on $D$ as well as a prescribed survival ratio. 
We analyze this problem using a penalization approach similar to that of \cite{Hernandez25} that allows to reduce this constrained problem into conditional McKean-Vlasov problems analyzed in this work. 
Second, we revisit the connection between McKean-Vlasov optimal control and potential mean field games. 
It turns our that combining our Pontriagin maximum principle with the insight of \cite{PossamaiTangpi} on mean field games allows to derive an easy proof of the fact that solutions of McKean-Vlasov control problems give rise to equilibrium strategies for associated potential mean field games defined e.g.\ in \cite{LasryLions, CarmonaDelarue}.
Although this result appears to be folklore, the first proof in a general framework seems to be due to \cite{Hofer2}.
In summary, the main results of this paper are:
\begin{itemize}
    \item A general Pontryagin maximum principle for conditional McKean--Vlasov control problems in the weak formulation.
    
    \item Existence and uniqueness of optimal controls under sufficient convexity property of the coefficient and mild regularity. The optimal controls are shown to belong to a {BMO space} even without boundedness assumptions on the drift or control.
    
    \item Application of our results to a version of the Schrödinger problem with killing; and to present an easy proof of the connection between McKean-Vlasov control problems and potential mean field games.
    
\end{itemize}
\subsection{The role of the weak formulation}
The stochastic control problem described above is in a probabilistic weak formulation, where the drift of the state $X$ is indirectly controlled by changing the underlying probability measure $\mathbb{P}^\alpha$. In contrast, the strong formulation of the problem would require to work with one fixed measure $\mathbb{P}$ under which $X$ needs to be a strong solution of \eqref{eq: state} and under which the expectation in \eqref{eq: def J} is taken. In the strongly formulated setting, proving a maximum principle would require to differentiate the exit time of the controlled state variable, which appears to be rather challenging. As the weak formulation is well known to have the advantage to require close to no regularity in the state variable, we chose this formulation to avoid regularity issues caused by the exit time. 
Moreover, even in the unconditional case $D = \mathbb R^d$, we are not aware of any results providing a probabilistic characterizations of solutions of McKean-Vlasov control problems in the weak formulation.\par

Although we do not carefully discuss the link between the weak and strong formulations of the problem, in light of the recent result of \cite{Lacker25} regarding the mimicking of processes with hard killing, both problems are expected to coincide for fairly general models.\par
Surprisingly, the adjoint equation \eqref{eq: adj eq intr} we derived does not involve derivatives with respect to the state $X$ as it is otherwise common for the dual variable in Pontryagin's maximum principle. 
To illustrate why, let us briefly present our approach to deriving the Pontryagin principle in the weak formulation 
%On the contrary, in
by considering the standard stochastic control setting which we recover when $D = \mathbb{R}^d$ and our coefficients depend neither on the law of $X$ nor on the exit probability.
In this case, \eqref{eq: adj eq intr} reduces to
\begin{equation}\label{eq: adj std}
    Y^\alpha_t = g(X) + \int_t^T h(s, X_{\cdot \wedge s}, \alpha_s, Z^\alpha_s)ds - \int_t^T Z^\alpha_s dW_s
\end{equation}
which is the backward equation arising from a dynamic programming approach, see e.g.\ \cite{Karoui}. 
%As it is also noted in \cite{Cvitanic12}, this discrepancy is due to the use of the weak formulation instead of the strong one.\par
%To provide some added explanation for the unexpected form of the adjoint equation,  The
Since the measure $\P^\alpha$ has density with respect to $\P$ given by the stochastic exponential $\mathcal{E}^\alpha_T = \mathcal{E}(\int_0^\cdot b(\Theta^\alpha_s) dW_s)_T$,
% that will be used as a density process to define $\mathbb{P}^\alpha$ in \eqref{eq: state}, 
we can consider it to be an additional state variable as well by rewriting the problem as 
$$\quad d\begin{pmatrix} X_t \\ \mathcal{E}^\alpha_t \end{pmatrix} = \begin{pmatrix} \sigma(t, X_{\cdot \wedge t}) \\ \mathcal{E}^\alpha_t(\sigma^{-1} b)(t, X_{\cdot \wedge t}, \alpha_t) \end{pmatrix} dW_t, \quad J(\alpha) = \mathbb{E}\Big[\int_0^T \mathcal{E}^\alpha_s f(s, X_{\cdot \wedge s}, \alpha_s)ds + \mathcal{E}^\alpha_Tg(X)\Big].$$
This leads us back to the strong formulation on witch the (extended) controlled state process $(X_t, \cE_t^\alpha)_{t\in [0,T]}$.
%By considering this reformulation as a strongly formulated control problem and 
Now, applying the standard Pontryagin principle, it is easy to check that the relevant adjoint equation is given by \eqref{eq: adj std}, and optimality is indeed characterized through \eqref{eq: ptw cnd}. This demonstrates that although our approach fundamentally relies on the Pontryagin maximum principle, our results are still structurally closely related to the dynamic programming framework. 
The technical challenge is to make this extension of the state space rigorous in the McKean-Vlasov case (with conditioning), which involves a careful analysis of the dependence of $\mathcal{L}_{\mathbb{P}^\alpha}(X_{t\wedge \cdot}\vert t<\tau)$ on $\mathcal{E}^\alpha_t$ to properly derive the variational process associated with $\mathcal{E}^\alpha$.
%The success of our approach underlines the link between the maximum principle and dynamic programing.
 
 %For the general setting of our problem that we have described at the beginning, this strong formulation approach is not immediately applicable due to the non standard dependence of the law of $X$ under $\mathbb{P}^\alpha$ on $\mathcal{E}^\alpha$. Still, the main idea of focusing on $\mathcal{E}^\alpha$ instead of $X$ remains similar.\par
\subsection{Related work}

% There have been multiple groups of work that study related systems that admit such or similar interactions. First, let us bring up the conditional exit control problem initiated by Lions in his lectures at the Collège de France in November 2016 \cite{LionsCours}. This problem is close to a standard stochastic control problem, but instead of minimizing the usual expected cost, the expectation is taken conditioned on remaining in the domain. That is, the problem admits the controlled state dynamics and the cost
% $$dX_t = \alpha_t dt + \sigma dW_t, \; \; J(\alpha) = \int_0^T \mathbb{E}[f(X_s, \alpha_s)\vert s < \tau] ds+ \mathbb{E}[g(X_T) \vert T < \tau].$$
% Such control problems can be fitted into the more general problem we formulated above. 
Interests in conditional McKean-Vlasov control problem (also termed conditional exit problem) started with Lions \cite{LionsCours} and Achdou, Laurière and Lions \cite{Achdou21}. Carmona, Laurière and Lions \cite{Carmona24} introduce a soft-killing version that is further discussed in Carmona and Daudin \cite{Daudin25}. For the hard killing version, Carmona and Lacker \cite{Lacker25} establish a mimicking result that relates open loop controls to feedback controls. Along similar lines, Jettkant \cite{Jettkant24} also provides a reformulation of the conditional exit control problem using a McKean-Vlasov version of the Fleming-Viot process, complementing earlier work by Tough and Nolen \cite{Tough} showing how McKean-Vlasov equations with conditional interaction as \eqref{eq: state} arise as the limit of Fleming-Viot particle systems.\par
% have also proven a limit result for the corresponding $N$-particle system.\par
In a recent preprint, Cardaliaguet, Jackson, and Souganidis \cite{Cardaliaguet25} characterize the control problem of sub probabilities via an infinite dimensional HJB equation. Such problems are covered by our framework as well, and for our sufficient condition, we also rely on the geometry on the space of sub probability measures. In contrast to our work, the dependence of the coefficients with respect to the subprobability cannot be separated into effects by $p$ and $\mu$, i.e. the mass and the conditioned measure, and the cost functions we consider in section \ref{sub:convex_cost_functions} fail to be continuous in their setting. In \cite{Cardaliaguet25}, the authors additionally discuss convergence of a corresponding $N$ particle problem, extending the well known limit results for McKean Vlasov control, see e.g.\ Lacker \cite{Lacker17}. We expect a similar construction to be possible as well for the problem we consider.\par
Particle systems with interaction through the conditioned law commonly arise in systemic risk models dealing with default contagion. In such models, banks are affected by the defaults of other banks and interact only with those banks that have not defaulted yet.
%, these effects occur through the terms $\mathbb{P}[t <\tau]$ and $\mathcal{L}(X_t\vert t<\tau)$. 
Earlier works analyzing the behavior of such particle systems include \cite{Bush11, Spiliopoulos15, Hambly17, Giesecke20}. Instantaneous default contagion effects are related to the supercooled Stefan problem, see e.g.\ \cite{Nadtochiy19, Hambly19, Delarue22, Cuchiero23, Baker25, Guo25}. The control of such systems has been studied by Hambly and Jettkant \cite{HamblyJettkant231, HamblyJettkant232, HamblyJettkant24} in the soft-killing case using SPDE methods.\par
Systems with conditional interactions with hard-killing have been studied in the context of \emph{mean field game} beginning with the work of 
%models with absorption. This was first done also using a weak probabilistic framework by
 Campi and Fischer \cite{Campi18}, and further extended notably  in 
 %Campi, Ghio, and Livieri 
 \cite{Campi21,Burzoni23}.
 % and applied to a bank model by Burzoni and Campi \cite{Burzoni23}. In \cite{Campi21}, the model was extended to consider interaction trough the empirical sub probability measure of surviving particles. This has a similar effect as allowing interaction through the conditional law $\mathcal{L}(X_t \vert t< \tau)$ and the probability of survival $\mathbb{P}[t<\tau]$ as in our model. Although in the mean field equilibrium, the state will follow a McKean-Vlasov SDE with conditional interaction terms, 
 Let us emphasize that the difficulties faced when studying mean field games are very different.
 In the games setting, the optimization step is rather standard as the conditional laws are held fixed, whereas, in the McKean-Vlasov control problem, one needs to take into account how the conditional law is directly affected by the control.
 % the difference to the McKean-Vlasov control problem we face. In mean field game models, the interaction terms are initially held constant so that one first considers a standard stochastic control problem and achieves the McKean-Vlasov property only after optimizing through an auxiliary fixed point condition. In contrast, in a McKean-Vlasov control problem, one directly controls a McKean-Vlasov system, so that one in particular has to take in account how ones control is not only affecting the state, but also its law. This discrepancy shows itself in the required regularity properties for the mean field terms.\par

Our results also contribute to the literature on \emph{unconditional} McKean--Vlasov control problems. A well-developed analytic theory has emerged for the McKean-Vlasov control problem; see, for example, \cite{SonerYan, bayraktar2018randomized} and references therein. These works typically characterize the value function as the unique viscosity solution of a Hamilton--Jacobi--Bellman equation on the Wasserstein space. 
% In some cases, a verification argument can be used to identify optimal controls.
Probabilistic characterizations have mostly been confined to the \emph{strong formulation}, as in \cite{CarmonaDelarue15, acciaio2019extended}, where the value function and optimal controls are characterized by a fully coupled forward--backward SDE (FBSDE) system. In contrast, relatively little is known in the \emph{weak formulation}. Some recent contributions—e.g., \cite{Djete22, pham2017dynamic, Djete21}—establish a dynamic programming principle (DPP). However, it remains unclear how to derive a BSDE characterization of the value function from DPP, as is possible in the classical (non-McKean-Vlasov) case; see \cite{PossamaiTangpi}.
The present paper offers a new perspective by providing a maximum principle that yields a characterization of optimal controls through a generalized BSDE. We also note the recent work by Djete \cite{Djete25}, which introduces Wasserstein BSDEs as an alternative approach to the characterization problem in McKean--Vlasov control.

\par
Our paper is structured as follows. In the next section we formally present the conditional McKean-Vlasov control problem in the weak formulation, and state our main results. In section \ref{sct: expl}, we discuss the assumption and provide many examples. Sections \ref{sct: ptry} and \ref{sct: exst} deal with the technical details required for our proof of the Pontryagin maximum principle, as well as well posedness of the control problem. Lastly, in section \ref{sct: aplc} we present two applications and specialize our results to the unconditional case.

\section{Preliminaries and main results}
\subsection{Notation}
Let us first gather some frequently used notation. For any Polish space $E$, let $\mathcal{P}(E)$ denote the space of probability measures in $E$. For any $\mu^1, \mu^2 \in \mathcal{P}(E)$, we denote with $\mathcal{H}$ the relative entropy, i.e.\ $\mathcal{H}(\mu^1 \parallel \mu^2) = \int_E \log(\frac{d\mu^1}{d\mu^2})d\mu^1$ if $\mu^1 \ll \mu^2$ and $\mathcal{H}(\mu^1 \parallel \mu^2) = \infty$ otherwise. When working on a metric space, we denote with $B_x(r)$ the open ball with radius $r$ around $x$. For matrices, unless otherwise specified, we always use the Frobenius norm.\par
We write $\mathcal{C}_E$ for the space of continuous $E$ valued functions on $[0, T]$ being equipped with the supremum metric for a fixed time horizon $T > 0$. We also write $\mathcal{C}$ for $\mathcal{C}_{\mathbb{R}^d}$.
Let $(\Omega, \mathcal{F}, \mathbb{P})$ be the completed canonical space of a $d$ dimensional Wiener process $W$ on $[0, T]$, i.e.\ $\Omega = \mathcal{C}$ and $\P$ is Wiener measure.
%For a filtered probability space $(\Omega, \mathcal{F}, \mathbb{P})$ with a finite time horizon $T \geq 0$, 
We equip the probability space with the $\P$-completed natural filtration $\mathbb{F}$ of $W$.
Unless otherwise specified, $[0, T] \times \Omega$ is always equipped with the progressive $\sigma$ algebra. \par

%A continuous martingale $M$ is said to be a $\mathbb{P}$-BMO martingale if there is a constant $k$ such that for any stopping time $\tau  \geq 0$, we have $\mathbb{E}[( M_T - M_\tau)^2\vert \mathcal{F}_\tau] \leq k^2$ $\P$-a.s. Its BMO norm $\Vert \cdot \Vert_{BMO}$ is defined as the smallest possible such constant $k$.  
A process $V$ with finite variation is said to be $\mathbb{P}$-BMO if there is a constant $k$ such that for any stopping time $\tau  \geq 0$, we have $\mathbb{E}[\int_\tau^T \vert dV_s \vert \vert \mathcal{F}_\tau] \leq k$ $\P$-a.s. and define its BMO norm $\Vert \cdot \Vert_{BMO}$ as the smallest possible such constant $k$.
A continuous martignale $M$ is a $P$-BMO martingale\footnote{Recall that by \cite[(76.4)]{DMB}, in the definition of the BMO norm for continuous martingales, it is enough to consider only deterministic times instead of stopping times.} if $\langle M\rangle $ is a $\P$-BMO finite variation process. 
In particular we have $\Vert M \Vert_{BMO}^2 = \Vert \langle M \rangle \Vert_{BMO}$. If the BMO norm is taken with respect to a different probability $\mathbb{Q}$, we write it as $\Vert \cdot \Vert_{\mathbb{Q}-BMO}$.\par
\subsection{Problem formulation}
Let $A \subset\mathbb{R}^k$ be a closed convex action space. Just for notational simplicity, we assume $0 \in A$. Let $D\subset \mathbb{R}^d$ be a non-empty open domain, $\xi$ a $D$-valued initial condition admitting finite polynomial moments of all degrees and $$b: [0, T] \times \mathcal{C} \times A \times \mathcal{P}(\mathcal{C}_D) \times (0,1] \rightarrow \mathbb{R}^d,\quad \sigma: [0, T] \times \mathcal{C} \rightarrow GL(\mathbb{R}^d)$$ a measurable drift function and an invertible measurable volatility function respectively.
%We will put $\beta = \sigma^{-1}b$. 
We assume that $b$ and $\sigma$ are progressively measurable in the sense that $b(t,\omega,a,\mu) = b(t, \omega_{\cdot\wedge t}, a, \mu_{|[0,t]})$, with $\mu_{\vert [0, t]} := \mu \circ (\omega \mapsto \omega_{\vert [0, t]})^{-1}$ and $\sigma(t,\omega) = \sigma(t,\omega_{\cdot\wedge t})$.
% if we fix all arguments besides the first two, the resulting process is progressively measurable. Further, we assume that $b(t, \cdot)$ only depends on $\mu \in \mathcal{P}(\mathcal{C}_D)$ through $\mu_{\vert [0, t]} := \mu \circ (\omega \mapsto \omega_{\vert [0, t]})^{-1}$. Equivalently, we can also say that the dependence occurs in the law of the process stopped at $t$ which we write as $\mu_{\cdot \wedge t}:= \mu \circ(\omega \mapsto \omega_{\cdot\wedge t})^{-1} \in \mathcal{P}(\mathcal{C}_D)$. Note that there is no mean field interaction or control in the volatility. We write $\mathbb{F}=(\mathcal{F}_t)_{t\in [0, T]}$ for the completed filtration generated by $W$ and $\xi$.\par
Under our initial measure $\mathbb{P}$, we consider a state process and its exit time from the domain $D$:
\begin{equation*}
    X_t = \xi + \int_0^t \sigma(s, X_{\cdot \wedge s}) dW_s\quad \text{and}\quad\tau := \inf \lbrace t \geq 0\vert X_t \notin D \rbrace
\end{equation*}
where we assume that $X$ is the unique strong solution to the above SDE. For notational simplicity, we will actually assume that $b$, $\sigma$ and $X$ are actually defined on a time horizon $T' > T$ so that when $X$ stays within $D$ in $[0, T]$, we have $\tau > T$. Still, the dynamics of $X$ after $T$ will not be relevant for our problem.
\begin{definition}
\label{def:admissible}
    The set of admissible controls $\mathbb{A}$ is the set of $A$-valued progressively measurable processes $\alpha$ for which there exists a probability measure $\mathbb{P}^\alpha$ that is equivalent to $\mathbb{P}$ satisfying
     \begin{equation}\label{eq: p def}
         \frac{d\mathbb{P}^\alpha}{d\mathbb{P}} = \mathcal{E}\bigg(\int_0^{\cdot\wedge\tau} \beta\big(s, X_{\cdot \wedge s}, \alpha_s, \mathcal{L}_{\mathbb{P}^\alpha}(X_{\cdot \wedge s}\vert s < \tau), \mathbb{P}^\alpha[s < \tau]\big)dW_s\bigg)_T\quad \text{with}\quad  \beta := \sigma^{-1}b
 \end{equation}
 and such that $\mathbb{E}^\alpha[\int_0^{T\wedge \tau} \Vert \alpha_s \Vert^2 ds] <\infty$ where $\mathbb{E}^\alpha$ denotes the expectation taken with respect to $\mathbb{P}^\alpha$ and $\mathcal{E}(M):=\exp(M- \frac12\langle M\rangle)$ is the stochastic exponential of a martingale $M$.
 \end{definition}

  Since the system will no longer depend on the control once $X$ leaves the domain, we will usually assume $\alpha_t \mathbf{1}_{\{t\geq \tau\}} = 0$. 
In Proposition \ref{prop: Pa exst} below, we will provide more details on the well-posedness of such $\mathbb{P}^\alpha$. We will also write $\mathbb{A}_{BMO}$ for all progressively measurable $A$ valued $\alpha$ for which $\Vert \int_0^{\cdot \wedge \tau} \Vert \alpha_s \Vert^2 ds \Vert_{BMO} < \infty$. With a slight abuse of notation, for $\alpha \in \mathbb{A}_{BMO}$, we also write $\Vert \alpha \Vert_{BMO}^2=\Vert\int_0^\cdot \Vert\alpha_s\Vert^2 ds\Vert_{BMO}$. By Proposition \ref{prop: Pa exst} below, if $A$ is bounded then $\mathbb{A} = \mathbb{A}_{BMO}$.\par

Defining $W^\alpha_t:=W_t - \int_0^t \beta(s, X_{\cdot \wedge s}, \alpha_s, \mathcal{L}_{\mathbb{P}^\alpha}(X_{\cdot \wedge s}\vert s < \tau), \mathbb{P}^\alpha[s<\tau]) ds$, by Girsanov's theorem, we can see that $X$ then becomes a weak solution of the controlled McKean-Vlasov SDE \eqref{eq: state}.\par
We will assume that for any $t \in [0, T]$, we have $\mathbb{P}[t < \tau] > 0$. For instance, this is guaranteed if $\sigma$ is bounded on $D$. As $D$ is open, using the Dambis-Dubins-Schwarz theorem for each $x$ in the support of $\xi$, we can lower bound $\tau$ by the exit time of a small ball centered at $x$, providing us with such an estimate.\par

Now, given a progressively measurable running cost $f:  [0, T] \times \mathcal{C} \times A \times \mathcal{P}(\mathcal{C}_D) \times (0,1] \rightarrow \mathbb{R}$ and a terminal cost function $g:\mathcal{C} \times \mathcal{P}(\mathcal{C}_D) \times (0, 1] \rightarrow \mathbb{R}$, we are interested in minimizing the cost functional J as defined in \eqref{eq: def J}. For $f$ at time $t$, we again assume that the mean field interaction only occurs through the conditioned law until time $s$. %Here and in the following, $\mathbb{E}^\alpha = \mathbb{E}^{\mathbb{P}^\alpha}$ denotes the expectation with respect to $\mathbb{P}^\alpha$.\par

\subsection{The Pontryagin maximum principle in the weak formulation}
We begin by discussing \eqref{eq: p def}, where $\mathbb{P}^\alpha$ is given as a fixed point. 
To show existence of such $\mathbb{P}^\alpha$ we need a regularity condition on $\beta$ in terms of the Le Cam distance defined as
$$d_{LC}(\mu, \mu')^2 := \frac{1}{4} \int \Big(\frac{d\mu}{d(\frac{\mu + \mu'}{2})} - \frac{d\mu'}{d(\frac{\mu + \mu'}{2})} \Big)^2 d(\frac{\mu + \mu'}{2})$$
 for any two probability measures $\mu, \mu'$ on a common probability space.
Up to rescaling, this distance first appeared in \cite[Chapter 4]{LeCam} where it is shown that $d_{LC}$ is a metric and that it is equivalent to the Hellinger metric. Furthermore, we have $d_{TV}^2 \leq d_{LC}^2 \leq d_{TV}$ where $d_{TV}$ denotes the total variation metric. Formulating our Lipschitz condition in terms of the Le Cam metric thus allows for a more general framework compared to if we have used the total variation metric instead. Additionally, the topology induced by $d_{LC}$ is the same to the one induced by the total variation metric. For more information, we also refer to \cite[Chapter 7]{Polyanskiy}.
\begin{assumption}\label{asmp: beta}
    \begin{itemize}
        \item[(i)] For a fixed $\mu_0$, the process $\beta(t, X_{\cdot \wedge t}, 0, \mu_0, 1)$ is $dt \times \mathbb{P}$ a.e.\ bounded.
        \item[(ii)] There is $L > 0$ such that for any $t, x, a, \mu, \mu', p, p' \in [0, T] \times \mathcal{C} \times A^2 \times \mathcal{P}(\mathcal{C}_D)^2 \times ( 0, 1]^2$, we have
        $$\Vert \beta(t, x, a, \mu_{\cdot \wedge t}, p) - \beta(t, x, a', \mu'_{\cdot \wedge t}, p')\Vert \leq L (\Vert a - a' \Vert + d_{LC}(\mu_{\cdot \wedge t}, \mu'_{\cdot \wedge t}) + \vert p - p' \vert).$$
    \end{itemize}
\end{assumption}
\begin{proposition}\label{prop: Pa exst}
    Under Assumption \ref{asmp: beta}, for any $\alpha \in \mathbb{A}_{BMO}$, there exists a unique $\mathbb{P}^\alpha$ equivalent to $\mathbb{P}$ satisfying \eqref{eq: p def}. In particular, $\mathbb{A}_{BMO} \subset \mathbb{A}$. Moreover, for any $\alpha \in \mathbb{A}$, there is at most one $\mathbb{P}^\alpha \in \mathcal{P}(\Omega)$ satisfying \eqref{eq: p def} and there exists a sequence $(\alpha^n)_{n \geq 1}\subset \mathbb{A}_{BMO}$ such that $\mathcal{H}(\mathbb{P}^\alpha \parallel \mathbb{P}^{\alpha^n}) \rightarrow 0$.
\end{proposition}

\subsubsection{The necessary condition}
%In the conditioned framework, we are not able to differentiate our state process. Instead, the derivation of the necessary condition will rely on differentiating the change of measure. To do so, we assume the following.
The derivation of the necessary condition of optimality will require the following regularity condition:
\begin{assumption}\label{asmp: ncs}
    \begin{itemize}
        \item[(i)] For any $t$, $x$ fixed, $\beta$ is jointly differentiable in $a, \mu, p$ in the sense that there exist progressively measurable $\beta_a : [0, T] \times \mathcal{C} \times A \times \mathcal{P}(\mathcal{C}_D) \times (0, 1] \rightarrow\mathbb{R}^{d \times k}$, $\frac{\delta \beta}{\delta m} : [0, T] \times \mathcal{C} \times A \times \mathcal{P}(\mathcal{C}_D) \times (0, 1] \times \mathcal{C}_D \rightarrow \mathbb{R}^{d}$, and $\beta_p : [0, T] \times \mathcal{C} \times A \times \mathcal{P}(\mathcal{C}_D) \times (0, 1] \rightarrow\mathbb{R}^{d}$ for which the dependence in $\mu$ only happens through $\mu_{\vert [0,t]}$ and the dependence in the additional variable $\tilde{x}$ in $\frac{\delta \beta}{\delta m}$ only through $\tilde{x}_{\vert [0,t]}$ so that for any $a, a', \mu, \mu', p, p'$, we have
        \begin{align*}
            &\beta(t, x, a', \mu', p') - \beta(t, x, a, \mu, p) \\
            &= \int_0^1 \beta_a\big(t, x, \lambda a' + (1-\lambda) a, \lambda \mu' + (1 -\lambda)\mu, \lambda p' + (1-\lambda) p\big)(a' - a)d\lambda\\
            &\quad+ \int_0^1 \int_{\mathcal{C}_D} \frac{\delta \beta}{\delta m} \big(t, x, \lambda a' + (1-\lambda) a, \lambda \mu' + (1 -\lambda)\mu, \lambda p' + (1-\lambda) p, \tilde{x}\big)(\mu' - \mu)(d\tilde{x})d\lambda\\
            &\quad + \int_0^1 \beta_p\big(t, x, \lambda a' + (1-\lambda) a, \lambda \mu' + (1 -\lambda)\mu, \lambda p' + (1-\lambda) p\big) (p' - p)d\lambda.
        \end{align*}
         These derivatives are assumed to be continuous in $a$, $\mu$ and $p$ where $\mathcal{P}(\mathcal{C}_D)$ is equipped with the topology induced by the total variation distance or equivalently the Le Cam distance.
        \item[(ii)] Similarly, we also assume $f$ to be jointly differentiable in $a, \mu, p$ and $g$ in $\mu, p$.
        \item[(iii)] For any $q > 1$ and some fixed $\mu^0 \in \mathcal{P}(\mathcal{C_D})$, the random variables $\int_0^{T \wedge \tau} \vert f(s, X_{\cdot \wedge s}, 0, \mu^0, 1)\vert ds$ and $\mathbf{1}_{\lbrace T < \tau \rbrace}\vert g(X, \mu^0, 1)\vert$ admit finite $q$-th moment with respect to $\mathbb{P}$.
        \item[(iv)] For some continuous non increasing $M: (0, 1] \rightarrow (0, \infty)$, we assume for any $t, x, a, \mu, p$ that we have 
        $$\vert f_a(t, x, a, \mu, p)\vert \leq M(p)(1 + \Vert a \Vert),  \quad \bigg(\int_{\mathcal{C}_D} \vert\frac{\delta f}{\delta m}(t, x, a, \mu, p, \tilde{x})\vert^2\mu(d\tilde{x})\bigg)^\frac{1}{2}\leq \frac{M(p)}{4}(1 + \Vert a \Vert^2)
        $$
        and
        $$\vert f_p(t, x, a, \mu, p)\vert \leq M(p)(1 + \Vert a \Vert^2).$$
        \item[(v)] Similarly, $(\int_{\mathcal{C}_D} \vert \frac{\delta g}{\delta m}(x, \mu, p, \tilde{x}) \vert^2 \mu(d\tilde{x}))^\frac{1}{2} \leq \frac{M(p)}{4}$ and $\vert g_{p}(x, \mu, p)\vert \leq M(p)$.
        % \item For some fixed $\mu^0 \in \mathcal{P}(\mathcal{C_D})$ the random variables
        % $\vert \mathbf{1}_{\lbrace t < \tau \rbrace} f(t, X_{\cdot \wedge t}, 0, \mu^0, 1)\vert$ and $\vert\mathbf{1}_{\lbrace T < \tau \rbrace}g(X, \mu^0, 1)\vert$ are uniformly bounded.
        % \item For some continuous non increasing $M: (0, 1] \rightarrow (0, \infty)$, we assume for any $t, x, a, \mu, p$ that we have $\vert f_a(t, x, a, \mu, p)\vert \leq M(p)(1 + \Vert a \Vert)$,  $(\int_{\mathcal{C}_D} \vert\frac{\delta f}{\delta m}(t, x, a, \mu, p, \tilde
        % {x})\vert^2\mu(d\tilde{x}))^\frac{1}{2}\leq \frac{M(p)}{4}(1 + \Vert a \Vert^2)$, and $\vert f_p(t, x, a, \mu, p)\vert \leq M(p)(1 + \Vert a \Vert^2)$.
        % \item Similarly, $(\int_{\mathcal{C}_D} \vert \frac{\delta g}{\delta m}(x, \mu, p, \tilde{x}) \vert^2 \mu(d\tilde{x}))^\frac{1}{2} \leq \frac{M(p)}{4}$ and $\vert g_{p}(x, \mu, p)\vert \leq M(p)$.
        % \item Further, for any $t$, $x$, $a$, $\mu$, $p$, assume either of the two:\begin{enumerate}
        %     \item For any $\tilde{x}$, we have $\Vert\frac{\delta \beta}{\delta m}(t, x, a, \mu, p, \tilde{x})\Vert \leq L$.
        %     \item $A$ is bounded, and $\int_{\mathcal{C}_D} \Vert \frac{\delta \beta}{\delta m}(t, x, a, \mu, p, \tilde{x})\Vert^2 \mu(d\tilde{x})\leq \frac{L^2}{16}$.
        % \end{enumerate}
    \end{itemize}
\end{assumption}
\begin{remark}\label{rmk: lin drv}
    One can see that the notion of derivative with respect to the measure argument we use is the linear functional derivative  discussed in \cite[Definition 5.43]{CarmonaDelarue} rather than the Lions derivative defined in \cite[Definition 5.22]{CarmonaDelarue}. Since the linear derivative is only defined up to an additive function independent of the additional variable $\tilde{x}$, we always assume that $\int_{\mathcal{C}_D} \frac{\delta f}{\delta m} (t, x, a, \mu, p, \tilde{x})\mu(d\tilde{x}) = 0$ for any $\mu$.
\end{remark}
The requirement that $f$ and $g$ must admit all moments seems strong, but is for instance satisfied if they admit polynomial growth in $X$ and $\sigma$ is bounded.
Moreover, these assumptions guarantee that $J(\alpha)<\infty$ for all admissible $\alpha \in \mathbb{A}$. 
Let us also note that bounding the second moment of the linear derivative with respect to the measure corresponds to bounding the Lipschitz constant with respect to $d_{LC}$; we further comment on this in Lemma \ref{lem: LCLips}.\par

Let us focus for now only on  controls $\alpha$ in $\mathbb{A}_{BMO}$ as the necessary condition requires additional integrability properties. 
By the approximation proven in Lemma \ref{lem: BMO aprx}, under some additional conditions, this does not change the optimal value of our control problem. 
To simplify the notation, in what follows, we write
\begin{equation}\label{eq: Tht var}
    \Theta^\alpha_t = (t, X_{\cdot \wedge t}, \alpha_t, \mathcal{L}_{\mathbb{P}^\alpha}(X_{\cdot \wedge t}\vert t < \tau), \mathbb{P}^\alpha[t < \tau])
\end{equation}
for any $\alpha \in \mathbb{A}$, and random variables with a $\tilde{  }$ denote copies on an independent duplicate space $(\tilde{\Omega}, \tilde{\mathcal{F}}, \tilde{\mathbb{P}})$ with the expectation $\tilde{\mathbb{E}}^\alpha$ taken with respect to the measure $\tilde{\mathbb{P}}^\alpha$ on $\tilde{\Omega}$ defined by $\frac{d\tilde{\mathbb{P}}^\alpha}{d\tilde{\mathbb{P}}} = \mathcal{E}(\int_0^{\cdot \wedge \tau}\beta(\tilde{\Theta}^{\alpha}_s)d\tilde{W}_s)$.\par
Let us now recall our adjoint backward equation:
%to characterize optimality. In the weak formulation, for a given $\alpha$, it takes on the form 
\begin{eqnarray}\label{eq: adj eq}
    &&Y^\alpha_{t\wedge \tau} = \mathbf{1}_{\lbrace T < \tau \rbrace} \Big( g(\Theta^\alpha_T) + \tilde{\mathbb{E}}^\alpha\Big[\frac{\delta g}{\delta m}(\tilde{\Theta}^\alpha_T, X)\vert T < \tilde{\tau}\Big] + \mathbb{E}^\alpha[\mathbf{1}_{\lbrace T < \tau \rbrace}g_p(\Theta^\alpha_T)] \Big)\\
    &+& \int_{t\wedge \tau}^{T \wedge \tau} \Big(f(\Theta^\alpha_s) + \tilde{\mathbb{E}}^\alpha\Big[\frac{\delta h}{\delta m}(\tilde{\Theta}^\alpha_s, \tilde{Z}^\alpha_s, X_{\cdot \wedge s})\vert s < \tilde{\tau} \Big] + \mathbb{E}^\alpha[\mathbf{1}_{\lbrace s<\tau \rbrace} h_p(\Theta^\alpha_s, Z^\alpha_s)]\Big)ds -\int_{t\wedge\tau}^{T \wedge \tau} Z^\alpha_s dW^\alpha_s \nonumber
\end{eqnarray}
where   
$$h(t, x, a, \mu, p, z):= f(t, x, a, \mu, p) + \beta(t, x, a, \mu, p)^\top z$$
denotes the Hamiltonian. Note that $\mathbb{E}^\alpha[\vert \frac{\delta g}{\delta m}(\tilde{\Theta}^\alpha_T, X) \vert^2 \vert T< \tau]\leq \frac{M}{4}$ implies that $\tilde{\mathbb{E}}^\alpha[\frac{\delta g}{\delta m}(\tilde{\Theta}^\alpha_T, X)\vert T < \tilde{\tau}]$ is $\mathbb{P}^\alpha$ a.s.\ finite and square integrable. A similar reasoning applies to $f$. Under our assumptions, for any given $\alpha\in\mathbb{A}_{BMO}$, as shown in Proposition \ref{prop: wp adj}, this BSDE admits a unique $\mathbb{P}^\alpha$ square integrable solution.\par

%In standard non McKean-Vlasov control, the presence of a derivative in $x$ in the adjoint equation differentiates the maximum principle approach from the other predominant Bellmann optimality approach using dynamic programming to characterize the value function directly. Indeed, when there is no mean field interaction in our problem, we recover with our adjoint equation the same BSDE that is usually derived using the dynamic programming principle, see e.g. \cite[Theorem 21.3.6.]{CE}, although we still find the BSDE through the Pontryagin principle approach, not the dynamic programming approach. Still, the strong similarities of the form of the adjoint equation may lead us to interpret our approach to lie in the middle of these two approaches.\par
%As an immediate consequence of our analysis in Proposition \ref{prop: var prc} and Theorem \ref{thm: gat drv}, we find the following necessary condition for the maximum principle.
With the above notation and definitions out of the way, we state the necessary condition for optimality:
\begin{theorem}\label{thm: ncs}
    Let Assumptions \ref{asmp: beta} and \ref{asmp: ncs} hold. If $\alpha \in \mathbb{A}_{BMO}$ satisfies $J(\alpha) = \inf_{\alpha' \in \mathbb{A}_{BMO}}J(\alpha')$, then, for every $\alpha' \in \mathbb{A}_{BMO}$, we have $\mathbb{E}^\alpha[\int_0^{T \wedge \tau}h_a(\Theta^\alpha_s, Z^\alpha_s)^\top(\alpha'_s - \alpha_s) ds] \geq 0$ where $\Theta^\alpha$ is given in \eqref{eq: Tht var} and $(Y^\alpha, Z^\alpha)$ is the solution to the adjoint equation \eqref{eq: adj eq}.\par
    In particular, a.e.\ on $\lbrace s < \tau \rbrace$, we have for any $e \in A$ that $h_a(\Theta^\alpha_s, Z^\alpha_s)^\top (e-\alpha_s) \geq 0$. If $h$ is convex in $a$, $\alpha_s$ must a.e.\ on $\lbrace s < \tau \rbrace$ belong to the set of minimizers of $a \mapsto h(s, X_{\cdot \wedge s}, a, \mathcal{L}_{\mathbb{P}^\alpha}(X_{\cdot \wedge s}\vert s < \tau), \mathbb{P}^\alpha[s<\tau], Z^\alpha_s)$.
\end{theorem}

Observe that the drift of our adjoint equation \eqref{eq: adj eq} is given in terms of the \emph{linear derivative} of the Hamiltonian rather than the \emph{Lions derivative}
as in the strong formulation, see e.g.\ \cite[Definition 6.5]{CarmonaDelarue}. 
In fact, recalling that the Lions derivative can be understood as the gradient of the linear derivative as shown in \cite[Proposition 5.48]{CarmonaDelarue}, it follows that the generator of the adjoint equation in the strong formulation is the gradient of the one in equation \eqref{eq: adj eq}.

% In standard, non-McKean–Vlasov control problems, the appearance of a derivative with respect to $x$ in the adjoint equation distinguishes the maximum principle approach from the Bellman optimality approach, which relies on dynamic programming to characterize the value function directly. Indeed, in the absence of mean-field interactions, our adjoint equation reduces to the same backward stochastic differential equation (BSDE) typically derived via the dynamic programming principle (see, e.g., \cite[Theorem 21.3.6.]{CE}). However, in our case, this BSDE arises through the Pontryagin maximum principle rather than from dynamic programming.
% Nevertheless, the structural similarity between the resulting adjoint equations suggests that our approach may be viewed as lying somewhere between the Pontryagin and Bellman frameworks.

%we establish existence of optimal controls, this extra optimality condition would also allow us to interpret the adjoint equation \eqref{eq: adj eq} as a generalized McKean-Vlasov BSDE.\par
%By a closed inspection of our proofs, we can see that we can weaken our assumptions when we have a priori estimates.}
\begin{remark}\label{rmk: asm meas}
    \begin{itemize}
        \item[(i)] Our coefficients do not actually need be defined on all of $\mathcal{P}(\Omega)$. For instance, when dealing with $\alpha \in \mathbb{A}$, we need the coefficients $g$ and $(b,f)$ to be defined only for measures that are equivalent (or absolutely continuous later in Theorem \ref{thm: opt non eq}) to $\mathcal{L}_\mathbb{P}(X \vert T < \tau)$ and $\mathcal{L}_\mathbb{P}(X_{\cdot \wedge t}\vert t<\tau)$ respectively. Note that these measures form a convex subset of $\mathcal{P}(\mathcal{C}_D)$. This applies to both the previous and upcoming assumptions.    
   \item[(ii)]  When $A$ is bounded (and thus $\beta$ is by Assumption \ref{asmp: beta} bounded as well) and in some other special cases, we are able to uniformly bound positive and negative moments of $\frac{d\mathbb{P}^\alpha}{d\mathbb{P}}$ using e.g.\ \cite[Theorem 15.4.6.]{CE}. Then, one can e.g.\ define $g$ only on the set of measures $\mu$ satisfying 
    $$\int \Big(\frac{d\mu}{d\mathcal{L}_\mathbb{P}(X\vert T<\tau)})^2 + (\frac{d\mu}{d\mathcal{L}_\mathbb{P}(X\vert T<\tau)}\Big)^{-1}d\mathcal{L}_\mathbb{P}(X\vert T<\tau) < \mathcal{M}$$
    for a constant $\mathcal{M}$ large enough to include all $\mathcal{L}_{\mathbb{P}^\alpha}(X\vert T<\tau)$. These measures again form a convex subset of $\mathcal{P}(\Omega)$.\par
    \item[(iii)] Bounds of the inverse moments can also be used to find a priori bounds for $p$ in the same spirit as what we do in the proof of Proposition \ref{prop: Pa exst}. That is, if it is possible to find $\underline{p}_t$ such that $\mathbb{P}^\alpha[t < \tau] \geq \underline{p}_t$, our coefficients do not need to be defined for all of $(0, 1]$ but it suffices for them to be defined on the compact intervals $[\underline{p}_t, 1]$.
\end{itemize}
\end{remark}
\subsubsection{The sufficient condition}
Let us now state the sufficient condition of optimality. 
Just as in the classical Pontryagin principle, we will need to leverage some sort of convexity of the cost functional.
To this end, we will require our cost functions to be jointly convex in $\mu$ and $p$ in a way that we now describe. 
\begin{definition}
    We say a function $\phi: \mathcal{C}_D \times \mathcal{P}(\mathcal{C}_D) \times (0,1] \rightarrow \mathbb{R}$ is $p$-convex when $p\int_{\mathcal{C}_D} \phi(x, \mu, p)\mu(dx)$ is linearly convex if considered as a function of the sub probability measure $p\mu$. That is, $\phi$ is $p$-convex if for any $\mu$, $\mu'$, $p$, $p'$, and $\lambda \in (0, 1)$ we have
\begin{equation}\label{eq: def pcnv}
    p^\lambda \int \phi(x, \overline{\mu}^\lambda, p^\lambda)\overline{\mu}^\lambda(dx) \leq \lambda p' \int \phi(x, \mu', p')\mu'(dx) + (1-\lambda) p \int \phi(x, \mu, p)\mu(dx)
\end{equation}
where $p^\lambda = \lambda p' + (1-\lambda)p$ and $\overline{\mu}^\lambda = \frac{\lambda p'}{p^\lambda} \mu' + \frac{(1-\lambda)p}{p^\lambda} \mu$.
\end{definition}
We will extensively elaborate on this notion of convexity in Subsection \ref{sub:convex_cost_functions} and provide several example of functions satisfying it.
\begin{assumption}\label{asmp: cnv}
    \begin{itemize}
        \item[(i)]  $b$ is independent of $\mu$ and $p$ and linear in $a$, i.e.\ there are $\mathbb{R}^{d\times k}$ valued $b^1$ and $\mathbb{R}^d$ valued $b^2$ such that $b(t, x, a) = b^1(t, x) a + b^2(t, x)$. We assume $b^1$ has $dt \times \mathbb{P}$ a.e.\ linearly independent columns. Accordingly, we also write $\beta(t, x, a) = \beta^1(t, x) a + \beta^2(t, x)$.
        \item[(ii)] $f$ is separable into the form $f(t, x, a, \mu, p) = f^1(t, x, a) + f^2(t, x, \mu, p)$.
        \item[(iii)]  $g$ and $f^2(t, \cdot)$ are $p$-convex for all $t\ge0$.
        \item[(iv)] $f^1$ is $m$-strongly convex in $a$ for some $m\geq0$, i.e.\ $f^1(t, x, a) - \frac{m}{2}\Vert a \Vert^2$ is convex\footnote{Unless otherwise specified, we do allow $m = 0$.} in $a$. 
        \item[(v)] One of the following two conditions holds:
            \begin{enumerate}
            \item[(1)] $A$ is bounded, 
            \item[(2)]We have: \begin{itemize}
            \item $\vert f^1(t, x, 0)\vert$, $\vert f^2(t, x, \mu^0, 1)\vert$, and $\vert g(x, \mu^0, 1) \vert$ are uniformly bounded.
            \item $f^2$ and $g$ are uniformly bounded from below.
            \item For some $M^1>0$ and some continuous non increasing $M^2:(0, 1] \rightarrow (0, \infty)$, we assume $\vert f^1_a(t, x, a) \vert \leq M^1(1+ \Vert a\Vert)$, $\vert \frac{\delta f^2}{\delta m}(t, x, \mu, p, \tilde{x})\vert \leq M^2(p)$, $f^2_p(t, x, \mu, p) \leq M^2(p)$, $\frac{\delta g}{\delta m}(x, \mu, p, \tilde{x}) \leq M^2(p)$, and $\vert g_p(x, \mu, p)\vert \leq M^2(p)$.
        \end{itemize}
    \end{enumerate}
    \end{itemize}
\end{assumption}
Note that in the last point, for the derivatives in $a$ and $p$, the growth condition matches the one from Assumption \ref{asmp: ncs} and only takes on an easier form due to the separability assumption. For the derivative in the measure, we strengthen the assumption to an almost sure bound instead of second moment bound in $\mu$.
Furthermore, the assumption that $b^1$ admits a.e.\ linearly independent columns is not essential. All the following results still remain true, but the proofs would require an additional measurable selection argument.\par

In the statement below, we write 
$$h^1(t, x, a, z) = f^1(t, x, a) + z^\top \beta(t, x, a),\quad \text{and}\quad \theta^\alpha_t = (t, X_{\cdot \wedge \tau}, \mathcal{L}_{\mathbb{P}^\alpha}(X_{\cdot \wedge t}\vert t<\tau), \mathbb{P}^\alpha[t<\tau]).$$ 
\begin{theorem}\label{thm: sfc}
    Let Assumptions \ref{asmp: beta}, \ref{asmp: ncs} and \ref{asmp: cnv} hold. 
    If for some $\alpha \in \mathbb{A}_{BMO}$, we have a.e.\ on $\lbrace s < \tau \rbrace$ that $\alpha_s \in  \arg\min_{a\in A} h^1(s, X_{\cdot \wedge s}, \alpha_s, Z^\alpha_s)$, where $(Y^\alpha, Z^\alpha)$ is the solution of \eqref{eq: adj eq}, then for any other $\alpha' \in \mathbb{A}$,
    $$J(\alpha') - J(\alpha) \geq \frac{m}{2}\mathbb{E}^{\alpha'}\bigg[\int_0^{T\wedge\tau}\Vert \alpha_s - \alpha'_s\Vert^2ds\bigg] \geq \frac{m}{L^2}\mathcal{H}(\mathbb{P}^{\alpha'} \parallel \mathbb{P}^\alpha).$$
    In particular, $\alpha$ is optimal over $\mathbb{A}$.
\end{theorem}
The above sufficient condition is reminiscent of Pythagorean Theorem for entropy \cite[Theorem 15.10]{Polyanskiy}. 
In fact, Theorem \ref{thm: sfc} generalizes this result, which is recovered by taking $\beta(t, x, a) = a$, $f^1(t, x, a) = \frac{1}{2}\Vert a \Vert^2$, $f^2 = 0$ and $g = 0$. The estimate in Theorem \ref{thm: sfc} also provides desirable absolute continuity properties for $\mathbb{P}^\alpha$ when $\alpha$ is optimal.\par

The above sufficient condition can be better understood through the lens of convexity  of $J$ (see Proposition \ref{prop: J cnv}). It fact, it follows from Assumption \ref{asmp: cnv} that $J$ is a convex functional; not in $\alpha$ but rather in the associated measure $\mathbb{P}^\alpha$. This observation along with Gateaux derivative in Theorem \ref{thm: gat drv} allow the interpretation of Theorem \ref{thm: sfc} as a first order characterization of minima for convex functions.\par
These convexity properties of $J$ heavily rely on the structural conditions in Assumption \ref{asmp: cnv}.
For instance it is essential that $b$ is linear in $a$ and independent on $(\mu,p)$ to guarantee convexity.
% Further, it turns out that once $b$ depends on $\mu$ or $p$, convexity breaks down, even if $b$ admits some kind of separable structure.
Separability of $f$ is also essential, already just to properly define $p$-convexity for $f$.
Such structural assumptions are often considered in the mean field game and mean field control literature.
\begin{remark}
    Although the process $Y^\alpha$ cannot be interpreted immediately as a remaining utility process, in the unconditioned case $D = \mathbb{R}^d$, $Y^\alpha$ is still connected to the value function. In fact, using the convention we chose in Remark \ref{rmk: lin drv}, we can then see that $\mathbb{E}^\alpha[Y^\alpha_0] = J(\alpha)$.
\end{remark}

\subsection{Wellposedness results}
In the setting of the previous section, our necessary and sufficient conditions relate optimal controls to the adjoint BSDE \eqref{eq: adj eq}. 
In Proposition \ref{prop: wp adj}, we show how to solve for $(Y^\alpha, Z^\alpha)$ for any given $\alpha \in \mathbb{A}_{BMO}$. When $b$ is assumed to be independent of $\mu$ and $p$, we can actually find a solution $(Y^\alpha, Z^\alpha)$ given any $\alpha \in \mathbb{A}$, see Lemma \ref{lem: opt ncs}. 
Note that in practice, one usually does not already have a candidate for the optimal control.
By Theorem \ref{thm: ncs} and \ref{thm: sfc}, we see that $\alpha \in \mathbb{A}_{BMO}$ is optimal if and only if $\alpha$ also minimizes the Hamiltonian along the solution of the adjoint equation. This leads us to study the adjoint BSDE with this added coupling condition:
\begin{equation}\label{eq: adj gnr}
    \begin{cases}
    Y^\alpha_{t\wedge \tau} = \mathbf{1}_{\lbrace T < \tau \rbrace}\Big(g(\theta^\alpha_T) + \tilde{\mathbb{E}}^\alpha\Big[\frac{\delta g}{\delta m}(\tilde{\theta}^\alpha_T, X)\vert T < \tilde{\tau}\Big] + \mathbb{E}^\alpha[\mathbf{1}_{\lbrace T < \tau \rbrace}g_p(\theta^\alpha_T)]\Big)+ \int_{t\wedge\tau}^{T\wedge\tau}\Big(f^1(s, X_{\cdot \wedge s}, \alpha_s) \\
    \quad + f^2(s, \theta^\alpha_s) + \tilde{\mathbb{E}}^\alpha\Big[\frac{\delta f^2}{\delta m}(s, \tilde{\theta}^\alpha_s, X_{\cdot \wedge s})\vert s<\tilde{\tau}\Big] + \mathbb{E}^\alpha\big[\mathbf{1}_{\lbrace s<\tau \rbrace}f^2_p(s, \theta^\alpha_s)\big]\Big)ds-\int_{t\wedge\tau}^{T\wedge\tau} Z^\alpha_s dW^\alpha_s\,\,\, \P^{\alpha}\text{-a.s.}\\
     \alpha_t \in \arg \min_{a \in A} h^1(t, X_{\cdot \wedge t}, a, Z^{\alpha}_t), dt \times \mathbb{P} \text{ a.e. on } \lbrace t < \tau \rbrace.
    \end{cases}
\end{equation}
In equation \eqref{eq: adj gnr}  $\alpha$ and $\mathbb{P}^\alpha$ are unknown.
That is, a solution would consists of $\alpha$ together with $(Y^\alpha, Z^\alpha)$. In particular, the underlying law $\mathbb{P}^\alpha$ and Wiener process $W^\alpha$ are a priori unknown, so that $\eqref{eq: adj gnr}$ becomes a so-called \emph{generalized McKean-Vlasov BSDE} as introduced in \cite{PossamaiTangpi}. Our previous results then show that for $\alpha \in \mathbb{A}_{BMO}$, optimal controls exactly correspond to solutions to \eqref{eq: adj gnr}. This kind of McKean-Vlasov BSDE play the same role as the standard forward-backward system encountered in the strong formulation. Unfortunately, such generalized McKean-Vlasov BSDEs are much more difficult to solve than regular BSDEs. In the following, we will characterize optimal controls for our conditioned McKean-Vlasov control problem in the convex setting, and with this also provide a well-posedness result for \eqref{eq: adj gnr}.\par
Before presenting our well-posedness results let us mention that for $m>0$ in Assumption \ref{asmp: cnv} the strong convexity also serves as a coercivity condition as it implies at least quadratic growth. Indeed, $\int_0^{T \wedge \tau} f^1(s, X_{\cdot \wedge s}, \alpha_s) ds \geq \int_0^{T \wedge \tau}f^1(s, X_{\cdot \wedge s}, 0) - 2\frac{(M^1)^2}{m} + \frac{m}{4}\Vert \alpha_s \Vert^2 ds$. Hence, when Assumption \ref{asmp: cnv} holds with $m$ positive and $A$ unbounded, $J(\alpha) < \infty$ is equivalent to $\mathbb{E}^\alpha[\int_0^{T\wedge\tau} \Vert \alpha_s \Vert^2 ds] < \infty$. In particular, $\mathbb{A}$ could have also been defined as the set of all $\alpha$ for which $\mathbb{P}^\alpha$ can be defined and $J(\alpha)<\infty$.
\begin{theorem}\label{thm: exs}
    Let Assumptions \ref{asmp: beta}, \ref{asmp: ncs}, and \ref{asmp: cnv} hold. If $f^1$ is strictly convex in $a$, then any optimal control in $\mathbb{A}$ is unique on $[0,\tau)$.\par
    If $A$ is bounded or $m>0$, then there exists an optimal control $\hat{\alpha} \in \mathbb{A}_{BMO}$, i.e.\ there is $\hat{\alpha} \in \mathbb{A}_{BMO}$ such that $J(\hat{\alpha}) = \inf_{\alpha\in \mathbb{A}}J(\alpha)$.
\end{theorem}
\begin{remark}
\label{rem:exists}
    \begin{enumerate}
        \item[(i)] It is noteworthy under the assumptions of Theorem \ref{thm: exs}, the optimal control belongs to $\mathbb{A}_{BMO}$ even when we optimize over the much larger set $\mathbb{A}$. This justifies the the fact that we focused on controls in $\mathbb{A}_{BMO}$ in our necessary and sufficient conditions.
        \item[(ii)] Theorem \ref{thm: exs} shows that for $f$ strictly convex, $\hat{\alpha}$ is the unique solution to \eqref{eq: adj gnr} within the class of solution such that $\Vert \alpha \Vert_{BMO}<\infty$. It remains open whether there can be different $\alpha$ solving this generalized McKean-Vlasov BSDE that do not satisfy $\Vert \alpha \Vert_{BMO} <\infty$. For $\hat{\alpha}$, our argument in Lemma \ref{lem: opt BMO} relies on the theory of quadratic BSDEs.
        As far as we know, for general $\alpha$, it would be unclear whether the required integrability condition for $(Y^\alpha, Z^\alpha)$ can be established.
        \item[(iii)] As described below, when $A$ is unbounded, we can recover that
        % let us also show how
         the control $\hat{\alpha}$ constructed in Theorem \ref{thm: exs} remains optimal even when considering a larger space of controls than $\mathbb{A}$.
         To this end, 
         \begin{center}
             we define $\tilde{\mathbb{A}}$ as $\mathbb{A}$ in Definition \ref{def:admissible} but with $\mathbb{P}^\alpha \ll \mathbb{P}$ instead of $\mathbb{P}^\alpha \sim \mathbb{P}$. 
         \end{center}
         That is, we consider $\tilde{\mathbb{A}}$ as the set of $A$-valued progressively measurable $\alpha$ for which we can define $\mathbb{P}^\alpha \in \mathcal{P}(\Omega)$ satisfying $\frac{d\mathbb{P}^\alpha}{d\mathbb{P}} = \mathbf{1}_{\lbrace \frac{d\mathbb{P}^\alpha}{d\mathbb{P}}>0 \rbrace} \mathcal{E}(\int_0^{\cdot \wedge \tau} \beta(s, X_{\cdot \wedge s}, \alpha_s)dW_s)_T$ and $\mathbb{E}^\alpha[\int_0^{T \wedge \tau} \Vert \alpha_s \Vert^2 ds] <\infty$. This way, $\alpha$ really needs to be defined only on the support of $\mathbb{P}^\alpha$, and we will usually assume that it is zero outside the support of $\mathbb{P}^\alpha$.
    \end{enumerate}
    \end{remark}
So far, when defining $\mathbb{A}$, we have excluded non equivalent changes of measures to avoid the case $\mathbb{P}^\alpha[t<\tau] = 0$ as $J$ is not necessarily well defined in this pathological case. To still make sense of the problem for such controls,
 $\alpha \in \tilde{\mathbb{A}}$ (see Remark \ref{rem:exists}.$(iii)$), 
we would need to assume that $f^2$ and $g$ are extendable to $p = 0$ in the following way:
We assume that there exists $F: [0, T] \times \mathcal{P}(\mathcal{C}_D) \times [0, 1] \rightarrow \mathbb{R}$, $G:\mathcal{P}(\mathcal{C}_D) \times [0, 1] \rightarrow \mathbb{R}$ such that we have
$$F(t, \mu, p) = p\int_{\mathcal{C}_D} f^2(t, x, \mu, p)\mu(dx) \quad \text{and} \quad G(\mu, p) = p \int_{\mathcal{C}_D} g(x, \mu, p)\mu(dx)$$
for each $t\in [0, T]$, $\mu\in\mathcal{P}(\mathcal{C}_D)$, and $p \in (0, 1]$. To bypass the issue of conditioning on a null set, we assume that for $p = 0$, the maps $F$ and $G$ are independent of $\mu$, i.e.\ we have for any $\mu, \mu'$ that $F(t, \mu, 0) = F(t, \mu', 0)$ and $G(\mu, 0) = G(\mu', 0)$. For such $F$ and $G$, we can now define
\begin{align*}
    \tilde{J}(\alpha) &= \mathbb{E}^\alpha\bigg[\int_0^{T \wedge \tau} f^1(s, X_{\cdot \wedge s}, \alpha_s)ds\bigg] + \int_0^T F(s, \mathcal{L}_{\mathbb{P}^\alpha}(X_{\cdot \wedge s}\vert s<\tau), \mathbb{P}^\alpha[s<\tau])ds \\
    &\qquad\qquad  + G\big(\mathcal{L}_{\mathbb{P}^\alpha}(X\vert T<\tau), \mathbb{P}^\alpha[T<\tau]\big).
\end{align*}
By construction, for any $\alpha \in \mathbb{A}$, we have $\tilde{J}(\alpha) = J(\alpha)$. When $A$ is bounded, $\tilde{\mathbb{A}} = \mathbb{A} = \mathbb{A}_{BMO}$.
\begin{theorem}\label{thm: opt non eq}
    Let Assumptions \ref{asmp: beta}, \ref{asmp: ncs} and \ref{asmp: cnv} hold with $A$ unbounded and $m>0$. Further, assume $f^2$ and $g$ are extendable as above with the maps $F$ and $G$ being continuous in $\mu$ and $p$, as well as uniformly bounded. Then, for each $\alpha \in \tilde{\mathbb{A}}$, there exist $\alpha^n\in \mathbb{A}$ such that $\alpha^n_t \rightarrow \alpha_t$ a.e. and $\lim_{n\rightarrow\infty}J(\alpha^n) = \tilde{J}(\alpha)$. Further, the previously found $\hat{\alpha}$ is the unique minimizer of $\tilde{J}$ over $\tilde{\mathbb{A}}$.
\end{theorem}

\subsection{Applications to Schr\"odinger-type problems and mean field games}

We illustrate the above results with two applications: A constrained version of Schr\"odinger's problem and mean field games.

%\subsubsection{Schr\"odinger-type problems}
\subsubsection{Schrödinger bridges with hard killing}\label{sct: schrö}
As an application of our results on conditioned McKean-Vlasov control problems
we introduce and analyze a new version of the celebrated Schrödinger problem in which we impose that particles exiting a given domain are killed.
 %with target constraints in the law below in Theorem \ref{thm: trgt}, let us analyze a variant of the Schrödinger problem. 

 The Schrödinger problem originally deals with finding the most likely path of a particle system with given initial and terminal distributions. 
 Using Sanov's theorem, this problem can be reformulated as an entropy minimization problem over all paths with prescribed initial and terminal marginals. When the reference measure is  Wiener's measure, Föllmer's drift allows to see this problem as a linear quadratic control problem with a constraint on the terminal marginal distribution of the state.
 That is, given an initial distribution $\nu =\mathbb{P}\circ \xi^{-1}$ and a terminal distribution $\hat\mu\in \cP_2(\R^d)$, Schrödinger problem is
 $$V_{\hat{\mu}} = \inf \lbrace \tilde J(\alpha)\vert \alpha \in \tilde{\mathbb{A}},\mathcal{L}_{\mathbb{P}^\alpha}(X_0) = \nu ,\, \mathcal{L}_{\mathbb{P}^\alpha}(X_T ) = \hat{\mu} \rbrace\quad \text{with}\quad \tilde J(\alpha) =  \mathbb{E}^\alpha\bigg[\int_0^{T}\frac{1}{2}\Vert \alpha_s \Vert^2ds\bigg] = \mathcal{H}(\mathbb{P}^\alpha\parallel\mathbb{P}).
 $$\par

 In this paper we extend Schrödinger problem to the case where particles are killed when they exit a given domain $D$.
 Furthermore, in addition to fixing the particles' initial and terminal marginals, we also fix a desired ratio of surviving particles.
%From this perspective, let us heuristically motivate the formulation of the problem we consider below. That is, we would like to model the most likely path for a system of particles for which the particles get killed once they leave our domain and we have information about the distribution of our particle system at initial time, and further the distribution at terminal time together with the ratio of particles that have survived in the domain. 
In this setting, 
%we thus want to minimize the entropy $\mathcal{H}(\mathbb{Q} \parallel \mathbb{P}_{\vert \mathcal{F}_{T\wedge \tau}})$ over all 
the reference measure is the restriction $\mathbb{P}_{\vert \mathcal{F}_{T\wedge \tau}}$ to $\P$ on $\mathcal{F}_{T\wedge \tau}$,
and the cost function to be minimized is
$$J(\alpha) =  \mathbb{E}^\alpha\bigg[\int_0^{T\wedge\tau}\frac{1}{2}\Vert \alpha_s \Vert^2ds\bigg] = \mathcal{H}(\mathbb{P}^\alpha_{\vert\mathcal{F}_{T\wedge \tau}}\parallel\mathbb{P}_{\vert \mathcal{F}_{T\wedge \tau}}).$$
Thus, given $\hat{p} \in (0, 1]$ and $\hat{\mu}\in \mathcal{P}(D)$, we are interested in the optimization problem
\begin{equation}
\label{eq:def.schro.hard.kill}
V_{\hat{p}, \hat{\mu}} = \inf \big\lbrace J(\alpha)\vert \alpha \in \tilde{\mathbb{A}},\mathbb{P}^\alpha[T<\tau] = \hat{p},\, \mathcal{L}_{\mathbb{P}^\alpha}(X_0) = \nu,\,  \mathcal{L}_{\mathbb{P}^\alpha}(X_T \vert T<\tau) = \hat{\mu} \big\rbrace.
\end{equation}
 We will show that this problem can be approximated by a sequence of McKean-Vlasov control problems similar to those studied so far,
like the ones we have studied before that replace the target constraint by a penalization term. 
 For each $l \geq 1$, we consider
 \begin{equation}
 \label{eq:def.schro.penal}
    V^l_{\hat{p}, \hat{\mu}} = \inf_{\alpha \in \tilde{\mathbb{A}}} J^l(\alpha)\quad \text{where}\quad J^l(\alpha) := J(\alpha) + \frac{l}{2}\Big\Vert \hat{p}\hat{\mu} - \mathbb{P}^\alpha[T<\tau]\mathcal{L}_{\mathbb{P}^\alpha}(X_T\vert s<\tau)\Big\Vert^2_{-s},
\end{equation}
    with $s > \frac{d}{2}$.
where we use the Fourier-Wasserstein norm we discuss in detail in section \ref{sct: FW mtrc}.  
%{\color{blue}why is the sup in $V_{p,\mu}$ over $\tilde A$ and over $A$ in $V^l_{p,\mu}$?}
Our main result on this problem is the following:

\begin{theorem}\label{thm: trgt}
    Assume that $b(t,x,a) = b^1(t,x)a+ b^2(t,x)$ for suitable bounded functions $b^1,b^2$. Then the following hold:
    \begin{itemize}
        \item[(i)]
    % \ref{asmp: beta}, \ref{asmp: ncs} and \ref{asmp: cnv} hold with either $A$ bounded or $m>0$. {\color{blue}but here $m=1$ right?} 
    It holds $V^l_{\hat{p}, \hat{\mu}} \nearrow V_{\hat{p}, \hat{\mu}}$. 
    \item[(ii)] If $V_{\hat{p}, \hat{\mu}} <\infty$, the infimum in \eqref{eq:def.schro.hard.kill} is attained at a unique feasible $\hat{\alpha} \in \tilde{\mathbb{A}}$
    %This minimizer is unique as soon as $f^1$ is strict convex in $a$.
    %{\color{blue}what is $f_1$? shouldn't we just say that the minimizer is unique?}\par
    and for each $l\ge1$, the problem \eqref{eq:def.schro.penal} admits a unique minimizer 
    %When $m>0$, we can consider 
    $\hat{\alpha}^l \in \mathbb{A}_{BMO}$ satisfying
    $$\mathcal{H}(\mathbb{P}^{\hat{\alpha}} \parallel \mathbb{P}^{\hat{\alpha}^l}) \leq \frac{L^2}{2}\mathbb{E}^{\hat{\alpha}}[\int_0^{T\wedge \tau}\Vert \hat{\alpha}_s - \hat{\alpha}^l_s \Vert^2 ds] \leq \frac{L^2}{m}(V_{\hat{p}, \hat{\mu}} - V^l_{\hat{p}, \hat{\mu}}) \rightarrow 0.$$
    \item[(iii)] If there exists a feasible $\alpha^0 \in \mathbb{A}$, then we also have $\hat{\alpha}\in \mathbb{A}$.
    \item[(iv)] Assume further $k = d$, $b^1$ equals the identity matrix and $b^2 = 0$. If $V_{\hat{p}, \hat{\mu}} <\infty$, then there exist measurable functions $\hat{\phi}, \hat{\psi}: D \rightarrow \mathbb{R}$ such that $\frac{d\mathbb{P}^{\hat\alpha}_{\vert \mathcal{F}_{T\wedge\tau}}}{d\mathbb{P}_{\vert\mathcal{F}_{T\wedge \tau}}} = e^{\hat{\phi}(X_0) + \mathbf{1}_{\lbrace T\wedge\tau \rbrace}\hat{\psi}(X_T)}$ $\mathbb{P}^{\hat{\alpha}}$-a.s. 
\end{itemize}
\end{theorem}
\begin{remark}
    In the setting of Theorem \ref{thm: trgt} $(iv)$, similarly to \cite[(D)]{Leonard14}, one can derive the dual of \eqref{eq:def.schro.hard.kill} to be of the form
     \begin{equation}
     \label{eq:dual.schro}
        \max \int \phi d \nu + \hat{p} \int \psi d\hat{\mu} - \log(\mathbb{E}^\mathbb{P}[e^{\phi(X_0) + \mathbf{1}_{\lbrace T\wedge \tau \rbrace}\psi(X_T)}])
    \end{equation}
    over measurable $\phi, \psi: D \rightarrow \mathbb{R}$. By weak duality, $\hat{\phi}$ and $\hat{\psi}$ are optimal for the dual \eqref{eq:dual.schro}.
\end{remark}
\begin{remark}
    The result above easily extends to the case where the cost $\frac{1}{2}\Vert a\Vert^2$ is replaced by an arbitrary function $f$ satisfying Assumptions  \ref{asmp: ncs} and \ref{asmp: cnv} with $m > 0$. Further, we can add to $J$ a terminal cost $g$ satisfying Assumptions  \ref{asmp: ncs} and \ref{asmp: cnv}. When $f$ and $g$ are not extendable as in Theorem \ref{thm: opt non eq}, the infimum in \eqref{eq:def.schro.penal} can also be taken only over $\mathbb{A}$ instead of $\tilde{\mathbb{A}}$ without changing $V^l_{\hat{p}, \hat{\mu}}$ or the optimal control $\hat{\alpha}^l$. Replacing $\tilde{\mathbb{A}}$ by $\mathbb{A}$ in \eqref{eq:def.schro.hard.kill} can potentially change $V_{\hat{p}, \hat{q}}$ and the rest of the problem as the set of feasible controls shrinks. Still, even if $f$ and $g$ are not extendable, \eqref{eq:def.schro.hard.kill} remains well defined as long as $\hat{p} > 0$.\par
    For $\hat{p} = 0$ (in this case $\hat{\mu}$ becomes irrelevant), in a similar way, it can be shown that the result above remains true. Of course, for general $f$ and $g$, it would require them both to be extendable or both infima need to be taken over $\mathbb{A}$ only.
\end{remark}
% First, we introduce and analyze a constrained version of the celebrated Schrödinger's problem
%  where particles may be killed upon exiting a given domain. We formulate the problem as an entropy minimization under constraints on both the initial distribution and the conditional terminal distribution of surviving particles. Using our results on optimality conditions for McKean–Vlasov control problems, we show how such problems can be rigorously addressed based on a penalization scheme. See Theorem \ref{thm: trgt}.

\subsubsection{Potential mean field games}
 We begin by recalling the definition of mean field games.
  Here, we impose some additional assumption on the Hamiltonian. We need that for any $t$, $x$, $\mu$ and $z$ that $a \mapsto h(t, x, a, \mu, z)$ admits a unique minimizer $a^*(t, x, \mu, z)$ so that $a^*$ is invertible in $z$, i.e.\ there exists a map $(a^*)^{-1}$ in $t$, $x$, $a$ and $\mu$ such that $z = (a^*)^{-1}(t, x, a^*(t, x, \mu, z), \mu)$. For simplicity, we assume $A = \mathbb{R}^k = \mathbb{R}^d$, as well as that $\beta_a(t, x, a, \mu)$ is valued in the space of invertible matrices with $\beta_a^{-1}$ being uniformly bounded. Then, such an inverse is characterized as $(a^*)^{-1}(t, x, a, \mu) = -\beta_a(t, x, a, \mu)^{-1}f_a(t, x, a, \mu)$ and will admit linear growth in $a$.
%Let us describe how we can define the potential game out of the McKean-Vlasov control problem described above. The potential game needs to be defined as a mean field game of controls, i.e.\ there is also interaction through the law of the control. 
Given a measurable flow of measures $\nu: [0, T] \rightarrow \mathcal{P}(\mathcal{C} \times A)$
% representing the joint law of $(X_{\cdot \wedge t}, \alpha_t)$. For $\nu \in \mathcal{P}(\mathcal{C}_D \times A)$, we denote 
with first marginal denoted $\nu^x(\cdot) = \int_A \nu(\cdot, da)$, let $\hat\alpha$ satisfy
\begin{equation*}
    \hat\alpha = \argmin_{\alpha \in \mathbb{A}_{\mathrm{BMO}}}J^{\mathrm{MFG}}(\alpha, \nu)
\end{equation*}
with 
\begin{align*}
    &J^{\mathrm{MFG}}(\alpha, \nu) = \mathbb{E}^{\overline{\mathbb{P}}^{\nu, \alpha}}\bigg[\int_0^T f(s, X_{\cdot \wedge s}, \alpha_s, \nu^x_s) + \int_{\mathcal{C}_D \times A}\frac{\delta f}{\delta m}(s, \tilde{x}, \tilde{a}, \nu^x_s, X_{\cdot\wedge s}) \\
    &\quad + \frac{\delta \beta}{\delta m}(s, \tilde{x}, \tilde{a}, \nu^x_s, X_{\cdot \wedge s})^\top(a^*)^{-1}(s, \tilde{x}, \tilde{a}, \nu^x_s)\nu_s(d\tilde{x}, d\tilde{a})ds + g(X, \nu^x_T) + \int_{\mathcal{C}_D} \frac{\delta g}{\delta m}(\tilde{x}, \nu^x_T, X) \nu^x_T(d\tilde{x})\bigg].
\end{align*}
% In the following, we will only consider control processes out of $\mathbb{A}_{BMO}$. Given any $\nu$, the state process is again controlled weakly, but now through the change of measure defined through 
and $\frac{d\overline{P}^{\nu, \alpha}}{d\mathbb{P}} = \mathcal{E}(\int_0^\cdot \beta(s, X_{\cdot \wedge s}, \alpha_s, \nu^x_s)dW_s)_T$.
% with the corresponding Wiener process $\overline{W}^{\mu, \alpha}_\cdot  = W_\cdot - \int_0^\cdot \beta(s, X_{\cdot \wedge s}, \alpha_s, \nu^x_s)ds$.\par
% We define the cost functional of the potential game as
%\begin{eqnarray*}
%    &&J^{MFG}(\alpha, \nu) = \mathbb{E}^{\overline{\mathbb{P}}^{\nu, \alpha}}[\int_0^T f(s, X_{\cdot \wedge s}, \alpha_s, \nu^x_s) + \int_{\mathcal{C}_D \times A}\frac{\delta f}{\delta m}(s, \tilde{x}, \tilde{a}, \nu^x_s, X_{\cdot\wedge s}) \\
%    &+& \frac{\delta \beta}{\delta m}(s, \tilde{x}, \tilde{a}, \nu^x_s, X_{\cdot \wedge s})^\top(a^*)^{-1}(s, \tilde{x}, \tilde{a}, \nu^x_s)\nu_s(d\tilde{x}, d\tilde{a})ds + g(X, \nu^x_T) + \int_{\mathcal{C}_D} \frac{\delta g}{\delta m}(\tilde{x}, \nu^x_T, X) \nu^x_T(d\tilde{x})].
%\end{eqnarray*}
A mean field equilibrium is a pair $(\hat \alpha, \nu)\in \mathbb{A}_{\mathrm{BMO}}$ that in addition, for a.e.\ $t \geq 0$, satisfies $\nu_t = \overline{\mathbb{P}}^{\nu,\hat  \alpha} \circ (X_{\cdot \wedge t},\hat \alpha_t)^{-1}$. % = \mathbb{P}^{\hat \alpha} \circ (X_{\cdot \wedge t}, \hat\alpha_t)^{-1}$.
Under the growth conditions of Assumption \ref{asmp: ncs}, $J^{\mathrm{MFG}}(\alpha, \nu)$ will be well defined for any $\nu$ of the form $\overline{\mathbb{P}}^{\nu, \alpha'} \circ (X_{\cdot \wedge t}, \alpha'_t)^{-1}$ for any $\alpha, \alpha' \in \mathbb{A}_{BMO}$.
\begin{remark}
    The running cost of in the mean field game just defined is particularly involved because of the general setting we are considering. The cost can be made much simpler in two interesting and often studied cases:
    \begin{itemize}
        \item[(i)] If the drift $b$ does not depend on the measure argument, we do not need to assume that $a^*$ exists or is invertible. In fact, in this case the term $\frac{\delta \beta}{\delta m}$ will not appear in the running cost.
        \item[(ii)]If $f$ is separable as in Assumption \ref{asmp: cnv}, the cost simplifies then the cost $J^{\mathrm{MFG}}$ depends only on $\nu^x$, not on $\nu$.
        That is, the mean field game is no longer a mean field game of control.
        % that there is no interaction through the law of the control. It can then be considered as a standard mean field game that only requires to consider the flow $\nu^x$ representing the law of $X$.
    \end{itemize}
\end{remark}

%We say a tuple $(\alpha, \nu)$ is a solution of this mean field game if $J(\alpha, \nu) = \inf_{\alpha' \in \mathbb{A}_{BMO}} J(\alpha', \nu)$ and secondly, for a.e.\ $t \geq 0$, we have $\nu_t = \overline{\mathbb{P}}^{\nu, \alpha} \circ (X_{\cdot \wedge t}, \alpha_t)^{-1} = \mathbb{P}^\alpha \circ (X_{\cdot \wedge t}, \alpha_t)^{-1}$.\par

\begin{theorem}\label{thm: pot mfg}
    Let Assumptions \ref{asmp: beta} and \ref{asmp: ncs} hold and assume that $f$ is  strongly convex in $a$, that $\vert f(t, x, 0, \mu)\vert$, $\vert g \vert$, $\vert \frac{\delta g}{\delta m} \vert$ and $\vert \frac{\delta \beta}{\delta m}\vert$ are uniformly bounded, and that $\vert \frac{\delta f}{\delta m}(t, x, a, \mu, \tilde{x})\vert \leq M(1 + \Vert a \Vert^2)$ for some $M>0$. Further assume either that $a^*$ exists and is invertible as described above or that $b$ is independent of $\mu$. Then, for any $\alpha \in \mathbb{A}_{BMO}$ such that $J(\alpha) = \inf_{\alpha' \in \mathbb{A}_{BMO}} J(\alpha')$, we have that $(\alpha, (\mathcal{L}_{\mathbb{P}^\alpha}(X_{\cdot \wedge t}, \alpha_t))_{t \in [0, T]})$ is a mean field equilibrium of the game with cost $J^{\mathrm{MFG}}$.\par
    If in addition Assumption \ref{asmp: cnv} holds, there is at most one solution to the mean field game with cost $J^{\mathrm{MFG}}$.
\end{theorem}

\section{Discussion and examples}\label{sct: expl}
\subsection{The Le Cam distance}
\label{sub:the_le_cam_distance}
In the analysis of mean field models, the Wasserstein metric is commonly employed on the space of probability measures. However, as noted for example in \cite{Carmona15, Z24}, the total variation distance often proves more suitable when addressing control problems in the probabilistic weak formulation, which is the setting considered here. This preference stems from the fact that the total variation distance is particularly well-suited for analyzing densities of probability measures. 
The Le Cam distance is similar in spirit, while generalizing frameworks where the regularity in measure holds with respect to the total variation since $d_{TV} \leq d_{LC}$.

%      In mean field literature, the most commonly used distance on the space of measures is the Wasserstein distance. In contrast, we have formulated our results in terms of the Le Cam distance which, to the best of our knowledge, has not yet seen usage in this field. This discrepancy goes back to the different approaches in the strong and weak formulation. As it has already been observed in e.g. \cite{Carmona15, Z24}, when working with the weak formulation, it is more convenient to consider the total variation distance, as it naturally emerges from the analysis of the densities. As we have noted before, $d_{TV} \leq d_{LC}$ and we have thus decided to use the Le Cam distance as a way to generalize our results compared to if we were only using the total variation distance.\par
The variational formulation of the total variation distance shows how the total variation distance behaves as an $\mathbb{L}^1$-distance. In fact, in the presence of a dominating measure, it is exactly the $\mathbb{L}^1$-distance between the densities. 
In particular, the total variation distance relates to the pairing between measures and bounded measurable functions. 
To some extent, the Le Cam distance can be understood as an $\mathbb{L}^2$ version of the total variation distance. For instance, whilst the total variation can be used to bound $\vert\int \phi d(\mu - \mu')\vert \leq \Vert \phi \Vert_{\infty} d_{TV}(\mu, \mu')$ for bounded $\phi$, the Le Cam distance relates to the pairing between measures and square integrable function.
That is,
\begin{equation}\label{estm LC}
    \vert \int \phi d(\mu-\mu')\vert \leq 2 \bigg(\int  \phi^2 d(\frac{\mu + \mu' }{2})\bigg)^{\frac{1}{2}} d_{LC}(\mu, \mu')
\end{equation}
for any $\phi$ such that $\int \phi^2 d\mu<\infty$ and $\int \phi^2 d\mu'<\infty$.\par
More importantly, the choice of distance used has major implications for the derivative in measure of linearly differentiable functions. In fact, a linear differentiable $u : \mathcal{P}(E) \rightarrow \mathbb{R}^r$ for some $r \geq 1$, is Lipschitz with respect to $d_{TV}$ if $\frac{\delta u}{\delta m}$ is uniformly bounded. 
For the Le Cam distance, we have the following result.
\begin{lemma}\label{lem: LCLips}
    Given a Polish space $E$ and $r \geq 1$, let $u: \mathcal{P}(E) \rightarrow \mathbb{R}^r$ be a linear differentiable function with derivative satisfying $\int_E \Vert \frac{\delta u}{\delta m}(\mu, \tilde{x})\Vert^2 \mu(d\tilde{x}) \leq C^2 $ for any $\mu \in \mathcal{P}(E)$ and some $C>0$. Then, $u$ is $4C$ Lipschitz with respect to the Le Cam distance.
\end{lemma}
\begin{proof}
    For any $\mu, \mu' \in \mathcal{P}(E)$, by \eqref{estm LC}, we have
    \begin{align*}
        \Vert u(\mu) - u(\mu') \Vert & = \bigg\Vert \int_0^1 \int_E \frac{\delta u}{\delta m}(\lambda \mu' + (1-\lambda)\mu, \tilde{x}) (\mu' - \mu)(d\tilde{x})d \lambda \bigg\Vert\\
        &\leq 2 \int_0^1 \bigg(\int_E \Big\Vert \frac{\delta u}{\delta m}(\lambda \mu' + (1-\lambda)\mu, \tilde{x})\Big\Vert^2\frac{\mu + \mu'}{2}(d\tilde{x})\bigg)^{\frac{1}{2}}d_{LC}(\mu, \mu')d\lambda\\
        &\leq 2 d_{LC}(\mu, \mu')\bigg\{\int_0^{\frac{1}{2}}\bigg(\frac{1}{2\lambda}\int_E \Big\Vert \frac{\delta u}{\delta m}(\lambda\mu' + (1-\lambda)\mu, \tilde{x}) \Big\Vert^2(\lambda \mu' + (1-\lambda) \mu)(d\tilde{x}) \bigg)^{\frac{1}{2}}d\lambda \\
        &\quad + \int_{\frac{1}{2}}^1\bigg( \frac{1}{2(1-\lambda)}\int_E \Big\Vert\frac{\delta u}{\delta m}(\lambda\mu' + (1-\lambda)\mu, \tilde{x})\Big\Vert^2(\lambda \mu' + (1-\lambda) \mu)(d\tilde{x}) \bigg)^{\frac{1}{2}}d\lambda \bigg\}\\ 
        & \leq 4Cd_{LC}(\mu, \mu')
    \end{align*}
and thus showing the Lipschitz property for $u$. 
\end{proof}
To further illustrate the gain in generality afforded to us by the Le Cam metric, consider the (important) case where the drift depends on an interaction kernel, i.e. $b(t, x, a, \mu, p) = \int_{\mathbb{R}^d} K(x_t, y_t)\mu(dy)$ for some kernel $K: \mathbb{R}^d \times \mathbb{R}^d \rightarrow \mathbb{R}^d$.
If $b$ is required to be Lipschitz with respect to $d_{TV}$, then $K$ must be uniformly bounded.
If we instead work with $d_{LC}$, we can allow $K$ to admit polynomial growth in its second variable, i.e. $K(x, y) \leq M(1 + \Vert y \Vert)^q$ for some $M>0$, $q>1$, since $\tilde{\mathbb{E}}^\alpha[\Vert K(X_t, \tilde{X}_t)\Vert^2] \leq \tilde{\mathbb{E}}[\Vert K(X_t, \tilde{X}_t) \Vert^4]^{\frac{1}{2}}\mathbb{E}^\alpha[(\frac{d\mathbb{P}^\alpha}{d\mathbb{P}})^2]$ can be bounded uniformly over $\alpha$ for instance if $A$ and $\sigma$ are bounded.
%We will elaborate more on this discussion below {\color{red}Where????}

%In particular, when we are working with a bounded set $A$ and have a priori estimates for the density as described in Remark \ref{rmk: asm meas}, this opens up our result for numerous additional applications. We will see some below as well, but at this point, let us demonstrate it in the lights of interaction potentials in the drift. That is, for some $V: \mathbb{R}^d \times \mathbb{R}^d \rightarrow \mathbb{R}^d$ consider $b$ of the form $b(t, x, a, \mu, p) = \int_{\mathbb{R}^d} V(x_t, y_t)\mu(dy)$ so that $\frac{\delta b}{\delta m}(t, x, a, \mu, p, \tilde{x}) = V(x_t, \tilde{x}_t)$. If $b$ is required to be Lipschitz with respect to $d_{TV}$, this would require $V$ to be uniformly bounded. If we instead work with $d_{LC}$, we can allow $V$ to admit polynomial growth in its second variable, i.e. $V(x, y) \leq M(1 + \Vert y \Vert)^q$ for some $M>0$, $q>1$, since $\tilde{\mathbb{E}}^\alpha[\Vert V(X_t, \tilde{X}_t)\Vert^2] \leq \tilde{\mathbb{E}}[\Vert V(X_t, \tilde{X}_t) \Vert^4]^{\frac{1}{2}}\mathbb{E}^\alpha[(\frac{d\mathbb{P}^\alpha}{d\mathbb{P}})^2]$ can be bounded uniformly over $\alpha$ as soon as $A$ is bounded and $\sigma$ is bounded.
\subsection{The conditional exit control problem}
As already touched upon in the introduction, we can apply our results to a weakly formulated version of the conditional exit control problem. To keep the notation close to the one used in \cite{Achdou21}, let us define it as follows. Let $\sigma$ be constant, $b(s, x, a, \mu, p) = a$, and the cost be of the form $f(t, x, a, \mu, p) = \frac{1}{p}(L(x_t, a) + \Phi(\mu_t))$ and $g(x, \mu, p) = \frac{1}{p}(\Psi(\mu_T) - \epsilon \log(p))$ for some $L: D \times A \rightarrow \mathbb{R}$ that is convex in $a$, $\Phi, \Psi: \mathcal{P}(D)\rightarrow \mathbb{R}$, and $\epsilon \geq 0$. Here, we wrote $\mu_t$ for $\mu \circ (\omega \mapsto \omega_t)^{-1}$ and we consider the action space $A = \lbrace a\in \mathbb{R}^d \vert \Vert a \Vert \leq M_A \rbrace$ for some positive constant $M_A$, or $A = \mathbb{R}^d$. This way, the cost functional we have defined in \eqref{eq: def J} coincides with the one defined in \cite[(6)]{Achdou21}:
\begin{equation}\label{eq: cnd cst}
    J(\alpha) = \int_0^T \mathbb{E}^\alpha[L(X_s, \alpha_s)\vert s<\tau] + \Phi(\mathcal{L}_{\mathbb{P}^\alpha}(X_s \vert s < \tau)) ds + \Psi(\mathcal{L}_{\mathbb{P}^\alpha}(X_T \vert T<\tau))-\epsilon\log(\mathbb{P}^\alpha[T<\tau])
\end{equation}
We will assume that in this form, $f$ and $g$ still satisfies Assumption \ref{asmp: ncs}. This results in slightly different regularity assumptions than in \cite{Achdou21}. The following characterization of optimal controls is an immediate consequence of Theorem \ref{thm: ncs}.
\begin{corollary}
    Let $\alpha \in \mathbb{A}_{BMO}$ be optimal for the conditional exit control problem \eqref{eq: cnd cst} as described above. Let $(Y^\alpha, Z^\alpha)$ be the unique solution of
    \begin{eqnarray}\label{eq: cnd adj}
        &&Y^\alpha_{t\wedge \tau} = \mathbf{1}_{T<\tau}(\frac{1}{\mathbb{P}^\alpha[T<\tau]}\frac{\delta \Psi}{\delta m}(\mathcal{L}_{\mathbb{P}^\alpha}(X_T\vert T<\tau), X_T) - \frac{\epsilon}{\mathbb{P}^\alpha[T<\tau]}) \\
        &+& \int_{t\wedge\tau}^{T\wedge\tau}\frac{1}{\mathbb{P}^\alpha[s<\tau]}(L(X_s, \alpha_s) - \mathbb{E}^\alpha[L(X_s, \alpha_s)\vert s<\tau] + \frac{\delta \Phi}{\delta m}(\mathcal{L}_{\mathbb{P}^\alpha}(X_s\vert s<\tau))ds-\int_{t\wedge\tau}^{T\wedge\tau}Z^\alpha_sdW^\alpha_s.\nonumber
    \end{eqnarray}
    Then, $dt \times \mathbb{P}$ a.e. on $\lbrace t<\tau\rbrace$, we have $\alpha_t \in \argmin_{a\in A} (\mathbb{P}^\alpha[t<\tau] Z^\alpha_t)^\top a + L(X_t, a)$.
\end{corollary}
\begin{remark}
    Let us assume that $\alpha_t = a(t, X_t)$ for some measurable $a:[0, T] \times D \rightarrow A$, i.e.\ it is a control of feedback form. In this case, under additional smoothness assumptions, we expect \ref{eq: cnd adj} to correspond to the solution $u$ of a PDE system. Writing out the system, we observe that it matches the ones in \cite[Theorem 2.6, Theorem 5.2]{Achdou} (Note that as we are using the convention described in Remark \ref{rmk: lin drv}, we do not need all the normalization terms they introduced in $c_1$ and $c_2$).
    % That is, we expect $Y^\alpha_t = u(t, X_t)$ and $Z^\alpha_t = Du(t, X_t)$ where $u$ is a solution to
    % \begin{equation}\label{eq: adj gnr}
    %     \begin{cases}
    %     \partial_t u(t, x)  \frac{\sigma^\top \sigma}{2}\Deltau(t, x) = \frac{H(x, \frac{Du(t, x)}{\mathbb{P}^\alpha[t<\tau]})}{\mathbb{P}^\alpha[t<<\tau]}
    %     \end{cases}
    % \end{equation}
\end{remark}
\begin{remark}
    As the cost functional formulated in \eqref{eq: cnd cst} fails to be convex, our existence result in Theorem \ref{thm: exs} is not immediately applicable. Still, when $\Phi$ and $\Psi$ are linearly convex, for bounded control spaces, as well as for the case $\epsilon > 0$, the proof can be adapted to the conditional problem as well. This relies on the observation that for a minimizing sequence $\alpha^n$, one constructs a limit $\hat{\alpha}$ for which $\mathbb{P}^{\alpha^n}[t<\tau] \rightarrow \mathbb{P}^{\hat{\alpha}}[t<\tau]$ so that the additional terms introduced by the conditioning vanish in the limit.
\end{remark}
\subsection{Convex cost functions}
\label{sub:convex_cost_functions}
The $p$-convexity condition defined in \eqref{eq: def pcnv} was required both for the sufficient condition of optimality and the existence result. 
Since this convexity condition is not standard (for instance it is not implied by joint convexity), let us provide a few examples.
The first example is that of functions that are independent of $\mu$ and $p$ (i.e. the non McKean-Vlasov case). 
Also observe that linear combinations of $p$-convex function with non negative coefficients are $p$-convex. 
% Note that $\phi$ being independent of $x$ does not imply $p$-convexity in general. Let us also point out that $p$-convexity is not implied by joint convexity in $p$ and $\mu$.\par
\subsubsection{Extending linear convexity}
Let $\phi:\mathcal{P}(\mathcal{C}_D)\to \R$ be a linearly convex function, i.e.
$$\phi(\lambda \mu' + (1-\lambda) \mu) \leq \lambda \phi(\mu') + (1 - \lambda)\phi(\mu)$$
for any $\mu, \mu' \in \mathcal{P}(\mathcal{C}_D)$ and $\lambda \in [0, 1]$. Then, it is easy to verify that $\eqref{eq: def pcnv}$ still holds, and $\phi$ is thus also $p$-convex.\par
We often consider functions of the form $\phi(\mu) = \phi_0(\mu \circ \kappa^{-1})$ for some measurable map $\kappa:\mathcal{C}_D \rightarrow E$. Let us first note that linear differentiability of $\phi$ is inherited from linear differentiability of $\phi_0$ since for any $\mu, \mu' \in \mathcal{P}(\mathcal{C}_D)$,
\begin{align*}
    \phi(\mu') - \phi(\mu) &= \phi_0(\mu'\circ\kappa^{-1}) - \phi_0(\mu\circ\kappa^{-1}) \\
    &= \int_E\int_0^1\frac{\delta \phi_0}{\delta m}(\lambda(\mu'\circ\kappa^{-1})+(1-\lambda)(\mu\circ\kappa^{-1}), e)d\lambda(\mu'\circ\kappa^{-1}-\mu\circ\kappa^{-1})(de)\\
    &=\int_{\mathcal{C}_D}\int_0^1\frac{\delta \phi_0}{\delta m}((\lambda\mu'+(1-\lambda)\mu)\circ\kappa^{-1}, \kappa(x))d\lambda(\mu'-\mu)(dx)
\end{align*}
from which we can see that $\frac{\delta \phi}{\delta m}(\mu, x) =\frac{\delta \phi_0}{\delta m}(\mu \circ \kappa^{-1}, \kappa(x))$. Next, it is also immediate to see that if $\phi_0$ is linearly convex, then so is $\phi$ since convex combinations commute with pushforward measures.
Most often in our setting, $E =D$ and $\kappa$ is the projection mapping $\kappa(x) = x_t$ for some $t\geq 0$.\par
As a specific example, let us consider the terminal cost $g(x, \mu, p) = -\int_{\mathcal{C}_D}\int_{\mathcal{C}_D} \Vert x_T - x'_T \Vert^2 \mu(dy)\mu(dy')$ which by the above discussion is differentiable and $p$-convex. If $D$ is bounded, one can even see that $g$ and $\frac{\delta g}{\delta m}$ are bounded. Setting $f=0$ results in the cost functional
$$J(\alpha) = -2 \mathbb{P}^\alpha[T<\tau]\Big(\mathbb{E}^\alpha\Big[(X_T - \mathbb{E}^\alpha[X_T \vert T<\tau])^2\vert T<\tau\Big]\Big)= -2\mathbb{P}^\alpha[T<\tau]\mathbb{V}^\alpha[X_T \vert T<\tau]$$
where $\mathbb{V}^\alpha[X_T \vert T<\tau]$ is the conditional variance of $X_T$ with respect to the probability measure $\mathbb{P}^\alpha$.
Minimizing this cost function amounts to maximizing the likelihood that particles stay in the domain $D$ while also maximizing the variance of the surviving particles at terminal time.

\subsubsection{The Fourier-Wasserstein metric}\label{sct: FW mtrc}
In some cases, (see e.g. \cite{hernandez2023propagation} and Subsection \ref{sct: schrö} below) it is desirable to consider interaction functions that are distances on the set of probability measures.
%As we will see when we discuss a penalization scheme for target problems in Theorem \ref{thm: trgt}, it can be useful to find a function $\phi$ satisfying our conditions that can serve as a distance function. 
There are certainly numerous choices of distance functions, but most lack either differentiability or convexity. As we will observe later, one possibility is to work with $f$-divergences that include among others the Le Cam metric and the relative entropy, but let us first present a much nicer choice: the Fourier Wasserstein-metric introduced in \cite{SonerYan} to study comparison theorems for viscosity solutions of Hamilton-Jacobi-Bellman equation on the Wasserstein space.\par
%Let us briefly recall the construction of the Fourier-Wasserstein distance. We fix some $s > \frac{d}{2}$. 
The Fourier-Wasserstein metric emerges out of the dual norm of Sobolev spaces of order $s$ for some $s>\frac{d}{2}$. 
It can also be characterized explicitly as follows: We write the Fourier basis as $e(x, \xi) := (2\pi)^{-\frac{d}{2}} e^{ix^\top\xi}$. For any finite signed Borel measure $\zeta$ on $\mathbb{R}^d$, we can define its Fourier transform as $\Phi(\zeta)(\xi) = \int \overline{e(x, \xi)}\zeta(dx)$ where $\overline{\cdot}$ denotes complex conjugation. With $\vert \cdot \vert$ also denoting the modulus for complex numbers, note that for any $\xi$, we have $\vert \Phi(\zeta)(\xi)\vert \leq \vert \zeta \vert$ with $\vert \zeta \vert$ denoting the total variation of $\zeta$. Thus, for any such $\zeta$, the Fourier-Wasserstein metric takes the form
$$\Vert \zeta \Vert^2_{-s}:= \int_{\mathbb{R}^d}(1+\Vert \xi\Vert^2)^{-s}\vert \Phi(\zeta)(\xi)\vert^2d\xi$$
which can be shown to be a norm on the space of finite signed measures.\par
For some fixed $p^0 \in [0, 1]$ and $\mu^0 \in \mathcal{P}(D)$ let us consider the terminal cost $g(x, \mu, p) = \frac{1}{2p}\Vert p (\mu \circ (\omega \mapsto \omega_T)^{-1}) - p^0 \mu^0\Vert_{-s}^2$. If one sets $f=0$, then the cost functional becomes 
$$J(\alpha) =  \frac{1}{2} \Big\Vert \mathbb{P}^\alpha[s<\tau]\mathcal{L}_{\mathbb{P}^\alpha}(X_T\vert T<\tau) - p^0\mu^0 \Big\Vert_{-s}^2,$$ incentivizing us to bring $\mathbb{P}^\alpha[s<\tau]$ "close" to $p^0$ and $\mathcal{L}_{\mathbb{P}^\alpha}(X_T\vert T<\tau)$ "close" to $\mu^0$.\par
The advantages of working with the Fourier-Wasserstein metric become clear when considering its differentiability and the properties of its derivatives. For simplicity of the subsequent calculation, we also write $\mu_T$ for $\mu \circ (\omega \mapsto \omega_T)^{-1}$. Following \cite[Lemma 5.2]{SonerYan}, one can see that $g$ is linearly differentiable with
\begin{eqnarray*}
    \frac{\delta g}{\delta m}(\mu, p, x)   &=&\frac{1}{2p}\int_{\mathbb{R}^d} (1+ \Vert \xi \Vert^2)^{-s}p(\Phi(p\mu_T-p^0\mu^0)(\xi)(e(x, \xi) - \overline{\Phi(\mu_T)(\xi)})\\
    &+& \overline{\Phi(p\mu_T - p^0\mu^0)(\xi)}(\overline{e(x, \xi)} - \Phi(\mu_T)(\xi))d\xi\\
    &=&\int_{\mathbb{R}^d}(1+\Vert\xi\Vert^2)^{-s}\mathfrak{Re}(\Phi(p\mu_T - p^0\mu^0)(\xi)(e(x, \xi) - \overline{\Phi(\mu_T)(\xi)}))d\xi
\end{eqnarray*}
where $\mathfrak{Re}$ denotes the real part of a complex number. Note that $\frac{\delta g}{\delta m}$ is bounded. Also,
$$g_p(\mu, p) = -\frac{1}{2p^2}\Vert p\mu_T-p^0\mu^0\Vert_{-s}^2 + \frac{1}{p}\int_{\mathbb{R}^d}(1+\Vert\xi\Vert^2)^{-s}\mathfrak{Re}(\Phi(p\mu_T-p^0\mu^0)(\xi)\overline{\Phi(\mu_T)(\xi)})d\xi.$$
Further, recall that for any $z^1, z^2 \in \mathbb{C}$, we have $\frac{1}{2}\vert z^1 \vert^2 - \frac{1}{2}\vert z^2 \vert^2 \geq \mathfrak{Re}((z^1 - z^2)\overline{z^2})$. Thus,
\begin{align*}
    &\frac{1}{2}\Vert p'\mu_T' - p^0\mu_T^0\Vert_{-s}^2 - \frac{1}{2}\Vert p\mu_T-p^0\mu^0\Vert_{-s}^2 \\
    &\quad =\frac{1}{2}\int_{\mathbb{R}^d}(1 + \Vert \xi \Vert^2)^{-s})\Big(\vert \Phi(p'\mu_T'-p^0\mu^0)(\xi)\vert^2 -\vert\Phi(p\mu_T-p^0\mu^0)(\xi) \vert^2\Big)d\xi\\ 
    &\quad \geq\frac{p'-p}{2p}\Vert p\mu_T-p^0\mu^0\Vert_{-s}^2 \\
    &\qquad + \int_{\mathbb{R}^d}(1+\Vert\xi\Vert^2)^{-s}\mathfrak{Re}\Big(\Phi(p\mu_T-p^0\mu^0)(\xi)(\overline{\Phi(p'\mu_T'-p\mu_T)(\xi)}- (p'-p)\overline{\Phi(\mu_T)(\xi)}) \Big)d\xi \\
    &\qquad - \frac{p'-p}{2p}\Vert p\mu_T-p^0\mu^0\Vert_{-s}^2 + (p' - p)\int_{\mathbb{R}^d}(1 + \Vert\xi\Vert^2)^{-s}\mathfrak{Re}(\Phi(p\mu_T - p^0\mu^0)(\xi)\overline{\Phi(\mu_T)(\xi)})d\xi
\end{align*}
showing that $g$ is $p$-convex by Proposition \ref{prop: pcnv diff}. We thus see that such $g$ satisfies Assumption \ref{asmp: ncs}, \ref{asmp: cnv}, and even the extendability required in Theorem \ref{thm: opt non eq}.
\subsubsection{Divergences}
For a Polish space $E$, a finite measure $\nu$ on $E$, and measurable maps $\kappa:\mathcal{C}_D \rightarrow E$ and $F:[0, \infty) \rightarrow \mathbb{R}$, we define 
$$\phi(\mu, p) = \frac{1}{p} \int_E F\Big(p\frac{d(\mu\circ \kappa^{-1})}{d\nu}\Big)d\nu.$$ We take $F$ to be non negative, convex and such that $F(1) = 0$. Of course, $\phi$ can only be well defined on a subset of measure in $\mathcal{P}(\mathcal{C}_D)$ satisfying some a priori bounds on the density as described in Remark \ref{rmk: asm meas}. When it can be guaranteed that $\mu \circ \kappa^{-1} \sim \nu$, we do not need $F$ to be defined at zero. In the context of probability measures, such maps are known as $f$-divergences, see e.g.\ \cite[Chapter 7]{Polyanskiy}.\par
When $F$ is convex, it is immediate to see that \eqref{eq: def pcnv} is satisfied so that $\phi$ is $p$-convex. Further, we can see that when $F$ equals zero only at $1$, we have $\phi\geq 0$ and $\phi(\mu, p) = 0$ if and only if $p\mu\circ \kappa^{-1} = \nu$ by Jensen's inequality. 
This shows how such $\phi$ can also be interpreted as a distance function. Common choices for $F$ are $F(x) = -\log(x) + x -1$, $F(x) = x\log(x) - x + 1$, $F(x) = \frac{1}{2}\vert x - 1 \vert$, and $F(x) = \frac{(1 - x)^2}{2x+2}$. The first two choices are related to the relative entropy, and the latter two to the total variation distance and the Le Cam distance.\par

Differentiability of $\phi$ is also easily checked. When $F$ is differentiable (in a weak sense suffices), we have for any $\mu, \mu' \in \mathcal{P}(\mathcal{C}_D)$, $p, p' \in (0, 1]$,
\begin{align*}
    \phi(\mu', p') & - \phi(\mu, p) = \int_E \frac{1}{p'}F\Big(p'\frac{d(\mu' \circ \kappa^{-1})}{d\nu} \Big) - \frac{1}{p}F\Big(p\frac{d(\mu\circ\kappa^{-1})}{d\nu} \Big) d\nu\\
    &=\int_E\int_0^1 -\frac{p'-p}{(p^\lambda)^2} F\Big(p^\lambda \frac{d(\mu^\lambda\circ\kappa^{-1})}{d\nu}\Big)d\lambda d\nu\\
    &\quad +\int_E\int_0^1 \frac{1}{p^\lambda}F'\Big(p^\lambda\frac{d(\mu^\lambda\circ\kappa^{-1})}{d\nu} \Big)\Big\{(p'-p)\frac{d(\mu^\lambda\circ\kappa^{-1})}{d\nu} + p^\lambda\Big(\frac{d(\mu'\circ\kappa^{-1})}{d\nu}-\frac{d(\mu\circ\kappa^{-1})}{d\nu}\Big)\Big\} d\lambda d\nu\\
    &= \int_{\mathcal{C}_D} \int_0^1 F'\Big(p^\lambda\frac{d(\mu^\lambda \circ \kappa^{-1})}{d\nu}\circ \kappa \Big) d\lambda d(\mu'-\mu) \\
    &\quad + \int_E\int_0^1 \Big\{-\frac{1}{(p^\lambda)^2}F\Big(p^\lambda\frac{d(\mu^\lambda\circ\kappa^{-1})}{d\nu} \Big) + \frac{1}{p^\lambda}F'\Big(p^\lambda\frac{d(\mu^\lambda\circ\kappa^{-1})}{d\nu}\Big)\frac{d(\mu^\lambda\circ\kappa^{-1})}{d\nu}\Big\}(p'-p)d\lambda d\nu
\end{align*}
where we wrote $p^\lambda = \lambda p' + (1-\lambda)p$ and $\mu^\lambda = \lambda\mu' + (1-\lambda)\mu$. This shows that $\phi$ is differentiable and its derivatives are given by
$$\frac{\delta \phi}{\delta m}(\mu, p, x) = F'\Big(p\frac{d(\mu\circ\kappa^{-1})}{d\nu}(\kappa(x)) \Big) - \int_{\mathcal{C}_D}F'\Big(p \frac{d(\mu\circ\kappa^{-1})}{d\nu}(\kappa(x)) \Big)\mu(dx)$$
and
$$\phi_p(\mu, p) = -\frac{1}{p^2}\int_E F\Big(p\frac{d(\mu\circ\kappa^{-1})}{d\nu} \Big)d\nu + \frac{1}{p}\int_{\mathcal{C}_D}F'\Big(p\frac{d(\mu\circ\kappa^{-1})}{d\nu}(\kappa(x))\Big)\mu(dx).$$
The drawback of working with divergences compared to the Fourier-Wasserstein metric is that the integrability conditions we require can be difficult to check, in particular when $F'$ is not bounded. In practice, it is very important to choose $\nu$ properly to  insure that $\phi$ is well defined. 
A more concrete example is obtained by choosing $E = D$, $\kappa(x) = x_T$ and the ``baseline'' measure $\nu = \mathbb{P}^{\alpha^0}[T<\tau]\mathcal{L}_{\mathbb{P}^{\alpha^0}}(X_T \vert T<\tau)$ for some $\alpha^0 \in \mathbb{A}$. When $A$ is bounded, it is easily checked that if $F$ and $F'$ admit polynomial growth in $x$ and $\frac{1}{x}$, then $\phi$ is well defined and satisfies the growth conditions in Assumption \ref{asmp: ncs}. 
%Such choices of $\nu$ can be useful when one wants the distribution under $\alpha$ to be similar to the one of $\alpha^0$ while considering other objectives as well.\par

Another interesting example is obtained when $A$ and $D$ are bounded. In this case, one can choose $\nu = dx$ to be Lebesgue measure on $D$. Following a similar argument as in \cite[Lemma C.1]{Tough}, for bounded drift, one can show that $\frac{d\mathcal{L}_{\mathbb{P}^\alpha}(X_T\vert T<\tau)}{dx}$ is uniformly bounded when the boundary of $D$ is nice enough. Then, for any continuous $F$ on $[0, \infty)$, the growth conditions in Assumption \ref{asmp: ncs} are satisfied. If we take e.g.\ $F(x) = x\log(x) -x + 1$, the resulting $\phi$ serves as some kind of entropic regularizer.

\section{Proof of the Pontryagin maximum principle}\label{sct: ptry}\label{sct: main}
\subsection{Well-posedness and preliminary estimates for the drift}
We begin by proving the existence of the measures $\P^\alpha$. We will also show two key estimates that will be used in the proof of the maximum principle

\begin{proof}[Proof of Proposition \ref{prop: Pa exst}]
    We first start by considering a fixed $\alpha \in \mathbb{A}_{BMO}$. Let $\Xi$ be the space of all measurable flows $(\mu, p): [0, T] \rightarrow \mathcal{P}(\mathcal{C}_D) \times (0, 1]$ equipped with the metric 
    $$d_\Xi((\mu, p), (\mu', p'))^2 := \esssup_{t \in [0, T]} \big( d_{LC}(\mu_t, \mu'_t)^2 + (p_t - p'_t)^2\big).$$ 
    As we have assumed $\beta$ to be Lipschitz, for any $(\mu, p) \in \Xi$, we can see that $\int_0^\cdot \beta(s, X_{\cdot \wedge s}, \alpha_s , \mu_{\cdot \wedge s}, p_s)dW_s$ is a $\mathbb{P}$-BMO martingale with the BMO norm being bounded over all $(\mu, p)$. Hence, we can define a map $\Psi: \Xi \rightarrow \mathcal{P}(\Omega)$ where $\Psi(\mu, p) = Q^{\mu,p}$ is the measure equivalent to $\mathbb{P}$ given by $\frac{dQ^{\mu, p}}{d\mathbb{P}} = \mathcal{E}(\int_0^\cdot \beta(s, X_{\cdot \wedge s}, \alpha_s, \mu_{\cdot \wedge s}, p_s)dW_s)_{T\wedge \tau}$. By \cite[Theorem A.8.24.]{CE}, there exists some $\rho>0$ depending on $\Vert \alpha \Vert_{BMO}$ such that $\mathbb{E}^{Q^{\mu, p}}[(\frac{d\mathbb{P}}{dQ^{\mu, p}})^{1 + \rho}] = \mathbb{E}[(\frac{d\mathbb{P}}{dQ^{\mu, p}})^\rho]$ is finite and uniformly bounded over all $(\mu, p)$. In particular, $Q^{\mu, p}[T<\tau]^{\frac{\rho}{1 + \rho}} \geq \mathbb{P}[T<\tau]\mathbb{E}^{Q^{\mu, p}}[(\frac{d\mathbb{P}}{dQ^{\mu, p}})^{1 + \rho}]^{-\frac{1}{1 + \rho}}$ is bounded away from zero uniformly in $(\mu, p)$ and we can fix a bound $p_\alpha > 0$ such that for any $(\mu,p)$ and $t \geq 0$, we have 
    \begin{equation}
    \label{eq:def.palpha}
        Q^{\mu, p}[t < \tau] \geq Q^{\mu, p}[T < \tau] \geq p_\alpha. 
    \end{equation}
    Additionally, note that by \cite[Theorem 8.8.21]{CE}, \cite[Theorem 3.1]{Kazamaki}, and Hölder's inequality, we have $\mathbb{E}^{Q^{\mu, p}}[\int_0^{T\wedge \tau}\Vert \alpha_s\Vert^2 ds] < \infty$, which guarantees that $\alpha$ is admissible.\par

    We can also define the map $\Phi: \Xi \rightarrow \Xi, (\mu, p) \mapsto (\mathcal{L}_{Q^{\mu, p}}(X_{\cdot \wedge t}\vert t < \tau), Q^{\mu, p}[t < \tau])_{t \in [0, T]}$. By \cite[(7.33),(7.35)]{Polyanskiy}, we have
    $$d_{LC}\Big(Q^{\mu, p}_{\vert \mathcal{F}_t}[\cdot \vert t<\tau], Q^{\mu', p'}_{\vert \mathcal{F}_t}[\cdot \vert t< \tau]\Big)^2 \leq \mathcal{H}\Big(Q^{\mu, p}_{\vert \mathcal{F}_t}[\cdot \vert t<\tau]\parallel Q^{\mu', p'}_{\vert \mathcal{F}_t}[\cdot \vert t< \tau]\Big).$$
    Then, by \cite[Theorem 2.15]{Polyanskiy}, we have
    \begin{align*}
        p_\alpha\Big(d_{LC}\big(&Q^{\mu, p}_{\vert \mathcal{F}_t}[\cdot \vert t<\tau] , Q^{\mu', p'}_{\vert \mathcal{F}_t}[\cdot \vert t< \tau]\big)^2 + \vert Q^{\mu, p}[t < \tau] - Q^{\mu', p'}[t < \tau] \vert^2\Big)\\
        &\leq Q^{\mu, q}[t <\tau] \mathcal{H}\big(Q^{\mu, p}_{\vert \mathcal{F}_t}[\cdot \vert t<\tau]\parallel Q^{\mu', p'}_{\vert \mathcal{F}_t}[\cdot \vert t< \tau]\big) + \mathcal{H}\big(Q^{\mu, q} \circ \mathbf{1}_{\lbrace t < \tau \rbrace}^{-1} \parallel Q^{\mu', q'} \circ \mathbf{1}_{\lbrace t < \tau \rbrace}^{-1}\big)\\ 
        &\leq \mathcal{H}(Q^{\mu, q}_{\vert \mathcal{F}_t}, Q^{\mu', q'}_{\vert \mathcal{F}_t})\\
        &=  \frac{1}{2}\mathbb{E}^{Q^{\mu, p}}\bigg[\int_0^{t\wedge \tau} \Vert \beta(s, X_{\cdot \wedge s}, \alpha_s, \mu_s, p_s) - \beta(s, X_{\cdot \wedge s}, \alpha_s, \mu'_s, p'_s)\Vert^2ds \bigg]\\ 
        &\leq L^2\int_0^t d_{LC}(\mu_s, \mu'_s)^2 + \vert p_s - p'_s \vert^2 ds.
    \end{align*}
    In particular, writing $(\mu^n, p^n) = \Phi^n(\mu, p)$ and $(\mu'^n, p'^n) = \Phi^n(\mu', p')$ and iterating the estimate above gives
    $$d_\Xi((\mu^n, p^n), (\mu'^n, p'^n))^2 \leq \frac{L^{2n}}{p_\alpha^n} \frac{T^n}{n!} d_\Xi((\mu, p), (\mu', p'))^2,$$
    showing that for sufficiently large $n$, $\Phi^n$ is a contraction.\par
    Let $d'_\Xi$ be defined just like $d_\Xi$ but with $d_{LC}$ being replaced by the total variation metric. This way, $\mathcal{P}(\mathcal{C}_D)$ can be seen as a closed bounded subset of the Banach space of all finite signed measures equipped with the total variation metric. $\Xi$ can then be seen as a subset of a Bochner space. If we restrict ourselves to the subset $\Xi^{p_\alpha}$ of flows for which $p_t \geq p_\alpha$, $\Xi^{p_\alpha}$ is closed and hence complete. Further, As $d_{TV}^2 \leq d_{LC}^2 \leq d_{TV}$, completeness of $\Xi^{p_\alpha}$ under $d'_\Xi$ is equivalent to completeness under $d_\Xi$. As $\Phi$ maps into $\Xi^{p_\alpha}$, by Banach's fixed point theorem, there thus must be a unique fixed point $(\mu^*, p^*)$, and $\mathbb{P}^\alpha := Q^{\mu^*, p^*}$ is our desired law.\par
    Now, let us consider any $\alpha\in \mathbb{A}$. We define $\alpha^n$ as 
    \begin{equation}
    \label{eq:def.alphan}
        \alpha^n_t:=\mathbf{1}_{\lbrace t < \tau^n \rbrace}\alpha_t \text{ with } 
        \tau^n := \tau \wedge \inf\bigg\lbrace t \geq 0 \vert \int_0^t \Vert \alpha_s \Vert^2 ds \geq n \bigg\rbrace. 
    \end{equation}
    Since $\mathbb{E}^\alpha[\int_0^{T \wedge \tau} \Vert \alpha_s \Vert^2 ds] < \infty$, we a.s.\ have $\tau^n \nearrow \tau$. It is clear that $\alpha^n \in \mathbb{A}_{BMO}$ and thus $\mathbb{P}^{\alpha^n}$ exists by the construction above. Then, for any $t \geq 0$,
    \begin{align*}
        \mathcal{H}(\mathbb{P}^\alpha_{\vert \mathcal{F}_t} \parallel \mathbb{P}^{\alpha^n}_{\vert \mathcal{F}_t}) & \leq \frac{3L^2}{2}\bigg(\mathbb{E}^\alpha\bigg[\int_{t \wedge \tau^n}^{t\wedge \tau} \Vert \alpha_s\Vert^2ds\bigg]
        + \int_0^t d_{LC}(\mathcal{L}_{\mathbb{P}^\alpha}(X_{\cdot \wedge s}\vert s<\tau), \mathcal{L}_{\mathbb{P}^{\alpha^n}}(X_{\cdot \wedge s}\vert s<\tau))^2\\ 
        &\quad  + \vert \mathbb{P}^\alpha[s<\tau] - \mathbb{P}^{\alpha^n}[s<\tau]\vert^2 ds\bigg)\\
        &\leq \frac{3L^2}{2}\bigg(\mathbb{E}^\alpha\bigg[\int_{T \wedge \tau^n}^{T\wedge\tau} \Vert \alpha_s \Vert^2ds\bigg] +\frac{1}{q_\alpha} \int_0^t \mathcal{H}(\mathbb{P}^\alpha_{\vert \mathcal{F}_s}, \mathbb{P}^{\alpha^n}_{\vert \mathcal{F}_s})ds\bigg)
    \end{align*}
    so that by Grönwall's inequality, $\mathcal{H}(\mathbb{P}^\alpha\parallel\mathbb{P}^{\alpha^n}) \leq \frac{3L^2e^{\frac{3L^2}{2q_\alpha}T}}{2}\mathbb{E}^\alpha[\int_{T\wedge\tau^n}^{T\wedge\tau}\Vert \alpha_s\Vert^2ds]$ which converges to zero by dominated convergence. Note that this convergence also shows uniqueness of $\mathbb{P}^\alpha$.
\end{proof}
Having wellposedness of the controlled measures $\P^{\alpha}$ for fixed $\alpha$, let us its establish with respect to $\alpha$.

\begin{proposition}\label{prop: stab P}
    Under Assumption \ref{asmp: beta}, for any $\Gamma > 0$, there exists a constant $C_P^\Gamma$ only depending on $L$, $T$, and $\Gamma$ such that for any $\alpha, \alpha' \in \mathbb{A}_{BMO}$ with $\Vert \alpha \Vert_{BMO}, \Vert \alpha' \Vert_{BMO} \leq \Gamma$, we have
    $$d_{LC}\big(\mathcal{L}_{\mathbb{P}^\alpha}(X_{\cdot \wedge t}\vert t< \tau), \mathcal{L}_{\mathbb{P}^{\alpha'}}(X_{\cdot \wedge t}\vert t< \tau)\big)^2 + \big(\mathbb{P}^\alpha[t < \tau] - \mathbb{P}^{\alpha'}[t < \tau] \big)^2 \leq C^\Gamma_P \Vert \alpha-\alpha'\Vert_{BMO}^2.$$
\end{proposition}
Note that in \eqref{eq:def.palpha}, the constant $p_\alpha$ depends on $\alpha$ only through $\Vert \alpha \Vert_{BMO}$. For any $\Gamma > 0$, we can just write $p_\Gamma$ as a uniform lower bound for $\mathbb{P}^\alpha[t<\tau]$ over all $\alpha$ such that $\Vert \alpha \Vert_{BMO} < \Gamma$.
\begin{proof}
    We have already seen that $\int_0^\cdot \beta(\Theta^\alpha_s)dW_s$ and $\int_0^\cdot \beta(\Theta^{\alpha'}_s)dW_s$ are $\mathbb{P}$-BMO martingales. By \cite[Theorem 3.6]{Kazamaki}, we can thus see that $\int_0^\cdot \beta(\Theta^\alpha_s) - \beta(\Theta^{\alpha'}_s) dW^{\alpha}_s$ is a $\mathbb{P}^\alpha$-BMO martingale. Additionally, by \cite[Theorem 3.6]{Kazamaki}, there exists a constant $C$ that depends only on $\Vert \int_0^\cdot \beta(\Theta^\alpha_s)dW_s \Vert_{BMO}$ and thus $\Gamma$, such that $\Vert \int_0^\cdot \beta(\Theta^\alpha_s) - \beta(\Theta^{\alpha'}_s) dW^\alpha_s \Vert_{\mathbb{P}^\alpha-BMO} \leq C \Vert \int_0^\cdot \beta(\Theta^\alpha_s) - \beta(\Theta^{\alpha'}_s) dW^\alpha_s \Vert_{BMO}$. Then, as the proof of Proposition \ref{prop: Pa exst}, using Assumption \ref{asmp: beta} we have
    \begin{align*}
        p_\Gamma\Big(&d_{LC}(\mathcal{L}_{\mathbb{P}^\alpha}(X_{\cdot \wedge t}\vert t<\tau), \mathcal{L}_{\mathbb{P}^{\alpha'}}(X_{\cdot \wedge t}\vert t<\tau)^2 + \vert \mathbb{P}^\alpha[t<\tau] - \mathbb{P}^{\alpha'}[t<\tau]\vert^2\Big)\\ 
        &\leq\mathcal{H}(\mathbb{P}^\alpha_{\vert \mathcal{F}_t}, \mathbb{P}^{\alpha'}_{\vert \mathcal{F}_t})
         =\frac{1}{2}\mathbb{E}^\alpha\bigg[\int_0^{t\wedge \tau} \Vert \beta(\Theta^\alpha_s) - \beta(\Theta^{\alpha'}_s)\Vert^2ds \bigg]\\ 
        & \leq\frac{1}{2}\bigg\Vert \int_0^\cdot \beta(\Theta^\alpha_s) - \beta(\Theta^{\alpha'}_s)dW^\alpha_s\bigg\Vert_{\mathbb{P}^\alpha\text{-BMO}}^2 
        \leq \frac{C^2}{2} \bigg\Vert \int_0^\cdot \beta(\Theta^\alpha_s) -\beta(\Theta^{\alpha'}_s) dW_s \bigg\Vert_{\text{BMO}}^2 \\
        &\leq \frac{3C^2L^2}{2}\bigg(\Vert \alpha - \alpha' \Vert_{BMO}^2 + \int_0^t d_{LC}(\mathcal{L}_{\mathbb{P}^\alpha}(X_{\cdot \wedge s}\vert s< \tau), \mathcal{L}_{\mathbb{P}^{\alpha'}}(X_{\cdot \wedge s}\vert s< \tau))^2\\ 
        &\qquad \qquad \qquad + \vert \mathbb{P}^\alpha[s<\tau] - \mathbb{P}^{\alpha'}[s < \tau]\vert^2ds\bigg)
    \end{align*}
and the result follows from Grönwall's inequality.
\end{proof}
To conclude this subsection, let us prove approximation of general controls by controls in $\mathbb{A}_{BMO}$.
\begin{lemma}\label{lem: BMO aprx}
    Under Assumptions \ref{asmp: beta} and \ref{asmp: ncs}, assume additionally that $b$ is independent of $\mu$ and $p$ and that $\vert g(x, \mu^0, 1)\vert$ and $\vert f(t, x, 0, \mu^0, 1) \vert$ are uniformly bounded for some fixed $\mu^0$. Then, for any $\alpha \in \mathbb{A}$, there is a sequence $(\alpha^n)_{n \geq 1}$ in $ \mathbb{A}_{BMO}$ such that $J(\alpha^n) \rightarrow J(\alpha)$ and $\mathcal{H}(\mathbb{P}^\alpha \parallel \mathbb{P}^{\alpha^n})\rightarrow 0$.
\end{lemma}

\begin{proof}
    Let $(\alpha^n)_{n \geq 1}$ be a sequence of elements of $\mathbb{A}_{BMO}$ and $\tau^n$ be defined by \eqref{eq:def.alphan}. %just like in the proof of Proposition \ref{prop: Pa exst} as the stopped control. 
    In particular, $\tau^n \rightarrow \tau$ a.s. and we have $\mathcal{H}(\mathbb{P}^\alpha \parallel \mathbb{P}^{\alpha^n})\rightarrow 0$ which by \cite[Theorem 2.15]{Polyanskiy} implies for any $0 \leq t \leq T$ that $\mathbb{P}^{\alpha^n}[t < \tau] \rightarrow \mathbb{P}^\alpha[t<\tau] > 0$ and $d_{LC}(\mathcal{L}_{\mathbb{P}^{\alpha^n}}(X_{\cdot \wedge t} \vert t<\tau) , \mathcal{L}_{\mathbb{P}^\alpha}(X_{\cdot \wedge t}\vert t<\tau)) \rightarrow 0$. Note that this implies that the $\mathbb{P}^{\alpha^n}[t<\tau]$ are bounded away from $0$ in $n$. Also, note that once $b$ is independent of $\mu$ and $p$, we have $\mathbb{P}^\alpha_{\vert \mathcal{F}_{T \wedge \tau^n}} = \mathbb{P}^{\alpha^n}_{\vert \mathcal{F}_{T\wedge \tau^n}}$. Then
    \begin{align*}
        \vert J(\alpha) - J(\alpha^n)\vert & \leq \mathbb{E}^\alpha\Big[\mathbf{1}_{\lbrace T < \tau \rbrace} \Big\vert g(X, \mathcal{L}_{\mathbb{P}^\alpha}(X \vert T<\tau), \mathbb{P}^\alpha[T < \tau]) - g(X, \mathcal{L}_{\mathbb{P}^{\alpha^n}}(X \vert T<\tau), \mathbb{P}^{\alpha^n}[T<\tau]) \Big\vert \Big]\\
        &\quad + \Big\vert(\mathbb{E}^{\alpha} - \mathbb{E}^{\alpha^n})\big[\mathbf{1}_{\lbrace T < \tau \rbrace}g(X, \mathcal{L}_{\mathbb{P}^{\alpha^n}}(X\vert T<\tau), \mathbb{P}^{\alpha^n}[T<\tau]) \big] \Big\vert\\
        & \quad + \mathbb{E}^\alpha\bigg[\int_0^{T\wedge \tau} \vert f(s, X_{\cdot \wedge s},\alpha_s, \mathcal{L}_{\mathbb{P}^\alpha}(X_{\cdot \wedge s}\vert s<\tau), \mathbb{P}^\alpha[s<\tau])\\ 
        &\qquad \qquad \qquad\qquad  - f(s, X_{\cdot \wedge s}, \alpha_s, \mathcal{L}_{\mathbb{P}^{\alpha^n}}(X_{\cdot \wedge s}\vert s<\tau), \mathbb{P}^{\alpha^n}[s<\tau])\vert ds\bigg]\\
        &\quad + \mathbb{E}^\alpha\bigg[\int_{T \wedge\tau^n}^{T \wedge \tau} \vert f(s, X_{\cdot \wedge s} , \alpha_s,\mathcal{L}_{\mathbb{P}^{\alpha^n}}(X_{\cdot \wedge s}\vert s<\tau), \mathbb{P}^{\alpha^n}[s<\tau])\vert ds\bigg] \\
        &\quad + \mathbb{E}^{\alpha^n}\bigg[\int_{T \wedge \tau^n}^{T \wedge \tau} \vert f(s, X_{\cdot \wedge s}, 0, \mathcal{L}_{\mathbb{P}^{\alpha^n}}(X_{\cdot \wedge s}\vert s<\tau), \mathbb{P}^{\alpha^n}[s<\tau]) \vert ds\bigg].
    \end{align*}
    Since under assumption \ref{asmp: ncs}, $f$ admits quadratic growth in $a$, we can see that by the dominated convergence theorem, this goes to zero.\par
    For the second statement, consider any $\alpha \in \mathbb{A}_{BMO}$. Now, since $\int_0^{\cdot} \beta (s, X_{\cdot \wedge s}, \alpha_s) dW^\alpha_s$ is a $\mathbb{P}^\alpha$-BMO martingale, $\int_0^\cdot \beta(s, X_{\cdot \wedge s}, \alpha^n_s) - \beta(s, X_{\cdot \wedge s},\alpha_s)dW^{\alpha}$ must also be a $\mathbb{P}^\alpha$-BMO martingale. Particularly, by \cite[Theorem 3.1]{Kazamaki}, there is some $r>1$ such that the $\mathbb{E}^\alpha[(\frac{d\mathbb{P}^{\alpha^n}}{d\mathbb{P}^\alpha})^r]$ are finite and bounded over $n$. 
\end{proof}

\subsection{Proof of the necessary condition}

We now turn our attention to the proof of the necessary condition of optimality.
The proof will be derived from several intermediate results. 
Along the way, we will show existence of the adjoint equation used in the statement.

\subsubsection{The variation process}
As in the standard approach to the Pontryagin maximum principle, we will derive the variation process.
In contrast to the strong formulation in which it represents the derivative of the state process; in the current setup the variation process will be obtained as the derivative of the density of the measure change. 
In the following, we fix a control $\alpha \in \mathbb{A}_{\text{BMO}}$ and a progressively measurable $\mathbb{R}^k$-valued perturbation $\eta$ so that for all sufficiently small $\epsilon$, we have $\alpha^\epsilon = \alpha + \epsilon \eta \in \mathbb{A}_{\text{BMO}}$. Note that we must have $\Vert\int_0^\cdot \Vert \eta_s \Vert^2 ds \Vert_{\text{BMO}} < \infty$ and by \cite[Theorem 3.6]{Kazamaki}, also $\Vert\int_0^\cdot \eta_s dW^\alpha_s \Vert_{\mathbb{P}^\alpha\text{-BMO}} < \infty$ so that $\frac{d\mathbb{P}^{\alpha^\epsilon}}{d\mathbb{P}^\alpha}$ is a true $\mathbb{P}^\alpha$ martingale. In particular, for any $r > 1$, by \cite[Theorem 3.1]{Kazamaki}, there exist $\epsilon_r, E_r > 0$ such that for any $\epsilon < \epsilon_r$, we have $\mathbb{E}^\alpha[(\frac{d\mathbb{P}^{\alpha^\epsilon}}{d\mathbb{P}^\alpha})^r] \leq E_r < \infty$. Also, note that for any $t\geq 0$, the conditional laws $\mathbb{P}^\alpha_{\vert \mathcal{F}_t}[\cdot \vert t <\tau]$ and $\mathbb{P}^{\alpha^\epsilon}_{\vert \mathcal{F}_t}[\cdot \vert t<\tau]$ are still equivalent with density $\frac{d\mathbb{P}^{\alpha^\epsilon}_{\vert \mathcal{F}_t}[\cdot \vert t<\tau]}{d\mathbb{P}^\alpha_{\vert \mathcal{F}_t}[\cdot \vert t<\tau]} = \frac{\mathbb{P}^\alpha[t < \tau]}{\mathbb{P}^{\alpha^\epsilon}[t<\tau]} \frac{d\mathbb{P}^{\alpha^\epsilon}_{\vert \mathcal{F}_t}}{d\mathbb{P}^\alpha_{\vert \mathcal{F}_t}}$ restricted to $\lbrace t < \tau \rbrace$.\par
We define $\Lambda^{\alpha, \eta}$ as the solution of the following linear McKean-Vlasov SDE:
$$\Lambda^{\alpha, \eta}_t = \int_0^{t \wedge \tau} \beta_a(\Theta^\alpha_s) \eta_s + \tilde{\mathbb{E}}^\alpha\Big[\frac{\delta \beta}{\delta m}(\Theta^\alpha_s, \tilde{X}_{\cdot \wedge s})\tilde{\Lambda}^{\alpha, \eta}_s\vert s<\tilde{\tau} \Big] + \beta_p(\Theta^\alpha_s) \mathbb{E}^\alpha\big[\mathbf{1}_{\lbrace s < \tau\rbrace} \Lambda^{\alpha, \eta}_s \big] dW^\alpha_s.$$
Since $\mathbb{P}^\alpha[s < \tau]$ is uniformly bounded away from zero, it follows by our bounds on the derivatives, that this McKean-Vlasov SDE admits a unique $\mathbb{P}^\alpha$ square integrable strong solution. In particular, $\Lambda^{\alpha, \eta}$ is a true $\mathbb{P}^\alpha$-martingale. 
In fact, since $\tilde{\mathbb{E}}^\alpha[\frac{\delta \beta}{\delta m}(\Theta^\alpha_s, \tilde{X}_{\cdot \wedge s})\tilde{\Lambda}^{\alpha, \eta}_s\vert s<\tilde{\tau}] + \beta_p(\Theta^\alpha_s) \mathbb{E}^\alpha[\mathbf{1}_{\lbrace s < \tau\rbrace} \Lambda^{\alpha, \eta}_s]$ can be a.e. bounded by a deterministic constant, $\Lambda^{\alpha, \eta}$ must actually be a $\mathbb{P}^\alpha$-BMO martingale.\par
In the following, we let $\Gamma$ be a common upper bound for $\Vert \alpha \Vert_{2, \infty}$ and $\Vert \alpha^\epsilon \Vert_{2, \infty}$. To simplify notation, in the following, we will also write $M$ instead of $M(p_\Gamma)$.
\begin{proposition}\label{prop: var prc}
    Under Assumption \ref{asmp: beta} and \ref{asmp: ncs}, for any $r\geq 1$, for $\epsilon \searrow 0$, we have
    $$\mathbb{E}^\alpha\bigg[\bigg\{ \frac{1}{\epsilon}\Big(\frac{d\mathbb{P}^{\alpha^{\epsilon}}_{\vert\mathcal{F}_{t\wedge \tau}}}{d\mathbb{P}^\alpha_{\vert\mathcal{F}_{t\wedge \tau}}} - 1\Big) - \Lambda^{\alpha, \eta}_t\bigg\}^r \Big\vert t < \tau\bigg] \rightarrow 0$$
    and further,
    $$\mathbb{E}^\alpha\bigg[\bigg\{\frac{1}{\epsilon}\Big(\frac{\mathbb{P}^\alpha[t<\tau]}{\mathbb{P}^{\alpha^{\epsilon}}[t<\tau]} \frac{d\mathbb{P}^{\alpha^{\epsilon}}_{\vert\mathcal{F}_{t\wedge \tau}}}{d\mathbb{P}^\alpha_{\vert\mathcal{F}_{t\wedge \tau}}} - 1\Big) - \Big(\Lambda^{\alpha, \eta}_t - \mathbb{E}^\alpha[\Lambda^{\alpha, \eta}_t \vert t < \tau]\Big)\bigg\}^r \Big\vert t < \tau \bigg] \rightarrow 0.$$
\end{proposition}
\begin{proof}
    For ease of notation, we assume $r \geq 2$. 
    The case $1\le r<2$ immediately follow by Hölder's inequality. 
    Further, it suffices to only consider $\epsilon < \epsilon_r$. 
    Let us write $\mathcal{E}^\epsilon_t := \frac{d\mathbb{P}^{\alpha^\epsilon}_{\vert \mathcal{F}_{t \wedge \tau}}}{d\mathbb{P}^\alpha_{\vert \mathcal{F}_{t \wedge \tau}}}$ and $L^\epsilon_t := \frac{1}{\epsilon}(\frac{d\mathbb{P}^{\alpha^\epsilon}_{\vert \mathcal{F}_{t \wedge \tau}}}{d\mathbb{P}^\alpha_{\vert \mathcal{F}_{t \wedge \tau}}} - 1)$. Then we have
    \begin{align*}
        \mathbb{E}^\alpha[(L^\epsilon_t - \Lambda^{\alpha, \eta}_t)^r\vert t<\tau]
        &\leq 2^{r - 1} \mathbb{E}^\alpha\bigg[\bigg(\int_0^{t\wedge \tau} (\mathcal{E}^\epsilon_s - 1)\frac{\beta(\Theta^{\alpha^\epsilon}_s) - \beta(\Theta^\alpha_s)}{\epsilon}dW^\alpha_s\bigg)^r \vert t < \tau \bigg]\\
        &\quad + 2^{r - 1} \mathbb{E}^\alpha\bigg[\bigg(\int_0^{t\wedge \tau} \frac{\beta(\Theta^{\alpha^\epsilon}_s) - \beta(\Theta^\alpha_s)}{\epsilon}dW^\alpha_s - \Lambda^{\alpha, \eta}_{t \wedge \tau}\bigg)^r\vert t< \tau\bigg]\\
        &\leq \frac{2^{r-1}}{p_\Gamma}\bigg\{\mathbb{E}^\alpha\bigg[\bigg(\int_0^{t\wedge \tau} (\mathcal{E}^\epsilon_s - 1)\frac{\beta(\Theta^{\alpha^\epsilon}_s) - \beta(\Theta^\alpha_s)}{\epsilon}dW^\alpha_s\bigg)^r\bigg]\\ 
        &\quad\qquad \qquad + \mathbb{E}^\alpha\bigg[\bigg(\int_0^{t\wedge \tau} \frac{\beta(\Theta^{\alpha^\epsilon}_s) - \beta(\Theta^\alpha_s)}{\epsilon}dW^\alpha_s - \Lambda^{\alpha, \eta}_{t \wedge \tau}\bigg)^r\bigg]\bigg\}\\
        &=: \frac{2^{r-1}}{p_\Gamma}(I^1_t + I^2_t).
    \end{align*}
    In the following, $c_r$ denotes the constant in the Burkholder-Davis-Gundy inequality. 
    By the Émery inequality \cite[Theorem A.8.15]{CE}, there exists $c'_r > 0$ such that
    \begin{eqnarray*}
        I^1_t &\leq& c_rc'_r \mathbb{E}^\alpha\bigg[\sup_{s \in [0, t]} (\mathcal{E}^\epsilon_{s\wedge \tau} - 1)^r\bigg] \frac{1}{\epsilon^r} \Big\Vert \int_0^{\cdot \wedge t\wedge\tau} \beta(\Theta^{\alpha^\epsilon}_s) - \beta(\Theta^\alpha_s) dW^\alpha_s \Big\Vert_{\mathbb{P}^\alpha\text{-BMO}}^r\\
        &\leq& \frac{c_rc'_rr^rL^r3^{\frac{r}{2}}}{(r-1)^r}\mathbb{E}^\alpha\big[(\mathcal{E}^\epsilon_t -1)^r\big]\Big(\Vert \eta \Vert_{\mathbb{P}^\alpha\text{-BMO}}^2 + C^\Gamma_Pt\Vert \eta \Vert_{\text{BMO}}^2\Big)^{\frac{r}{2}}
    \end{eqnarray*}
    where in the last inequality, have used Proposition \ref{prop: stab P}, as well as Doob's inequality since $\vert \mathcal{E}^\epsilon_\cdot - 1\vert^r$ is a submartingale. For sufficiently small $\epsilon$, the $(\mathcal{E}^\epsilon_t - 1)^r$ are uniformly integrable with respect to $\mathbb{P}^\alpha$, therefore $I^1_t$ converges to zero.\par
    As we have assumed $\beta$ to be jointly differentiable, the second term can be bounded by
    \begin{align*}
        \frac{1}{c_r}I^2_t &\leq 3^{r-1}\mathbb{E}\bigg[\bigg(\int_0^{t \wedge \tau} \Big\Vert\int_0^1 \beta_a(\Theta^{\alpha^\epsilon, \lambda}_s) \eta_s  - \beta_a(\Theta^\alpha_s) \eta_sd\lambda\Big\Vert^2ds\bigg)^\frac{r}{2} \bigg]\\
        &\quad+ 3^{r-1}\mathbb{E}^\alpha\bigg[\bigg(\int_0^{t\wedge \tau}\Big\Vert\tilde{\mathbb{E}}^\alpha\bigg[\frac{1}{\epsilon}\Big(\frac{\mathbb{P}^\alpha[s < \tau]}{\mathbb{P}^{\alpha^\epsilon}[s < \tau]} \tilde{D}^\epsilon_s - 1 \Big)\int_0^1\frac{\delta \beta}{\delta m}(\Theta^{\alpha^\epsilon, \lambda}_s, \tilde{X}_{\cdot \wedge s})d\lambda\\
        &\qquad\qquad \qquad \qquad - \tilde{\Lambda}^{\alpha, \eta}_s \frac{\delta \beta}{\delta m}(\Theta^\alpha_s, \tilde{X}_{\cdot \wedge s}) \Big\vert s<\tilde{\tau}]\Big\Vert^2ds\bigg)^\frac{r}{2}\bigg]\\
        &\quad +  3^{r-1}\mathbb{E}^\alpha\bigg[\bigg(\int_0^{t\wedge \tau} \Big\Vert\frac{1}{\epsilon}\int_0^1 \beta_p(\Theta^{\alpha^\epsilon, \lambda}_s)d\lambda (\mathbb{P}^{\alpha^\epsilon} - \mathbb{P}^\alpha)[s < \tau]- \beta_p(\Theta^\alpha_s)\mathbb{E}^\alpha[\mathbf{1}_{\lbrace s<\tau \rbrace}\Lambda^{\alpha, \eta}_s]\Big\Vert^2 ds\bigg)^\frac{r}{2}\bigg] \\
        &=: 3^{r-1}(I^{2, 1}_t + I^{2, 2}_t + I^{2, 3}_t)
    \end{align*}
    where we wrote
    \begin{align*}
        \Theta^{\alpha^\epsilon, \lambda}_t & := \lambda \Theta^{\alpha^\epsilon}_t + (1 - \lambda) \Theta^\alpha_t\\
        & = \Big(t, X_{\cdot \wedge t}, \lambda \alpha^\epsilon_t + ( 1-\lambda) \alpha_t, \lambda \mathcal{L}_{\mathbb{P}^{\alpha^\epsilon}}(X_{\cdot \wedge t}\vert t < \tau) + (1-\lambda)\mathcal{L}_{\mathbb{P}^\alpha}(X_{\cdot \wedge t}\vert t<\tau),\\ 
        &\qquad  \lambda \mathbb{P}^{\alpha^\epsilon}[s<\tau] + (1-\lambda)\mathbb{P}^\alpha[s<\tau]\Big).
    \end{align*}
    Clearly, $I^{2, 1}_t$ goes to zero by the dominated convergence theorem as $\beta_a$ is bounded and  continuous. For $I^{2, 2}_t$, first note that
    \begin{align*}
        %&\mathbb{E}^\alpha\bigg[\bigg(\int_0^{t\wedge \tau}\Big\Vert\tilde{\mathbb{E}}^\alpha\Big[\frac{1}{\epsilon}\Big(\frac{\mathbb{P}^\alpha[s < \tau]}{\mathbb{P}^{\alpha^\epsilon}[s < \tau]} \tilde{D}^\epsilon_s - 1\Big)\int_0^1\frac{\delta \beta}{\delta m}(\Theta^{\alpha^\epsilon, \lambda}_s, \tilde{X}_{\cdot \wedge s})d\lambda - \tilde{\Lambda}^{\alpha, \eta}_s \frac{\delta \beta}{\delta m}(\Theta^\alpha_s, \tilde{X}_{\cdot \wedge s}) \vert s<\tilde{\tau}]\Vert^2ds\bigg)^\frac{r}{2}\bigg]\\
        I^{2,2}
        &\leq 3^{r-1}\mathbb{E}^\alpha\bigg[\bigg(\int_0^{t\wedge \tau}\Big\Vert\tilde{\mathbb{E}}^\alpha\Big[\frac{1}{\epsilon}\Big(\frac{\mathbb{P}^\alpha[s < \tau]}{\mathbb{P}^{\alpha^\epsilon}[s < \tau]} \tilde{D}^\epsilon_s - 1\Big)\Big\{\int_0^1\frac{\delta \beta}{\delta m}(\Theta^{\alpha^\epsilon, \lambda}_s, \tilde{X}_{\cdot \wedge s})d\lambda -  \frac{\delta \beta}{\delta m}(\Theta^\alpha_s, \tilde{X}_{\cdot \wedge s}) \Big\}\Big\vert s<\tilde{\tau}\Big]\Big\Vert^2ds\bigg)^\frac{r}{2}\bigg]\\
        &\quad + 3^{r-1}\mathbb{E}^\alpha\bigg[\bigg(\int_0^{t\wedge \tau}\Big\Vert\tilde{\mathbb{E}}^\alpha\Big[\frac{1}{\epsilon}\Big(\frac{\mathbb{P}^\alpha[s < \tau]}{\mathbb{P}^{\alpha^\epsilon}[s < \tau]} - 1\Big)\tilde{D}^\epsilon_s  \frac{\delta \beta}{\delta m}(\Theta^\alpha_s, \tilde{X}_{\cdot \wedge s})\Big\vert s<\tilde{\tau}\Big]\Big\Vert^2ds\bigg)^\frac{r}{2}\bigg]\\
        &\quad + 3^{r-1}\mathbb{E}^\alpha\bigg[\bigg(\int_0^{t\wedge \tau}\Big\Vert\tilde{\mathbb{E}}^\alpha\Big[(\tilde{L}^\epsilon_s - \tilde{\Lambda}^{\alpha, \eta}_s)  \frac{\delta \beta}{\delta m}(\Theta^\alpha_s, \tilde{X}_{\cdot \wedge s})\Big\vert s<\tilde{\tau}\Big]\Big\Vert^2ds\bigg)^\frac{r}{2}\bigg]\\ 
        & \leq 3^{r-1}(I^{2, 2, 1}_t + I^{2, 2, 2}_t + I^{2, 2, 3}_t).
    \end{align*}
    For the first term, we fix some $\rho > 2$. Then,
    \begin{align*}
        I^{2, 2, 1}_t &\leq \mathbb{E}^\alpha\bigg[\bigg(\int_0^{t\wedge \tau}\bigg(\int_0^1 \tilde{\mathbb{E}}^\alpha\Big[\frac{1}{\epsilon}\Big\vert\frac{\mathbb{P}^\alpha[s<\tau]}{\mathbb{P}^{\alpha^\epsilon}[s<\tau]}\tilde{D}^\epsilon_s-1\Big\vert\Big\Vert\frac{\delta \beta}{\delta m}(\Theta^{\alpha^\epsilon, \lambda}_s, \tilde{X}_{\cdot \wedge s}) -  \frac{\delta \beta}{\delta m}(\Theta^\alpha_s, \tilde{X}_{\cdot \wedge s})\Big\Vert \big\vert s<\tilde{\tau}\Big] d\lambda\bigg)^2ds\bigg)^{\frac{r}{2}}\bigg]\\
        &\leq \mathbb{E}^\alpha\bigg[\bigg(\int_0^{t\wedge \tau}\bigg(\int_0^1 \tilde{\mathbb{E}}^\alpha\Big[\frac{1}{\epsilon^{\rho}}\Big\vert\frac{\mathbb{P}^\alpha[s<\tau]}{\mathbb{P}^{\alpha^\epsilon}[s<\tau]}\tilde{D}^\epsilon_s-1\Big\vert^{\rho} \big\vert s < \tilde{\tau}\Big]^{\frac{1}{\rho}}\times\\ 
        &\qquad \quad \quad\tilde{\mathbb{E}}^\alpha\bigg[\Big\Vert\frac{\delta \beta}{\delta m}(\Theta^{\alpha^\epsilon, \lambda}_s, \tilde{X}_{\cdot \wedge s}) - \frac{\delta \beta}{\delta m}(\Theta^\alpha_s, \tilde{X}_{\cdot \wedge s})\Big\Vert^{\frac{\rho}{\rho-1}}\vert s<\tilde{\tau}\Big]^{\frac{\rho-1}{\rho}} d\lambda\bigg)^2ds\bigg)^{\frac{r}{2}}\bigg].
    \end{align*}
    Using the Burkholder Davis Gundy inequality, for some $c_\rho>0$, using our previous discussion, 
    \begin{align*}
        \tilde{\mathbb{E}}^\alpha\Big[\frac{1}{\epsilon^{\rho}}\Big\vert&\frac{\mathbb{P}^\alpha[s<\tau]}{\mathbb{P}^{\alpha^\epsilon}[s<\tau]}\tilde{D}^\epsilon_s -1\Big\vert^{\rho} \big\vert s < \tilde{\tau}\Big]\\ 
        & \leq \frac{2^{\rho - 1}\tilde{\mathbb{E}}^\alpha[(\tilde{D}^\epsilon_s)^{\rho}]}{\epsilon^{\rho}}\Big\vert\frac{1}{\mathbb{P}^{\alpha^\epsilon}[s<\tau]} - \frac{1}{\mathbb{P}^\alpha[s<\tau]}\Big\vert^{\rho} + \frac{2^{\rho - 1}}{\mathbb{P}^{\alpha}[s<\tau]^{\rho}}\tilde{\mathbb{E}}^\alpha\Big[\frac{\vert \tilde{D}^\epsilon_s-1\vert^{\rho}}{\epsilon^{\rho}}\Big]\\
        &\leq \frac{2^{\rho - 1}E_\rho}{p_\Gamma^{2\rho}}\frac{\vert\mathbb{P}^{\alpha^\epsilon}[s<\tau]-\mathbb{P}^\alpha[s<\tau]\vert^{\rho}}{\epsilon^{\rho}}\\ 
        &\qquad  + \frac{2^{\rho - 1}c_\rho}{p_\Gamma^{\rho}}\tilde{\mathbb{E}}^\alpha\bigg[\bigg(\int_0^s (\tilde{D}^\epsilon_u)^2 \frac{\Vert \beta(\widetilde \Theta^{\alpha^\epsilon}_u) - \beta(\widetilde\Theta^\alpha_u)\Vert^{2}}{\epsilon^2}du\bigg)^{\frac{\rho}{2}}\bigg]\\
        &\leq \frac{2^{\rho - 1}E_\rho(C_P^\Gamma)^{\frac{\rho}{2}}}{p_\Gamma^{2\rho}} \Vert \eta \Vert_{\mathrm{BMO}}^\rho + \frac{2^{2\rho - 2}3^{\frac{\rho}{2}}c_\rho c'_\rho \rho^\rho L^\rho}{p_\Gamma^\rho(\rho - 1)^\rho}(1 + E_\rho)\big(\Vert \eta \Vert_{\mathbb{P}^\alpha\mathrm{-BMO}}^2 + C^\Gamma_Pt\Vert \eta \Vert_{\text{BMO}}^2\big)^{\frac{\rho}{2}}
    \end{align*}
    so that this first factor is uniformly bounded for $\epsilon < \epsilon_\rho$ small enough that $\mathbb{E}^\alpha[(\mathcal{E}^\epsilon_s)^\rho] \leq E_\rho$. Moreover, we note that by our continuity assumption, for any $\lambda$, $\frac{\delta \beta}{\delta m}(\Theta^{\alpha^\epsilon, \lambda}_s, \tilde{X}_{\cdot\wedge s})$ converges to $\frac{\delta \beta}{\delta m}(\Theta^\alpha_s, \tilde{X}_{\cdot \wedge s})$ almost surely. 
    As $\frac{\rho}{\rho-1} < 2$, for any $\lambda \in[0, 1)$ and $\mathbb{P}^\alpha$ a.s., the family $\Vert \frac{\delta \beta}{\delta m}(\Theta^{\alpha^\epsilon, \lambda}_s, \tilde{X}_{\cdot \wedge s})\Vert^{\frac{\rho}{\rho-1}}$ is $\tilde{\mathbb{P}}^\alpha$-uniformly integrable and thus $\tilde{\mathbb{E}}^\alpha[\Vert\frac{\delta \beta}{\delta m}(\Theta^{\alpha^\epsilon, \lambda}_s, \tilde{X}_{\cdot \wedge s})-\frac{\delta \beta}{\delta m}(\Theta^\alpha_s, \tilde{X}_{\cdot \wedge s})\Vert^{\frac{\rho}{\rho-1}}\vert s<\tilde{\tau}] \rightarrow 0$. 
    Furthermore, we have $\mathbb{P}^\alpha$ a.s.\ and for any $\lambda \in [0, 1)$ that
    \begin{eqnarray*}
        \tilde{\mathbb{E}}^\alpha\Big[\Big\Vert\frac{\delta \beta}{\delta m}(\Theta^{\alpha^\epsilon, \lambda}_s, \tilde{X}_{\cdot \wedge s})-\frac{\delta \beta}{\delta m}(\Theta^\alpha_s, \tilde{X}_{\cdot \wedge s})\Big\Vert^{\frac{\rho}{\rho-1}}\vert s<\tilde{\tau}\Big]^{\frac{\rho-1}{\rho}} \leq \frac{L}{4}\Big(1 + \frac{1}{\sqrt{1 - \lambda}}\Big).
    \end{eqnarray*}
    As $\mathbb{E}^\alpha[\big(\int_0^{t\wedge\tau}(\int_0^1 1 + \frac{1}{\sqrt{1 - \lambda}} d\lambda)^2ds\big)^{\frac{r}{2}}]<\infty$, this shows by the dominated convergence theorem that $I^{2, 2, 1}_t$ converges to zero.\par

    Regarding $I^{2, 2, 2}_t$, note again that $\frac{1}{\epsilon}(\frac{\mathbb{P}^\alpha[s < \tau]}{\mathbb{P}^{\alpha^\epsilon}[s < \tau]} - 1)$ is uniformly bounded by a constant over all $\epsilon$. 
    Further, recall that by assumption $\tilde{\mathbb{E}}^\alpha[\frac{\delta\beta}{\delta m}(\Theta^\alpha_s, \tilde{X}_{\cdot \wedge s})\vert s < \tilde{\tau}] = 0$. Thus,
    \begin{eqnarray*}
        \mathbb{E}^\alpha\bigg[\bigg( \int_0^{t\wedge \tau}\Big\Vert\tilde{\mathbb{E}}^\alpha\Big[\tilde{D}^\epsilon_s  \frac{\delta \beta}{\delta m}(\Theta^\alpha_s, \tilde{X}_{\cdot \wedge s})\vert s<\tilde{\tau}\Big]\Big\Vert^2ds\bigg)^\frac{r}{2}\bigg] = \mathbb{E}^\alpha\bigg[\int_0^{t\wedge \tau}\Big\Vert\tilde{\mathbb{E}}^\alpha\Big[(\tilde{D}^\epsilon_s - 1)  \frac{\delta \beta}{\delta m}(\Theta^\alpha_s, \tilde{X}_{\cdot \wedge s})\vert s<\tilde{\tau}\Big]\Big\Vert^2ds\bigg)^\frac{r}{2}\bigg]
    \end{eqnarray*}
    which again converges to zero as the integrand is uniformly integrable. Thus, $I^{2, 2, 2}_t$ also converges to zero. The last term $I^{2, 2, 3}_t$ can be bounded by $\frac{L^r T^{\frac{r - 2}{r}}}{4^r} \int_0^t \mathbb{E}^\alpha[(L^\epsilon_s -\Lambda^{\alpha, \eta}_s)^r \vert s< \tau] ds$.\par
    
    Finally, we can write
    \begin{align*}
        I^{2, 3}_t &\leq 2^{r-1} \mathbb{E}^\alpha\bigg[\bigg(\int_0^{t\wedge \tau} \frac{1}{\epsilon^2} (\mathbb{P}^{\alpha^\epsilon} - \mathbb{P}^\alpha)[s < \tau]^2\Big\Vert\int_0^1 \beta_p(\Theta^{\alpha^\epsilon, \lambda}_s)d\lambda- \beta_p(\Theta^\alpha_s)\Big\Vert^2ds\bigg)^\frac{r}{2}\bigg]\\
        &\quad +2^{r-1} \mathbb{E}^\alpha\bigg[\bigg(\int_0^{t\wedge \tau} (\frac{1}{\epsilon} (\mathbb{P}^{\alpha^\epsilon} - \mathbb{P}^\alpha)[s < \tau]- \mathbb{E}^\alpha[\mathbf{1}_{\lbrace s<\tau \rbrace}\Lambda^{\alpha, \eta}_s])^2 \Vert\beta_p(\Theta^\alpha_s)\Vert^2 ds\bigg)^\frac{r}{2}\bigg]\\ 
        & =: 2^{r-1}(I^{2, 3, 1}_t + I^{2, 3, 2}_t).
    \end{align*}
    As $\frac{1}{\epsilon^2}(\mathbb{P}^{\alpha^\epsilon} - \mathbb{P}^\alpha)[s < \tau]^2$ is bounded by a constant, we can again see that $I^{2, 3, 1}_t$ goes to zero. For the second term, we can see that
    $$(\frac{1}{\epsilon} (\mathbb{P}^{\alpha^\epsilon} - \mathbb{P}^\alpha)[s < \tau]- \mathbb{E}^\alpha[\mathbf{1}_{\lbrace s<\tau \rbrace}\Lambda^{\alpha, \eta}_s])^2\leq \mathbb{E}^\alpha[\mathbf{1}_{\lbrace s<\tau \rbrace}(\frac{1}{\epsilon} (\mathcal{E}^\epsilon_s-1) - \Lambda^{\alpha, \eta}_s)^2] \leq \mathbb{E}^\alpha[(L^\epsilon_t -\Lambda^{\alpha, \eta}_s)^2\vert s < \tau]$$
    and thus $I^{2, 3, 2}_t \leq L^rT^\frac{r-2}{r} \int_0^t \mathbb{E}^\alpha[(L^\epsilon_s -\Lambda^{\alpha, \eta}_s)^r\vert s < \tau ]ds$.\par
    Putting everything together, we can thus see
    \begin{align*}
        \mathbb{E}^\alpha\big[ (L^\epsilon_t - \Lambda^{\alpha, \eta}_t)^r\vert t<\tau \big] &
        \leq \frac{2^{r-1}}{p_\Gamma} \Big(I^1_t + c_r3^{r-1} \big(I^{2,1}_t + 3^{r-1} I^{2, 2, 1}_t + 3^{r-1}I^{2, 2, 2}_t + 2^{r-1}I^{2, 3, 1}_t\big)\Big)\\
        &\quad + \frac{(\frac{1}{4} (\frac{9}{2})^{r-1}+ 12^{r-1})c_rL^rT^\frac{r-2}{2}}{p_\Gamma} \int_0^t \mathbb{E}^\alpha[(L^\epsilon_s - \Lambda^{\alpha, \eta}_s)^r\vert s<\tau] ds.
    \end{align*}
    Since $I^1_t, I^{2, 1}_t, I^{2, 2, 1}_t, I^{2, 2, 2}_t, I^{2, 3, 1}_t$ are non-decreasing in $t$ but converge to $0$ for $\epsilon \searrow 0$, by Grönwall's lemma, this shows our first convergence statement.
    \par

    For the second statement, we can see that
    \begin{align*}
        &\mathbb{E}^\alpha\bigg[ \bigg( \frac{1}{\epsilon}\Big(\frac{\mathbb{P}^\alpha[t<\tau]}{\mathbb{P}^{\alpha^{\epsilon}}[t<\tau]} \frac{d\mathbb{P}^{\alpha^{\epsilon}}_{\vert\mathcal{F}_{t\wedge \tau}}}{d\mathbb{P}^\alpha_{\vert\mathcal{F}_{t\wedge \tau}}} - 1\Big) - \Big(\Lambda^{\alpha, \eta}_t - \mathbb{E}^\alpha[\Lambda^{\alpha, \eta}_t \big\vert t < \tau] \Big) \bigg)^r\big\vert t < \tau \bigg] \\
        &\quad \leq 3^{r-1}\mathbb{E}^\alpha\bigg[\bigg(\frac{1}{\epsilon}\Big(\frac{d\mathbb{P}^{\alpha^{\epsilon}}_{\vert\mathcal{F}_{t\wedge \tau}}}{d\mathbb{P}^\alpha_{\vert\mathcal{F}_{t\wedge \tau}}} - 1\Big) - \Lambda^{\alpha, \eta}_t\bigg)^r  \big\vert t < \tau \bigg] + \frac{3^{r-1}}{\epsilon^r}\Big(\frac{\mathbb{P}^\alpha[t<\tau]}{\mathbb{P}^{\alpha^{\epsilon}}[t<\tau]} - 1\Big)^r\mathbb{E}^\alpha\bigg[\Big(\frac{d\mathbb{P}^{\alpha^{\epsilon}}_{\vert\mathcal{F}_{t\wedge \tau}}}{d\mathbb{P}^\alpha_{\vert\mathcal{F}_{t\wedge \tau}}}-1\Big)^r\big\vert t < \tau\bigg]\\
        &\quad \quad +3^{r-1}\bigg(\frac{1}{\epsilon}\Big(\frac{\mathbb{P}^\alpha[t<\tau]}{\mathbb{P}^{\alpha^{\epsilon}}[t<\tau]} - 1\Big) + \mathbb{E}^\alpha[\Lambda^{\alpha, \eta}_t \big\vert t < \tau] \bigg)^r.
    \end{align*}
    Using our estimates from above, we already know that the first two terms go to zero. The last term can be rewritten as
    \begin{align*}
        \mathbb{E}^\alpha\bigg[\frac{1}{\epsilon}\frac{\mathbb{P}^\alpha[t<\tau]}{\mathbb{P}^{\alpha^{\epsilon}}[t<\tau]}\Big(1 - \frac{d\mathbb{P}^{\alpha^{\epsilon}}_{\vert\mathcal{F}_{t\wedge \tau}}}{d\mathbb{P}^\alpha_{\vert\mathcal{F}_{t\wedge \tau}}}\Big) + \Lambda^{\alpha, \eta}_t \big\vert t< \tau\bigg]^r 
        &\leq 2^{r-1}\bigg(\frac{\mathbb{P}^\alpha[t<\tau]}{\mathbb{P}^{\alpha^{\epsilon}}[t<\tau]}-1\bigg)^r\mathbb{E}^\alpha\bigg[\frac{1}{\epsilon^r}\Big(1 - \frac{d\mathbb{P}^{\alpha^{\epsilon}}_{\vert\mathcal{F}_{t\wedge \tau}}}{d\mathbb{P}^\alpha_{\vert\mathcal{F}_{t\wedge \tau}}}\Big)^r\bigg]\\ 
        &\quad + 2^{r-1}\mathbb{E}^\alpha\bigg[\bigg(\frac{1}{\epsilon}\Big(\frac{d\mathbb{P}^{\alpha^{\epsilon}}_{\vert\mathcal{F}_{t\wedge \tau}}}{d\mathbb{P}^\alpha_{\vert\mathcal{F}_{t\wedge \tau}}} - 1\Big) - \Lambda^{\alpha, \eta}_t\bigg)^r\big\vert t < \tau\bigg]
    \end{align*}
    which we have already shown to go to zero.
\end{proof}
\subsubsection{Well-posedness of the adjoint equation}
Before finishing the proof of the necessary condition for optimality, let us justify well-posedness of the adjoint equation \eqref{eq: adj eq} for any $\alpha \in \mathbb{A}_{\text{BMO}}$. 
We do so following a similar approach as in \cite[Proposition 5.2]{PossamaiTangpi} and will thus only sketch the proof.\par
It is more convenient to consider this BSDE under the probability measure $\mathbb{P}^\alpha$.
Given $\alpha \in \mathbb{A}_{\text{BMO}}$, since $\int_0^\cdot \beta(\Theta^\alpha_s)dW_s$ is a BMO martingale, it follows by \cite[Theorem 3.1]{Kazamaki}, that there exists some $\rho_\alpha > 1$ such that $\mathbb{E}[(\frac{d\mathbb{P}^\alpha}{d\mathbb{P}})^{\rho_\alpha}] < \infty$. In particular, under Assumption \ref{asmp: ncs}, we must have $\mathbb{E}^\alpha[(\int_0^{T\wedge \tau} \vert f(s, X_{\cdot \wedge s}, 0, \mu^0, 1)ds)^2 + \mathbf{1}_{\lbrace T< \tau\rbrace}\vert g(X, \mu^0, 1)\vert^2] < \infty$. This way, we are able to apply the standard $\mathbb{L}^2$ fixed point argument below. It remains open whether the adjoint equation is solvable under less integrability, for instance whether arguments as in e.g.\ \cite{Briand03} can be adapted to such McKean-Vlasov BSDEs.
% First, note that in terms of $W$ and $\mathbb{P}^\alpha$, \eqref{eq: adj eq} can also be rewritten as
% \begin{eqnarray}\label{eq: adj eq P}
%     &&Y^\alpha_{t\wedge \tau} = \mathbf{1}_{\lbrace T < \tau \rbrace} (g(\Theta^\alpha_T) + \tilde{\mathbb{E}}^\alpha[\frac{\delta g}{\delta m}(\tilde{\Theta}^\alpha_T, X)\vert T < \tilde{\tau}] + \mathbb{E}^\alpha[\mathbf{1}_{\lbrace T < \tau \rbrace}g_p(\Theta^\alpha_T)])\\
%     &+& \int_{t\wedge \tau}^{T \wedge \tau} h(\Theta^\alpha_s, Z^\alpha_s) + \tilde{\mathbb{E}}^\alpha[\frac{\delta h}{\delta m}(\tilde{\Theta}^\alpha_s, \tilde{Z}^\alpha_s, X_{\cdot \wedge s})\vert s < \tilde{\tau}] + \mathbb{E}^\alpha[\mathbf{1}_{\lbrace s<\tau \rbrace} h_p(\Theta^\alpha_s, Z^\alpha_s)]ds - \int_{t\wedge\tau}^{T \wedge \tau} Z^\alpha_s dW_s\nonumber
% \end{eqnarray}
We are going to look for solutions within the space $\mathcal{M}_{1, \alpha}^2 \times \mathcal{M}_{d, \alpha}^2$ where
%$\mathcal{S}_\alpha^2$ consists of all adapted continuous processes $Y$ such that $\mathbb{E}^\alpha[\sup_{t \in [0, T]} \vert Y_t \vert^r] < \infty$ and
$\mathcal{M}_{n, \alpha}^2$ consists of all $\mathbb{R}^n$-valued progressively measurable processes $\phi$ such that $\mathbb{E}^\alpha[\int_0^T \Vert \phi_s \Vert^2 ds] < \infty$. Note that this BSDE admits a random but finite terminal time, and we will use the convention $Y_t = Y_{t \wedge \tau}$ and $Z_t \mathbf{1}_{t \geq \tau} = 0$.
\begin{proposition}\label{prop: wp adj}
    Under Assumptions \ref{asmp: beta} and \ref{asmp: ncs}, for any $\alpha \in \mathbb{A}_{\text{BMO}}$, the BSDE \eqref{eq: adj eq} admits a unique solution $(Y^\alpha, Z^\alpha)$ within $\mathcal{M}_{1, \alpha}^2 \times \mathcal{M}_{d, \alpha}^2$.
\end{proposition}
\begin{proof}
    Let us write $G := g(\Theta^\alpha_T) + \tilde{\mathbb{E}}^\alpha[\frac{\delta g}{\delta m}(\tilde{\Theta}^\alpha_T, X)\vert T < \tilde{\tau}] + \mathbb{E}^\alpha[\mathbf{1}_{\lbrace T < \tau \rbrace}g_p(\Theta^\alpha_T)]$ and $F_t := f(\Theta^\alpha_t) + \tilde{\mathbb{E}}^\alpha[\frac{\delta f}{\delta m}(\tilde{\Theta}^\alpha_t,  X_{\cdot \wedge t})\vert t < \tilde{\tau}] + \mathbb{E}^\alpha[\mathbf{1}_{\lbrace t<\tau \rbrace} f_p(\Theta^\alpha_t)]$. 
    With these notation, we can write \eqref{eq: adj eq} as
    $$
        Y_{t\wedge \tau} = \mathbf{1}_{\lbrace T < \tau \rbrace} G + \int_{t\wedge\tau}^{T\wedge\tau} F_s + \tilde{\mathbb{E}}^\alpha\Big[\frac{\delta \beta}{\delta m}(\tilde{\Theta}^\alpha_s, X_{\cdot \wedge s})^\top\tilde{Z}_s\vert s < \tilde{\tau}\Big] + \mathbb{E}^\alpha\big[ \mathbf{1}_{\lbrace s < \tau \rbrace} \beta_p(\Theta^\alpha_s)^\top Z_s \big]ds-\int_{t\wedge \tau}^{T\wedge\tau}Z_s dW^\alpha_s.$$
    By the discussion above and Assumption \ref{asmp: ncs}, we know that $\mathbb{E}^\alpha[(\int_0^{T \wedge \tau}\vert F_s\vert ds)^2 + \mathbf{1}_{\lbrace T < \tau \rbrace}\vert G \vert^2]<\infty$.
% Let us now, for each $n \geq 1$, define $G^n = \mathbf{1}_{\lbrace\vert G \vert \leq n\rbrace}$ and $F^n_t = \mathbf{1}_{\lbrace\vert F \vert \leq n\rbrace} F_t$ and consider the BSDE
% $$Y^n_{t\wedge \tau} = \mathbf{1}_{\lbrace T < \tau \rbrace} G^n + \int_{t\wedge\tau}^{T\wedge\tau} F^n_s + \tilde{\mathbb{E}}^\alpha[\frac{\delta \beta}{\delta m}(\tilde{\Theta}^\alpha_s, X_{\cdot \wedge s})^\top\tilde{Z}^n_s\vert s < \tilde{\tau}] + \mathbb{E}^\alpha[\mathbf{1}_{\lbrace s < \tau \rbrace} \beta_p(\Theta^\alpha_s)^\top Z^n_s]ds-\int_{t\wedge \tau}^{T\wedge\tau}Z^n_s dW^\alpha_s$$
% Using a standard $\mathbb{L}^2$ fixed point argument, we first show for each $n$, there exists a unique solution $(Y^n, Z^n)$ in $\mathcal{S}^2_\alpha \times \mathcal{M}^2_\alpha$.
We define the fixed point mapping $\Phi: \mathcal{M}_{1, \alpha}^2 \times \mathcal{M}_{d, \alpha}^2 \rightarrow \mathcal{M}_{1, \alpha}^2 \times \mathcal{M}_{d, \alpha}^2$ where for $(y, z) \in \mathcal{M}_{1, \alpha}^2 \times \mathcal{M}_{d, \alpha}^2$, we define $\Phi(y, z) = (Y, Z)$ as the solution of
$$Y_{t\wedge \tau} = \mathbf{1}_{\lbrace T < \tau \rbrace} G + \int_{t\wedge\tau}^{T\wedge\tau} F_s + \tilde{\mathbb{E}}^\alpha[\frac{\delta \beta}{\delta m}(\tilde{\Theta}^\alpha_s, X_{\cdot \wedge s})^\top\tilde{z}_s\vert s < \tilde{\tau}] + \mathbb{E}^\alpha[\mathbf{1}_{\lbrace s < \tau \rbrace} \beta_p(\Theta^\alpha_s)^\top z_s]ds-\int_{t\wedge \tau}^{T\wedge\tau}Z_s dW^\alpha_s.$$
By the martingale representation theorem applied to $\mathbb{P}^\alpha$ and $W^\alpha$, see e.g \cite[Theorem III.5.24]{JS}, $\Phi$ is well defined. Now, consider any $(y^1, z^1), (y^2, z^2)\in \mathcal{M}_{1, \alpha}^2 \times \mathcal{M}_{d, \alpha}^2$, as well as $(Y^1, Z^1) := \Phi(y^1, z^1)$ and $(Y^2, Z^2) := \Phi(y^2, z^2)$. By Itô's lemma, for $\epsilon < L^{-2}(1 + \frac{1}{4q_\alpha})^{-2}$ and $\kappa > \frac{1}{\epsilon}$, we have
\begin{align*}
    e^{\kappa t}\vert Y^1_{t\wedge \tau} - Y^2_{t\wedge\tau}\vert^2& + \int_{t\wedge\tau}^{T\wedge\tau} e^{\kappa s} \Vert Z^1_s - Z^2_s \Vert^2 ds \leq -\int_{t\wedge\tau}^{T\wedge\tau}e^{\kappa s}\bigg\{(\kappa - \frac{1}{\epsilon})\Vert Y^1_s - Y^2_s\Vert^2 \\
    &\quad + \epsilon L^2(1 + \frac{1}{4q_\alpha})^2 \mathbb{E}^\alpha[\Vert z^1_s - z^2_s\Vert^2]\bigg\}ds  - 2\int_{t\wedge \tau}^{T\wedge \tau}e^{\kappa s}(Y^1_s - Y^2_s)(Z^1_s-Z^2_s) dW^\alpha_s
\end{align*}
from which it is straightforward to check that $\Phi$ is a contraction with respect to the norm $\Vert \phi \Vert_\kappa^2 = \mathbb{E}^\alpha[\int_0^T e^{\kappa s} \Vert \phi_s \Vert^2 ds]$ on $\mathcal{M}_{1, \alpha}$ and $\mathcal{M}_{d, \alpha}$ respectively, showing existence and uniqueness of a solution.
\end{proof}

\subsubsection{The Gateaux derivative of the cost functional}
Similar to the approach in the strong formulation, proving the necessary condition relies on differentiating the cost functional. We can do so using the variation process derived above. 
\begin{theorem}\label{thm: gat drv}
    Under Assumptions \ref{asmp: beta} and \ref{asmp: ncs}, and with $\alpha$, $\eta$, and $\alpha^\epsilon$ as in proposition \ref{prop: var prc}, 
    \begin{align*}
        \lim_{\epsilon\searrow 0} \frac{J(\alpha^\epsilon) - J(\alpha)}{\epsilon} 
        &= \mathbb{E}^\alpha\bigg[\mathbf{1}_{\lbrace T < \tau \rbrace} \Lambda^{\alpha, \eta}_T(g(\Theta^\alpha_T) + \tilde{\mathbb{E}}^\alpha\Big[\frac{\delta g}{\delta m}(\tilde{\Theta}^\alpha_T, X)\vert T < \tilde{\tau}\Big] + \mathbb{E}^\alpha[\mathbf{1}_{\lbrace T < \tau \rbrace}g_p(\Theta^\alpha_T)])\bigg]\\
        &\quad + \mathbb{E}^\alpha\bigg[\int_0^{T\wedge\tau} f_a(\Theta^\alpha_s)^\top\eta_s + \Lambda^{\alpha, \eta}_s(f(\Theta^\alpha_s) + \tilde{\mathbb{E}}^\alpha\Big[\frac{\delta f}{\delta m}(\tilde{\Theta}^\alpha_s, X_{\cdot \wedge s})\vert s < \tilde{\tau}\Big]\\
        & \qquad\qquad  + \mathbb{E}^\alpha[\mathbf{1}_{\lbrace s<\tau \rbrace} f_p(\Theta^\alpha_s)])ds\bigg]\\
        %&= \mathbb{E}^\alpha\bigg[\int_0^{T\wedge\tau} \big(f_a(\Theta^\alpha_s) + \beta_a(\Theta^\alpha_s)Z^\alpha_s\big)^\top\eta_s ds\bigg] 
        & =\mathbb{E}^\alpha\bigg[\int_0^{T\wedge \tau}h_a(\Theta^\alpha_s, Z^\alpha_s)^\top\eta_sds\bigg].
    \end{align*}
\end{theorem}
\begin{remark}
    Since in Proposition \ref{prop: var prc}, we showed convergence for all $r\geq 1$, the first equality actually holds under weaker conditions than required in Assumption \ref{asmp: ncs}.
    In fact, the statement remains true if the random variables $\int_0^{T \wedge \tau} \vert f(s, X_{\cdot \wedge s}, 0, \mu^0, 1)\vert ds$ and $\mathbf{1}_{\lbrace T < \tau \rbrace}\vert g(X, \mu^0, 1)\vert$ admit finite $q$-th moment with respect to $\mathbb{P}$ for some $q>1$. The stronger integrability in Assumption \ref{asmp: ncs}.$(iii)$ is required only to ensure well-posedness of the adjoint equation as shown in Proposition \ref{prop: wp adj}.
\end{remark}
\begin{proof}[Proof of Theorem \ref{thm: gat drv}]
    Throughout the proof, we use the same notation as in the proof of Proposition \ref{prop: var prc}. We can consider the running and terminal costs separately. For the running cost terms, we have
    \begin{align*}
        \frac{1}{\epsilon}\bigg(\mathbb{E}^{\alpha^\epsilon}\bigg[\int_0^{T\wedge\tau} f(\Theta^{\alpha^\epsilon}_s)ds\bigg] &- \mathbb{E}^\alpha\bigg[\int_0^{T\wedge\tau} f(\Theta^\alpha_s)ds\bigg]\bigg) \\ 
        &\quad = \mathbb{E}^{\alpha^\epsilon}\bigg[\int_0^{T\wedge\tau} \frac{f(\Theta^{\alpha^\epsilon}_s) - f(\Theta^\alpha_s)}{\epsilon} ds\bigg] + \frac{\mathbb{E}^{\alpha^\epsilon} - \mathbb{E}^\alpha}{\epsilon}\bigg[\int_0^{T\wedge\tau} f(\Theta^\alpha_s)ds \bigg].
    \end{align*}
    By the growth conditions on $f$, we have
    \begin{align*}
        &\bigg\vert\int_0^{T \wedge \tau}\frac{f(\Theta^{\alpha^\epsilon}_s)  -f(\Theta^\alpha_s)}{\epsilon}ds \bigg\vert\\
        &\leq \bigg\vert \int_0^{T \wedge \tau} \frac{M}{\epsilon}\big(1 + \Vert \alpha^\epsilon_s \Vert + \Vert \alpha_s \Vert \big)\Big(\epsilon \Vert \eta_s \Vert + d_{LC}(\mathcal{L}_{\mathbb{P}^{\alpha^\epsilon}}(X_{\cdot\wedge s}\vert s < \tau), \mathcal{L}_{\mathbb{P}^\alpha}(X_{\cdot\wedge s}\vert s < \tau)) \\ 
        &\qquad\qquad\qquad + \vert (\mathbb{P}^{\alpha^{\epsilon}}  - \mathbb{P}^\alpha)[s< \tau]\vert \Big) ds\bigg\vert\\
        &\leq 3M\bigg( T + \int_0^{T\wedge\tau} 2\epsilon^2 \Vert \eta_s \Vert^2 + 3\Vert \alpha_s \Vert^2 ds \bigg)^{\frac{1}{2}}\bigg( \int_0^{T \wedge \tau}\Vert \eta_s\Vert^2 ds + T C^\Gamma_P \Vert \eta\Vert_{\text{BMO}}^2 \bigg)^{\frac{1}{2}}
    \end{align*}
    so that using \cite[Theorem 8.2.21]{CE}, we can see that $\mathbb{E}^\alpha[(\int_0^{T \wedge \tau}\frac{f(\Theta^{\alpha^\epsilon}_s)-f(\Theta^\alpha_s)}{\epsilon}ds)^2]$ is finite and bounded over $\epsilon$. Recall that we wrote $M$ for $M(q_\Gamma)$.
    Thus, by Proposition \ref{prop: var prc}, $(\mathbb{E}^\alpha - \mathbb{E}^{\alpha^\epsilon})[\int_0^{T\wedge\tau} \frac{f(\Theta^{\alpha^\epsilon}_s) - f(\Theta^\alpha_s)}{\epsilon} ds] \rightarrow 0$. It thus suffices to instead consider
    \begin{align*}
        &\mathbb{E}^\alpha\bigg[\int_0^{T\wedge\tau} \frac{f(\Theta^{\alpha^\epsilon}_s) - f(\Theta^\alpha_s)}{\epsilon} ds\bigg] = \mathbb{E}^\alpha\bigg[\int_0^{T\wedge\tau} \int_0^1 f_a(\Theta^{\alpha^\epsilon, \lambda}_s) \eta_s d\lambda ds \bigg]\\
        &\qquad +\mathbb{E}^\alpha\bigg[ \int_0^{T\wedge\tau}\int_0^1 \tilde{\mathbb{E}}^\alpha\bigg[\frac{1}{\epsilon}\Big(\frac{\mathbb{P}^\alpha[s < \tau]}{\mathbb{P}^{\alpha^\epsilon}[s < \tau]} \tilde{D}^\epsilon_s - 1\Big) \frac{\delta f}{\delta m}(\Theta^{\alpha^\epsilon, \lambda}, \tilde{X}_{\cdot \wedge s})\big\vert s < \tilde{\tau}\bigg] \\
        &\qquad + \mathbb{E}^\alpha\bigg[\frac{1}{\epsilon}(\mathcal{E}^\epsilon_s-1)\mathbf{1}_{\lbrace s<\tau \rbrace}f_p(\Theta^{\alpha^\epsilon, \lambda}_s) d\lambda ds\bigg].
    \end{align*}
    By Proposition \ref{prop: var prc}, we know that $dt \times \mathbb{P}^\alpha$ a.e.\ and for every $\lambda \in [0, 1)$ that for $\epsilon \searrow 0$,
    \begin{align*}
        &\tilde{\mathbb{E}}^\alpha\bigg[\frac{1}{\epsilon}\Big(\frac{\mathbb{P}^\alpha[s<\tau]}{\mathbb{P}^{\alpha^\epsilon}[s<\tau]}\tilde{D}^\epsilon_s - 1\Big)\frac{\delta f}{\delta m}(\Theta^{\alpha^\epsilon, \lambda}, \tilde{X}_{\cdot \wedge s})\vert s<\tilde{\tau} \bigg]\\ 
        &\qquad \longrightarrow \tilde{\mathbb{E}}^\alpha\bigg[ \Big(\tilde{\Lambda}^{\alpha, \eta}_s - \mathbb{E}^\alpha [\Lambda^{\alpha, \eta}_s\vert s<\tau]\Big)\frac{\delta f}{\delta m}(\Theta^\alpha_s, \tilde{X}_{\cdot \wedge s}))\big\vert s <\tilde{\tau} \bigg]\\
        &\qquad\quad  =\tilde{\mathbb{E}}^\alpha\Big[\tilde{\Lambda}^{\alpha, \eta}_s \frac{\delta f}{\delta m}(\Theta^\alpha_s, \tilde{X}_{\cdot \wedge s})\vert s<\tilde{\tau} \Big]
    \end{align*}
    since we know that for any $r > 2$, the random variables $\frac{\delta f}{\delta m}(\Theta^{\alpha^\epsilon, \lambda}, \tilde{X}_{\cdot \wedge s})^{\frac{r}{r-1}}$ are uniformly integrable with respect to $\tilde{\mathbb{P}}^\alpha$. 
    Further, we can see that the random variables $\tilde{\mathbb{E}}^\alpha[\frac{1}{\epsilon}(\frac{\mathbb{P}^\alpha[s<\tau]}{\mathbb{P}^{\alpha^\epsilon}[s<\tau]}\tilde{D}^\epsilon_s - 1)\frac{\delta f}{\delta m}(\Theta^{\alpha^\epsilon, \lambda}, \tilde{X}_{\cdot \wedge s})\vert s<\tilde{\tau}]$ are dominated by $\tilde{\mathbb{E}}^\alpha[\frac{1}{\epsilon^2}(\frac{\mathbb{P}^\alpha[s<\tau]}{\mathbb{P}^{\alpha^\epsilon}[s<\tau]} \tilde{D}^\epsilon_s - 1)^2]^{\frac{1}{2}}\frac{M}{2\sqrt{1-\lambda}}(1 + \Vert \alpha^\epsilon_s \Vert^2 + \Vert \alpha_s \Vert^2)^{\frac{1}{2}}$ which is $dt \times  \mathbb{P}^\alpha \times d\lambda$ uniformly integrable.\par
    Accordingly, $dt \times \mathbb{P}^\alpha \times d\lambda$ a.e., we also know that the $\mathbb{E}[\frac{1}{\epsilon}(\mathcal{E}^\epsilon_s - 1)\mathbf{1}_{\lbrace s<\tau \rbrace}]f_p(\Theta^{\alpha^\epsilon, \lambda}_s)$ are $dt \times\mathbb{P}^\alpha \times d\lambda$ uniformly integrable and converge to $\mathbb{E}^\alpha[\Lambda^{\alpha, \eta}_s \mathbf{1}_{\lbrace s<\tau \rbrace}]f_p(\Theta^\alpha_s)$. Together, we can see that by the dominated convergence theorem,  $\mathbb{E}^{\alpha^\epsilon}[\int_0^{T\wedge\tau} \frac{f(\Theta^{\alpha^\epsilon}_s) - f(\Theta^\alpha_s)}{\epsilon} ds]$ converges to
    $$\mathbb{E}^\alpha\bigg[\int_0^{T\wedge\tau} f_a(\Theta^\alpha_s)^\top\eta_s + \tilde{\mathbb{E}}^\alpha\Big[\tilde{\Lambda}^{\alpha, \eta}_s \frac{\delta f}{\delta m}(\Theta^\alpha_s, \tilde{X}_{\cdot \wedge s})\vert s < \tilde{\tau}\Big] + f_p(\Theta^\alpha_s) \mathbb{E}^\alpha[\Lambda^{\alpha, \eta}_s\mathbf{1}_{\lbrace s<\tau \rbrace}]ds\bigg]$$
    which by Fubini's theorem equals to
    $$\mathbb{E}^\alpha\bigg[\int_0^{T\wedge\tau} f_a(\Theta^\alpha_s)\eta_s + \Lambda^{\alpha, \eta}_s \Big(\tilde{\mathbb{E}}^\alpha\Big[\frac{\delta f}{\delta m}(\tilde{\Theta}^\alpha_s, X_{\cdot \wedge s})\vert s < \tilde{\tau}\Big] + \mathbb{E}^\alpha\big[\mathbf{1}_{\lbrace s<\tau \rbrace} f_p(\Theta^\alpha_s)\big]\Big)ds\bigg].$$
    Next, under Assumption \ref{asmp: ncs}, $\vert\int_0^{T\wedge \tau} f(\Theta^\alpha_s)ds\vert$ is $q$-integrable with respect to $\mathbb{P}$ for some $q > 1$. For any $1 < q' <q$, for sufficiently small $\epsilon$, we thus know $\vert\int_0^{T\wedge \tau} f(\Theta^\alpha_s)ds\vert$ is $q'$-integrable with respect to $\mathbb{P}^\alpha$. By Proposition \ref{prop: var prc}, 
    $$\frac{\mathbb{E}^{\alpha^\epsilon} - \mathbb{E}^\alpha}{\epsilon}\bigg[\int_0^{T\wedge\tau} f(\Theta^\alpha_s)ds\bigg]  \rightarrow \mathbb{E}^\alpha \bigg[\Lambda^{\alpha, \eta}_T \int_0^{T\wedge\tau} f(\Theta^\alpha_s)ds\bigg] = \mathbb{E}^\alpha\bigg[\int_0^{T\wedge\tau} \Lambda^{\alpha, \eta}_s f(\Theta^\alpha_s)ds\bigg]$$
    The terminal cost part can be analyzed in a similar fashion. For the terminal cost, although it takes in less arguments, we still write $\Theta^\alpha_T$ as its argument. Then as $\epsilon \searrow 0$, 
    \begin{align*}
        &\frac{1}{\epsilon}\Big(\mathbb{E}^{\alpha^\epsilon}[\mathbf{1}_{\lbrace T < \tau \rbrace} g(\Theta^{\alpha^\epsilon}_T)]) - \mathbb{E}^\alpha[\mathbf{1}_{\lbrace T < \tau \rbrace}g(\Theta^\alpha_T)]\Big)\\
        &\longrightarrow \mathbb{E}^\alpha\Big[\mathbf{1}_{\lbrace T < \tau \rbrace} \Lambda^{\alpha, \eta}_T\Big(g(\Theta^\alpha_T) + \tilde{\mathbb{E}}^\alpha\Big[\frac{\delta g}{\delta m}(\tilde{\Theta}^\alpha_T, X)\vert T < \tilde{\tau}\Big] + \mathbb{E}^\alpha[\mathbf{1}_{\lbrace T < \tau \rbrace}g_p(\Theta^\alpha_T)]\Big)\Big].
    \end{align*}
    Together, this shows the first equality. The second one is an immediate consequence of partial integration and the definition of the adjoint variables and $\Lambda^{\alpha, \eta}$. Note that $\int_0^{\cdot \wedge \tau} \Lambda^{\alpha, \eta}_s Z^\alpha_s dW^\alpha_s$ is a true martingale since
    \begin{align*}
        \mathbb{E}^\alpha\bigg[\sup_{t \in [0, T]} \int_0^{t \wedge \tau} \Lambda^{\alpha, \eta}_s Z^\alpha_s dW^\alpha_s \bigg] &\leq c_1 \mathbb{E}^\alpha\bigg[\bigg(\int_0^{t \wedge \tau}\Vert\Lambda^{\alpha, \eta}_s Z^\alpha_s\Vert^2ds\bigg)^\frac{1}{2} \bigg]\\ 
        &\leq c_1 \mathbb{E}^\alpha\bigg[\sup_{t \in [0, T]} (\Lambda^{\alpha, \eta}_{t\wedge \tau})^2 + \int_0^{t\wedge\tau}\Vert Z^\alpha_s\Vert^2 ds\bigg]< \infty
    \end{align*}
    and $\int_0^{\cdot \wedge \tau} Y^\alpha_s d\Lambda_s$ is a true martingale by the Émery's inequality \cite[Theorem A.8.15]{CE}.
\end{proof}

\begin{proof}[Proof of Theorem \ref{thm: ncs}] 
    Let $\alpha' \in \mathbb{A}_{BMO}$ be any other control. As above, let $\eta = \alpha'$ and $\alpha^\epsilon = \alpha + \epsilon \eta$. When $\alpha$ is optimal over $\mathbb{A}_{BMO}$, we thus have $J(\alpha) \leq J(\alpha^\epsilon)$ and thus using Theorem \ref{thm: gat drv}, $\mathbb{E}^\alpha[\int_0^{T \wedge \tau}h_a(\Theta^\alpha_s, Z^\alpha_s)^\top(\alpha'_s - \alpha_s) ds] \geq 0$. For the second statement, if it was not the case, there would exists some progressively measurable $U \subset [0, T] \times \Omega$ and $e\in A$ such that $\mathbb{E}^\alpha[\int_0^{T\wedge\tau}\mathbf{1}_U(s, \omega)h_a(\Theta^\alpha_s, Z^\alpha_s)^\top (e -\alpha_s) ds] < 0$. This would contradict with the previous statement by choosing $\alpha' = \mathbf{1}_U e + \mathbf{1}_{U^\complement}\alpha$. The last statement is an immediate consequence of the first order condition for optimality for convex functions.
\end{proof}

\subsection{Proofs for the sufficient condition}
To prove the sufficient condition, we first provide a differential characterization of $p$-convexity.
\begin{proposition}\label{prop: pcnv diff}
    Let the function $\phi:\mathcal{C}\times \mathcal{P}(\mathcal{C})\times (0,1]\to \R$ be differentiable in the sense of Assumption \ref{asmp: ncs}. Then, $\phi$ is $p$-convex if and only if for any $\mu$, $\mu'$, $p$, and $p'$, we have
    \begin{align}
    \nonumber
        p' \int_{\mathcal{C}} \phi(x,& \mu', p') \mu'(dx) - p \int_{\mathcal{C}} \phi(x, \mu, p) \mu(dx) \\ \label{eq: pcnv diff}
        &\geq \int_{\mathcal{C}}\bigg\{ \phi(x, \mu, p) + \int_{\mathcal{C}}\frac{\delta \phi}{\delta m}(\tilde{x}, \mu, p, x)\mu(d\tilde{x}) + p \int_{\mathcal{C}} \phi_p(\tilde{x}, \mu, p) \mu(d\tilde{x})\bigg\}(p'\mu'(dx) - p\mu(dx)).
    \end{align}
\end{proposition}
\begin{proof}
    First, assume that \eqref{eq: pcnv diff} holds. Then,
    \begin{align*}
        &\lambda\bigg(p' \int_{\mathcal{C}} \phi(x, \mu', p') \mu'(dx) - p^\lambda \int_{\mathcal{C}} \phi(x, \overline{\mu}^\lambda, p^\lambda) \overline{\mu}^\lambda(dx) \bigg)\\
        &\geq \lambda \int_{\mathcal{C}}\bigg\{ \phi(x, \overline{\mu}^\lambda, p^\lambda) + \int_{\mathcal{C}}\frac{\delta \phi}{\delta m}(\tilde{x}, \overline{\mu}^\lambda, p^\lambda, x)\overline{\mu}^\lambda(d\tilde{x}) + p^\lambda \int_{\mathcal{C}} \phi_p(\tilde{x}, \overline{\mu}^\lambda, p^\lambda) \overline{\mu}^\lambda(d\tilde{x})\bigg\}(p'\mu'(dx) - p\overline{\mu}^\lambda(dx))
    \end{align*}
    and
    \begin{align*}
        & (1-\lambda)\bigg( p \int_{\mathcal{C}} \phi(x, \mu, p) \mu(dx) - p^\lambda \int_{\mathcal{C}} \phi(x, \overline{\mu}^\lambda, p^\lambda) \overline{\mu}^\lambda(dx) )\\
        &\geq(1-\lambda) \int_{\mathcal{C}}\bigg\{ \phi(x, \overline{\mu}^\lambda, p^\lambda) + \int\frac{\delta \phi}{\delta m}(\tilde{x}, \overline{\mu}^\lambda, p^\lambda, x)\overline{\mu}^\lambda(d\tilde{x}) + p^\lambda \int \phi_p(\tilde{x}, \overline{\mu}^\lambda, p^\lambda)\bigg\} \overline{\mu}^\lambda(d\tilde{x})\big( p\mu(dx) - p\overline{\mu}^\lambda(dx) \big)
    \end{align*}
    and adding these two inequalities proves that $\phi$ is $p$-convex.\par

    Let us now prove the converse. Consider any $\mu$, $\mu'$, $p$, $p'$ and $p^\lambda$ and $\overline{\mu}^\lambda$ as in the definition of $p$-convexity, and write $\overline{\mu}^{\lambda, \theta} = \theta \overline{\mu}^\lambda + (1-\theta) \mu$. For a $p$-convex $\phi$, we have for any $\lambda$ that
    $$p'\int_{\mathcal{C}} \phi(x, \mu', p')\mu'(dx) - p\int_{\mathcal{C}} \phi(x, \mu, p)\mu(dx) \geq \frac{1}{\lambda}\bigg(p^\lambda \int_{\mathcal{C}} \phi(x, \overline{\mu}^\lambda, p^\lambda) \overline{\mu}^\lambda(dx) - p \int_{\mathcal{C}} \phi(x, \mu, p)\mu(dx)\bigg).$$
    Note that for any $\lambda$ and $x$ it holds
    $$\phi(x, \overline{\mu}^\lambda, p^\lambda) - \phi(x, \mu, p) = \int_0^1 \int_{\mathcal{C}} \frac{\delta \phi}{\delta m}(x, \overline{\mu}^{\lambda, \theta}, p^{\lambda\theta}, \tilde{x})(\overline{\mu}^\lambda - \mu)(d\tilde{x})+ \phi_p(x, \overline{\mu}^{\lambda, \theta}, p^{\lambda\theta})(p^\lambda-p)d\theta$$
    so that using the fundamental theorem of calculus applied to the function $\theta\mapsto p^{\lambda\theta}\int_{\mathcal{C}}\phi(x, \bar\mu^{\lambda\theta},p^{\lambda\theta})\bar\mu^{\lambda\theta}(dx)$, we have 
    %({\color{blue}I don't get the idea below.})
    \begin{align*}
        \frac{1}{\lambda}\bigg(p^\lambda &\int_{\mathcal{C}} \phi(x, \overline{\mu}^\lambda, p^\lambda) \overline{\mu}^\lambda(dx) - p \int_{\mathcal{C}} \phi(x, \mu, p)\mu(dx)\bigg)\\
        &=\frac{1}{\lambda}\int_0^1(p^\lambda-p) \bigg\{\int_{\mathcal{C}}  \phi(x, \overline{\mu}^{\lambda, \theta}, p^{\lambda \theta}) \overline{\mu}^{\lambda, \theta}(dx) +  p^{\lambda\theta} \int_{\mathcal{C}} \phi(x, \overline{\mu}^{\lambda, \theta}, p^{\lambda\theta})(\overline{\mu}^\lambda - \mu)(dx)\bigg\} d\theta\\
        &\quad +\frac{1}{\lambda}\int_0^1  p^{\lambda\theta} \int_{\mathcal{C}}\bigg\{ \int_{\mathcal{C}} \frac{\delta \phi}{\delta m}(x, \overline{\mu}^{\lambda, \theta}, p^{\lambda\theta}, \tilde{x})(\overline{\mu}^\lambda - \mu)(d\tilde{x}) + \phi_p(x, \overline{\mu}^{\lambda, \theta}, p^{\lambda, \theta})(p^\lambda-p)\bigg\} \overline{\mu}^{\lambda, \theta}(dx)d\theta.
    \end{align*}
    Note that $\frac{p^\lambda - p}{\lambda} = p' - p$ and $\frac{1}{\lambda}(\overline{\mu}^\lambda - \lambda) = \frac{p'}{p^\lambda}(\mu' - \mu)$. Further, note $d_{LC}(\overline{\mu}^{\lambda, \theta}, \mu)^2 \leq \lambda \frac{\theta p'}{p^\lambda} d_{LC}(\mu', \mu)$. 
    One can thus see that
    \begin{align*}
        \lim_{\lambda\rightarrow 0}&\frac{1}{\lambda}\bigg( p^\lambda \int_{\mathcal{C}} \phi(x, \overline{\mu}^\lambda, p^\lambda) \overline{\mu}^\lambda(dx) - p \int_{\mathcal{C}} \phi(x, \mu, p)\mu(dx)\bigg) \\
        &=(p' - p)\int_{\mathcal{C}} \phi(x, \mu, p)\mu(dx) + p \int_{\mathcal{C}} \phi(x, \mu, p)\frac{p'}{p}(\mu'-\mu)(dx)\\
        &\quad +p\int_{\mathcal{C}}\bigg\{\int_{\mathcal{C}} \frac{\delta \phi}{\delta m}(x, \mu, p, \tilde{x})\frac{p'}{p}(\mu'-\mu)(d\tilde{x}) + \phi_p(x, \mu, p)(p'-p)\bigg\}\mu(dx)\\
        &=\int_{\mathcal{C}}\bigg\{ \phi(x, \mu, p) + \int_{\mathcal{C}}\frac{\delta \phi}{\delta m}(\tilde{x}, \mu, p, x)\mu(d\tilde{x}) + p \int_{\mathcal{C}} \phi_p(\tilde{x}, \mu, p) \mu(d\tilde{x})\bigg\}(p'\mu'(dx) - p\mu(dx))
    \end{align*}
    which shows the converse direction.
\end{proof}
\begin{proof}[Proof of Theorem \ref{thm: sfc}]
    We first consider the case where $A$ is bounded. Then, for any $\alpha' \in\mathbb{A}$, we have $\mathbb{E}^\alpha[(\frac{d\mathbb{P}^{\alpha'}}{d\mathbb{P}})^2]<\infty$. In particular, we can see that $\int_0^\cdot Z^\alpha_s dW^{\alpha'}_s$ is a true $\mathbb{P}^{\alpha'}$-martingale, and further, that
    \begin{align}\nonumber
        & J(\alpha') - J(\alpha) = \mathbb{E}^{\alpha'}[\mathbf{1}_{\lbrace T < \tau \rbrace}g(\theta^{\alpha'}_T)] - \mathbb{E}^\alpha[\mathbf{1}_{\lbrace T < \tau \rbrace}g(\theta^\alpha_T)] + \mathbb{E}^{\alpha'}\bigg[\int_0^{T \wedge \tau} f(\Theta^{\alpha'}_s)ds \bigg] - \mathbb{E}^\alpha\bigg[\int_0^{T \wedge \tau}f(\Theta^\alpha_s)ds \bigg]\nonumber\\
        &\geq (\mathbb{E}^{\alpha'}-\mathbb{E}^\alpha)\bigg[Y^\alpha_T + \int_0^{T \wedge \tau} f^2(\theta^\alpha_s) + \tilde{\mathbb{E}}^\alpha\Big[\frac{\delta f^2}{\delta m}(\tilde{\theta}^\alpha_s, X_{\cdot \wedge s}) \Big] + \mathbb{E}^\alpha[\mathbf{1}_{\lbrace s<\tau \rbrace} f^2_p(\theta^\alpha_s)]  + f^1(s, X_{\cdot \wedge s}, \alpha_s)ds\bigg]\nonumber \\ 
        &+ \mathbb{E}^{\alpha'}\bigg[\int_0^{T\wedge\tau}  h^1_a(s, X_{\cdot \wedge s}, \alpha_s, Z^\alpha_s)^\top (\alpha'_s - \alpha_s) + \frac{m}{2}\Vert \alpha'_s - \alpha_s \Vert^2 - \Big(\beta(s, X_{\cdot \wedge s}, \alpha'_s) - \beta(s, X_{\cdot \wedge s}, \alpha_s)\Big)^\top Z^\alpha_s ds\bigg] \nonumber\\
        &\geq (\mathbb{E}^{\alpha'} - \mathbb{E}^\alpha)[Y^\alpha_0] + \frac{m}{2} \mathbb{E}^{\alpha'}\bigg[\int_0^{T\wedge\tau} \Vert \alpha'_s -\alpha_s \Vert^2 ds\bigg].
        \label{eq: sfc}
    \end{align}
    Let us now consider the case when $A$ is not bounded, but with the remaining assumptions in Assumption \ref{asmp: cnv} still holding true. 
    Let $\alpha\in \mathbb{A}_{\text{BMO}}$ be given as in the statement. 
    Let us first consider any other $\alpha' \in \mathbb{A}_{\text{BMO}}$. 
    Since by \cite[Theorem 3.6]{Kazamaki}, $\int_0^\cdot \beta(\Theta^{\alpha'}_s)-\beta(\Theta^\alpha_s)dW^\alpha_s$ is a $\mathbb{P}^\alpha$-BMO martingale, we have $\mathbb{E}^\alpha[(\frac{d\mathbb{P}^{\alpha'}}{d\mathbb{P}^\alpha})^q] < \infty$ for some $q > 1$. 
    Next noting that $\beta$ is independent of $\mu$ and $p$, the driver of the adjoint equation \eqref{eq: adj eq} simplifies and $Z^\alpha$ is immediately given by the martingale representation theorem. That is,
    \begin{align*}
        \int_0^\cdot Z^\alpha_s dW^\alpha_s&= \mathbb{E}^\alpha\bigg[\mathbf{1}_{\lbrace T < \tau\rbrace}\Big(g(\Theta^\alpha_T) + \tilde{\mathbb{E}}^\alpha\Big[\frac{\delta g}{\delta m}(\tilde{\Theta}^\alpha_T, X)\vert T < \tilde{\tau} \Big] + \mathbb{E}^\alpha[\mathbf{1}_{\lbrace T < \tau\rbrace}g_p(\Theta^\alpha_T)]\Big)\\
            & +\int_{0}^{T \wedge \tau} f(\Theta^\alpha_s) + \tilde{\mathbb{E}}^\alpha\Big[\frac{\delta f^2}{\delta m}(\tilde{\Theta}^\alpha_s, \tilde{Z}^\alpha_s, X_{\cdot \wedge s})\vert s < \tilde{\tau}\Big] + \mathbb{E}^\alpha[\mathbf{1}_{\lbrace s<\tau \rbrace} f^2_p(\Theta^\alpha_s, Z^\alpha_s)]ds\vert \mathcal{F}_\cdot\bigg].
    \end{align*}
    By the additional growth conditions in Assumption \ref{asmp: cnv} and Doob's inequality, we can see that we must have $\mathbb{E}^\alpha[(\int_0^{T \wedge \tau} \Vert Z^\alpha_s \Vert^2 ds)^{\frac{r}{2}}] < \infty$ for any $r > 1$. In particular, this shows that $\int_0^\cdot Z^\alpha_s dW^{\alpha'}_s$ is a true $\mathbb{P}^{\alpha'}$-martingale. Further, one can check that \eqref{eq: sfc} still holds proving the statement for $\alpha' \in \mathbb{A}_{\text{BMO}}$.\par
    
    For general $\alpha' \in \mathbb{A}$, we can consider a sequence $(\alpha^n)_{n\geq 0} $ in $ \mathbb{A}_{\text{BMO}}$ such that $J(\alpha^n) \rightarrow J(\alpha')$ and $\mathcal{H}(\mathbb{P}^{\alpha'}\parallel \mathbb{P}^{\alpha^n})$ like in Proposition \ref{prop: Pa exst} and Lemma \ref{lem: BMO aprx} (such a sequence is given for instance by \eqref{eq:def.alphan}). 
    By construction of $\alpha^n$ we have that a.s., $\frac{d\mathbb{P}^{\alpha^n}}{d\mathbb{P}}\int_0^{T \wedge \tau} \Vert \alpha_s - \alpha^n_s\Vert^2ds \rightarrow \frac{d\mathbb{P}^{\alpha'}}{d\mathbb{P}}\int_0^{T \wedge \tau} \Vert \alpha_s - \alpha'_s \Vert^2 ds$. The statement then follows from Fatou's lemma.
\end{proof}

\section{Existence and uniqueness of optimal controls}\label{sct: exst}
This section is dedicated to the proof of Theorem \ref{thm: exs}.
In particular, we will prove existence of optimal controls and derive integrability properties thereof.

\subsection{Convexity of the cost functional}
We begin by showing convexity of the cost $J$.
To this end, first observe that the set of 
%Let us first demonstrate why under Assumption \ref{asmp: cnv}, 
$\lbrace \mathbb{P}^\alpha \vert \alpha \in \mathbb{A}\rbrace$ is convex.
%, the set of all measures that are attainable through an admissible control, is a convex set. 
In fact, given $\alpha', \alpha\in \mathbb{A}$ and $\lambda \in [0,1]$, if we define the processes 
$$\iota_t := \frac{\lambda \mathcal{E}(\int_0^{\cdot\wedge \tau} \beta(s, X_{\cdot\wedge s}, \alpha'_s)dW_s)_t}{\lambda \mathcal{E}(\int_0^{\cdot\wedge \tau} \beta(s, X_{\cdot\wedge s}, \alpha'_s)dW_s)_t + (1-\lambda)\mathcal{E}(\int_0^{\cdot\wedge \tau} \beta(s, X_{\cdot\wedge s}, \alpha_s)dW_s)_t} \text{ and } \alpha^\lambda := \iota_t \alpha'_t + (1-\iota_t) \alpha_t,$$ then using that $\beta$ is linear in $\alpha$, we can see that $\mathbb{P}^{\alpha^\lambda} = \lambda \mathbb{P}^{\alpha'} + (1-\lambda) \mathbb{P}^\alpha \in \lbrace \mathbb{P}^\alpha \vert \alpha \in \mathbb{A}\rbrace$.

\begin{proposition}\label{prop: J cnv}
    Let Assumptions \ref{asmp: beta}, \ref{asmp: ncs} and \ref{asmp: cnv} hold. Then, $J$ is (strongly) convex in $\mathbb{P}^\alpha$ in the following sense: For any $\alpha, \alpha' \in \mathbb{A}$ and $\lambda\in [0,1]$, let $\alpha^\lambda$ be defined as above. Then, $\alpha^\lambda \in \mathbb{A}$, and
    \begin{align*}
        \lambda J(\alpha') + (1-\lambda) J(\alpha) & \geq J(\alpha^\lambda) + \frac{m}{L^2}\Big(\lambda\mathcal{H}(\mathbb{P}^{\alpha'}\parallel\mathbb{P}^{\alpha^\lambda})+(1-\lambda)\mathcal{H}(\mathbb{P}^\alpha\parallel\mathbb{P}^{\alpha^\lambda}) \Big)\\
        &\geq J(\alpha^\lambda)+\frac{m}{L^2}\lambda(1-\lambda)d_{LC}(\mathbb{P}^\alpha, \mathbb{P}^{\alpha'})^2.
    \end{align*}
    If $m=0$ but $f^1$ is strictly convex in $a$, then strict convexity holds for $J$ as well, i.e.\ the inequality above becomes strict.
\end{proposition}
\begin{remark}\label{rmk: gnr cnv}
    Following the proof of Proposition \ref{prop: J cnv}, we can see that this convexity property  still holds true when the involved change of measures are allowed to be non equivalent as in the setting of Theorem \ref{thm: opt non eq}.
\end{remark}

% For bounded controls, Proposition \ref{prop: J cnv} can also be proven using a similar approach as the proof above by modifying the terms of the adjoint equation properly. Below, the convexity result is proven in a slightly different way that generalizes to general drift.
\begin{proof}[Proof of Proposition \ref{prop: J cnv}]
    Since $\iota_t = \lambda \frac{d\mathbb{P}^{\alpha'}_{\vert \mathcal{F}_t}}{d\mathbb{P}^{\alpha^\lambda}_{\vert \mathcal{F}_t}}$, we have
    \begin{align*}
        \mathbb{E}^{\alpha^\lambda}\bigg[\int_0^{T\wedge \tau} \Vert \alpha^{\lambda}_s \Vert^2ds\bigg] &\leq \mathbb{E}^{\alpha^\lambda}\bigg[\int_0^{T\wedge\tau} \iota_s \Vert\alpha'_s\Vert^2 + (1-\iota_s)\Vert \alpha_s \Vert^2ds\bigg]\\
        &= \lambda \mathbb{E}^{\alpha'}\bigg[\int_0^{T\wedge\tau}\Vert \alpha'_s\Vert^2ds\bigg] + (1-\lambda)\mathbb{E}^\alpha\bigg[\int_0^{T\wedge\tau}\Vert\alpha_s\Vert^2ds\bigg] < \infty
    \end{align*}
    so that $\alpha \in \mathbb{A}$. Similarly, using our convexity properties, we have
    \begin{align*}
        \lambda J(\alpha') &+ (1-\lambda)J(\alpha)\\ 
        & = \mathbb{E}^{\alpha^\lambda}\bigg[\int_0^{T\wedge \tau} \iota_sf^1(s, X_{\cdot \wedge s}, \alpha'_s)  + (1 - \iota_s)f^1(s, X_{\cdot \wedge s}, \alpha_s) ds\bigg]\\
        &\quad +\lambda\mathbb{E}^{\alpha'}\bigg[\int_0^{T\wedge\tau} f^2(\theta^{\alpha'}_s)ds + \mathbf{1}_{\lbrace T < \tau \rbrace}g(\theta^{\alpha'}_T)\bigg] + (1 - \lambda)\mathbb{E}^\alpha\bigg[\int_0^{T\wedge\tau}f^2(\theta^\alpha_s)ds + \mathbf{1}_{\lbrace T < \tau \rbrace}g(\theta^\alpha_T)\bigg]\\
        &\geq \mathbb{E}^{\alpha^\lambda}\bigg[\int_0^{T\wedge\tau}f^1(s, X_{\cdot \wedge s}, \alpha^\lambda_s) + \frac{m}{2}\iota_s(1-\iota_s)\Vert \alpha'_s - \alpha_s \Vert^2 + f^2(\theta^{\alpha^\lambda}_s)ds + \mathbf{1}_{\lbrace T < \tau \rbrace} g(\theta^{\alpha^\lambda}_T)\bigg]\\
        &= J(\alpha^\lambda)+ \frac{m}{2}\mathbb{E}^{\alpha^\lambda}\bigg[\int_0^{T\wedge\tau}\iota_s \Vert \alpha'_s - \alpha^\lambda_s \Vert^2 + (1-\iota_s)\Vert\alpha_s - \alpha^\lambda_s\Vert^2ds\bigg] \\
        &\geq J(\alpha^\lambda) + \frac{m}{L^2}\Big(\lambda\mathcal{H}(\mathbb{P}^{\alpha'}\parallel\mathbb{P}^{\alpha^\lambda})+(1-\lambda)\mathcal{H}(\mathbb{P}^\alpha\parallel\mathbb{P}^{\alpha^\lambda})\Big)
    \end{align*}
    where the first of the two inequalities must be strict if $f^1$ is strictly convex and $\alpha'$ does not a.e. coincide with $\alpha$. 
    To justify the second statement, putting $\Gamma = \frac{d\mathbb{P}^{\alpha'}}{d(\frac{\mathbb{P}^{\alpha'}+\mathbb{P}^\alpha}{2})}$ and using that $x\mapsto x\log(x)$ is $\frac{1}{2}$-strongly convex on $[0, 2]$, use that 
    \begin{align*}
        &\lambda \mathcal{H}(\mathbb{P}^{\alpha'}\parallel \mathbb{P}^{\overline{\alpha}^\lambda}) + (1-\lambda)\mathcal{H}(\mathbb{P}^\alpha\parallel\mathbb{P}^{\overline{\alpha}^\lambda}) \\
        &= \mathbb{E}^{\frac{\mathbb{P}^{\alpha'} + \mathbb{P}^\alpha}{2}}\Big[\lambda \Gamma \log(\Gamma) + (1-\lambda)(2-\Gamma)\log(2-\Gamma)-(\lambda \Gamma + (1-\lambda)(2-\Gamma))\log(\lambda \Gamma + (1-\lambda)(2-\Gamma))\Big]\\
        &\geq \frac{\lambda(1-\lambda)}{4}\mathbb{E}^{\frac{\mathbb{P}^{\alpha'}+\mathbb{P}^\alpha}{2}}[(\Gamma-(2-\Gamma))^2].
    \end{align*}
\end{proof}
Our proof of the existence of an optimal control for bounded $
A$ relies on the convexity property stated above to construct a candidate control. 
As this construction will be used frequently, we first present it as a standalone result in a general framework. 
The key idea is that, since the control processes are typically not compact, we instead work with the corresponding densities of changes of measure.

Before that, we state another preliminary result. For this construction—and later developments—it will be convenient to work with uniform convergence on compacts in probability (ucp) and the Émery topology as defined in \cite[Definitions 12.4.1, 12.4.3]{CE}. For our purposes, ucp means convergence in probability of random variables valued in the space of continuous functions on $[0,T]$, with the supremum norm. The Émery topology is a stronger topology on semimartingales. As we could not find a suitable reference, we briefly recall some standard properties here for completeness.

\begin{lemma}\label{lem: sem top}
    Let $M^n$ and $M$ be continuous local martingales.\begin{itemize}
        \item[(i)]If $M^n$ and $M$ are true martingales and if the $M^n_T$ converges in $\mathbb{L}^1$ to $M_T$, the $M^n$ converge in ucp to $M$.
        \item[(ii)]If the $M^n$ converge in ucp to $M$, then $\langle M^n - M \rangle_T$ converges in probability to $0$. In particular, $M^n$ converges to $M$ in the Émery topology. 
        \item[(iii)]If we have continuous processes $P^n$, $P$ and $Q^n$ and $Q$ such that $P^n$ converges in ucp to $P$, $Q^n$ converges in ucp to $Q$, and the $Q^n$ and $Q$ are strictly positive, then $\frac{P^n}{Q^n}$ converge in ucp to $\frac{P}{Q}$.
    \end{itemize}
\end{lemma}
\begin{proof}
\begin{itemize}
    \item[(i)]By Doob's inequality, for any $\epsilon > 0$, we have $\mathbb{P}[\sup_{t \in [0, T]} \vert M^n_t - M_t \vert > \epsilon] \leq \frac{1}{\epsilon} \mathbb{E}[\vert M^n_T - M_T\vert] \rightarrow 0$ and therefore, ucp convergence.
    \item[(ii)]Let $\nu^n = \inf \lbrace t \geq 0 \vert \vert M^n_s - M_s \vert \geq 1 \rbrace \wedge T$. For each $n$, $M^n_{\cdot \wedge \nu^N} - M_{\cdot \wedge \nu^N}$ is a bounded martingale, therefore, $\mathbb{E}[\langle M^n - M \rangle_{\nu^N}] = \mathbb{E}[(M^n - M)_{\nu^N}^2] \leq \mathbb{E}[1 \wedge \sup_{t \in [0, T]} \vert M^n_t - M_t\vert] \rightarrow 0$. By the ucp convergence, we know that a.s.\ $\nu^N \rightarrow T$, and therefore the $\langle M^n - M \rangle_T$ converge in probability to $0$. By \cite[Proposition 2.7]{Kardaras}, this shows convergence in the Émery topology.
    \item[(iii)]For deterministic continuous functions, it is easy to check convergence in the supremum metric of the quotients. Ucp convergence of the processes then follows from the continuous mapping theorem.
\end{itemize}
\end{proof}
It is immediate to see that convergence in ucp is invariant under a change to an equivalent probability measure. The same applies to the Émery topology as well by \cite[Proposition 6]{Emery}.\par
Let us also briefly recall some properties of the density process of absolutely continuous but not necessarily equivalent measures from \cite{Leonard}. 
By \cite[Theorem 1, 2]{Leonard}, for any $\mathbb{Q} \ll \mathbb{P}$ such that $\mathcal{H}(\mathbb{Q} \parallel \mathbb{P}) <\infty$, there exists a progressively measurable process $\beta^\mathbb{Q}$ defined $dt \times \mathbb{Q}$ a.e.\ such that $\frac{d\mathbb{Q}}{d\mathbb{P}} = \mathbf{1}_{\lbrace \frac{d\mathbb{Q}}{d\mathbb{P}}>0 \rbrace} \mathcal{E}(\int_0^{\cdot \wedge \tau} \beta^\mathbb{Q} dW_s)_T$. Further, by \cite[Corollary 1]{Leonard}, we can still $dt \times \mathbb{P}$ a.e.\ make sense of $\frac{d\mathbb{Q}_{\vert \mathcal{F}_t}}{d\mathbb{P}_{\vert \mathcal{F}_t}} = \mathcal{E}(\int_0^{\cdot \wedge \tau} \beta^\mathbb{Q}_s dW_s)_t$ as a stochastic exponential that eventually will reach zero. Once it hits zero, it will stay at zero, and will be zero exactly outside of the support of $\mathbb{Q}$.

\begin{lemma}\label{lem: drft cstr}
    Assume $\beta$ satisfies Assumption \ref{asmp: cnv}. Consider a sequence of $A$ valued progressively processes $\alpha^n$ and let $\mathbb{P}^{\alpha^n} \ll \mathbb{P}$ be such that $\frac{d\mathbb{P}^{\alpha^n}}{d\mathbb{P}} = \mathbf{1}_{\lbrace \frac{d\mathbb{P}^{\alpha^n}}{d\mathbb{P}}>0 \rbrace} \mathcal{E}(\int_0^{\cdot \wedge \tau} \beta(s, X_{\cdot \wedge s}, \alpha^n_s)dW_s)_T$. Assume $\sup_n \mathcal{H}(\mathbb{P}^{\alpha^n} \parallel \mathbb{P}) <\infty$.\par
    Then, up to a subsequence, the $\mathbb{P}^{\alpha^n}$ converge setwise to some probability measure $\tilde{\mathbb{P}}$ (i.e.\ for any bounded random variable $Z$, we have $\mathbb{E}^{\alpha^n}[Z] \rightarrow \mathbb{E}^{\tilde{\mathbb{P}}}[Z]$) that satisfies $\mathcal{H}(\tilde{\mathbb{P}}\parallel\mathbb{P})< \infty$ and thus $\tilde{\mathbb{P}} \ll \mathbb{P}$.
    Moreover, the measure $\tilde{\mathbb{P}}$ is given through a control as well, i.e.\ there exists an $A$-valued progressively measurable process $\tilde{\alpha}$ that is a.e. unique on $\lbrace \frac{d\tilde{\mathbb{P}}}{d\mathbb{P}} > 0 \rbrace \cap [0,\tau)$ such that $d\tilde{\mathbb{P}} = \mathbf{1}_{\lbrace \frac{d\mathbb{P}^{\tilde{\alpha}}}{d\mathbb{P}}>0 \rbrace} \mathcal{E}(\int_0^{\cdot \wedge \tau} \beta(s, X_{\cdot \wedge s}, \tilde{\alpha}_s)dW_s)_Td\P$. 
\end{lemma}
\begin{remark}
    The measure $\tilde{\mathbb{P}}$ can be constructed as follows: For $1 \leq i \leq j$, there are progressively measurable $[0,1]$-valued processes $\iota^{i, j}$ such that for each $i$, all but finitely many $\iota^{i, j}$ are zero, $dt \times \mathbb{P}$ a.e.\ $\sum_{j \geq i} \iota^{i, j}_t = 1$, and such that for $\tilde{\alpha}^i_t = \sum_{j \geq i} \iota^{i, j}_t \alpha_t$, for each $i\geq 1$, the $\mathbb{P}^{\tilde{\alpha}^i}$ are a finite convex combination of the $(\mathbb{P}^{\alpha^j})_{j \geq i}$. Further, the $\mathbb{P}^{\tilde{\alpha}^j}$ converge in total variation to $\tilde{\mathbb{P}}$ and $dt \times \tilde{\mathbb{P}}$ a.e.\ on $\lbrace t < \tau \rbrace$, the $\tilde{\alpha}^i_t$ converge to $\tilde{\alpha}_t$ for $i \rightarrow \infty$. 
\end{remark}
\begin{proof}
    By the de la Vallée-Poussin theorem, the family $(\frac{d\mathbb{P}^{\alpha^n}}{d\mathbb{P}})_{n\ge 1}$ is uniformly integrable with respect to $\mathbb{P}$. By the Dunford-Pettis theorem, they are thus weakly $\mathbb{L}^1$ compact with respect to $\mathbb{P}$. Up to a subsequence, the $\frac{d\mathbb{P}^{\alpha^n}}{d\mathbb{P}}$ thus converge weakly in $\mathbb{L}^1$ to an integrable non negative random variable with expectation $1$. For notational convenience, we will often neglect that convergence only occurs through subsequences. We can thus define a measure $\tilde{\mathbb{P}}$ equivalent to $\mathbb{P}$ such that up to a subsequence, $(\frac{d\mathbb{P}^{\alpha^n}}{d\mathbb{P}})_{n\ge1}$ converges weakly in $\mathbb{L}^1$ with respect to $\mathbb{P}$ to $\frac{d\tilde{\mathbb{P}}}{d\mathbb{P}}$. Note that this implies that the sequence $(\mathbb{P}^{\alpha^n})_{n\ge1}$ converges to $\tilde{\mathbb{P}}$ in the setwise sense. 
    By lower semicontinuity of the entropy \cite[Theorem 4.9]{Polyanskiy}, $\mathcal{H}(\tilde{\mathbb{P}}\parallel\mathbb{P})$ is finite. Thus, by \cite[Theorem 2]{Leonard}, there exists a progressively measurable process $\tilde{\beta}$ defined on the support of $\tilde{\mathbb{P}}$ such that $\tilde{\mathbb{P}}$ a.s.\ $\int_0^T \Vert \tilde{\beta}_s \Vert^2ds < \infty$ and $\frac{d\mathbb{\tilde{\mathbb{P}}}}{d\mathbb{P}} = \mathbf{1}_{\lbrace \frac{d\mathbb{\tilde{\mathbb{P}}}}{d\mathbb{P}} > 0\rbrace}\mathcal{E}(\int_0^{\cdot \wedge \tau} \tilde{\beta}_sdW_s)_T$. Outside of the support of $\tilde{\mathbb{P}}$, we set $\tilde{\beta} = 0$.\par  
    By Mazur's lemma \cite[Lemma 10.19]{Renardy}, for $1 \leq i \leq j$, there exists $\lambda^{i, j} \in [0, 1]$ such that for each $i$, all but finitely many $\lambda^{i, j}$ are finite, $\sum_{j \geq i} \lambda^{i, j} = 1$, and the $\sum_{j \geq i} \lambda^{i, j} \frac{d\mathbb{P}^{\alpha^i}}{d\mathbb{P}}$ converge strongly in $\mathbb{L}^1$ with respect to $\mathbb{P}$ to $\frac{d\tilde{\mathbb{P}}}{d\mathbb{P}}$. Notice that this implies that $\frac{d\tilde{\mathbb{P}}}{d\mathbb{P}}$ is $\mathcal{F}_{T\wedge \tau}$ measurable. We can now define the process such that
    \begin{equation}
    \label{eq:def.tilde.alpha.lemma}
        \tilde{\alpha}^i_t := \sum_{j \geq i} \iota^{i, j}_t \alpha^j_t \text{ with } \iota^{i, j}_t = \frac{\lambda^{i, j} \mathcal{E}(\int_0^{\cdot \wedge \tau} \beta(s, X_{\cdot \wedge s}, \alpha^j)dW_s)_t}{\sum_{j' \geq i} \lambda^{i, j'} \mathcal{E}(\int_0^{\cdot \wedge \tau} \beta(s, X_{\cdot \wedge s}, \alpha^{j'})dW_s)_t}\mathbf{1}_{B_i} + \mathbf{1}_{\{i=j\}}\mathbf{1}_{B^c_i}
    \end{equation}
    with
    $B_i := \lbrace \sum_{j' \geq i} \lambda^{i, j'} \mathcal{E}(\int_0^{\cdot \wedge \tau} \beta(s, X_{\cdot \wedge s}, \alpha^{j'})dW_s)_t > 0 \rbrace$.
    % and otherwise $\iota^{i, j}_t = \mathbf{1}_{\lbrace i = j\rbrace}$. 
    %Also, let $\tilde{\alpha}^i_t := \sum_{j \geq i} \iota^{i, j}_t \alpha^j_t$. 
    By linearity of $\beta$ in $a$, we have $\beta(t, X_{\cdot \wedge t}, \tilde{\alpha}^i_t) = \sum_{j \geq i} \iota^{i, j}_t \beta(t, X_{\cdot \wedge t}, \alpha^j_t)$, and it is therefore easy to see that $\frac{d\mathbb{P}^{\tilde{\alpha}^i}}{d\mathbb{P}} = \sum_{j\geq i}\lambda^{i, j}\frac{d\mathbb{P}^{\alpha^j}}{d\mathbb{P}}$.\par
    Let us for now additionally assume $\tilde{\mathbb{P}} \sim \mathbb{P}$ and $\mathbb{P}^{\alpha^n} \sim \mathbb{P}$ for all $n \geq 0$. As we have strong convergence of $(\frac{d\mathbb{P}^{\tilde{\alpha}^n}}{d\mathbb{P}})_{n\ge1}$ to $\frac{d\tilde{\mathbb{P}}}{d\mathbb{P}}$, by Lemma \ref{lem: sem top}, the martingales $\mathcal{E}(\int_0^{\cdot \wedge \tau} \beta(s, X_{\cdot \wedge s}, \tilde{\alpha}^n_s)dW_s)$ converge in ucp to $\mathcal{E}(\int_0^{\cdot \wedge \tau} \tilde{\beta}_s dW_s)$ and therefore also in the Émery topology. Further $\frac{1}{\mathcal{E}(\int_0^{\cdot \wedge \tau} \beta(s, X_{\cdot \wedge s}, \tilde{\alpha}^n_s)dW_s)}$ converges in ucp to $\frac{1}{\mathcal{E}(\int_0^{\cdot \wedge \tau} \tilde{\beta}_s dW_s)}$, and thus by \cite[Theorem V.15]{Protter13}, the $\int_0^{\cdot \wedge\tau} \beta(s, X_{\cdot \wedge s}, \tilde{\alpha}^n_s) dW_s = \int_0^\cdot \frac{1}{\mathcal{E}(\int_0^{\cdot \wedge \tau} \beta(s, X_{\cdot \wedge s}, \tilde{\alpha}^n_s)dW_s)}d\mathcal{E}(\int_0^{\cdot \wedge \tau} \beta(s, X_{\cdot \wedge s}, \tilde{\alpha}^n_s)dW_s)$ converge in ucp to $\int_0^{\cdot \wedge \tau} \tilde{\beta}_s dW_s$. Again by Lemma \ref{lem: sem top}, this shows $\int_0^{T \wedge \tau} \Vert\beta(s, X_{\cdot \wedge s}, \tilde{\alpha}^n_s) - \tilde{\beta}_s \Vert^2ds$ converges in probability to zero and therefore, up to a subsequence, $dt \times \mathbb{P}$ a.e.\ on $\lbrace t < \tau \rbrace$ that $\beta(t, X_{\cdot \wedge t}, \tilde{\alpha}^n_t) \rightarrow \tilde{\beta}_t$. In particular, since $\beta^1$ has a.e.\ linearly independent columns, this shows that there exists an $A$-valued process $\tilde{\alpha}$ such that $dt \times \mathbb{P}$ a.e.\ on $\lbrace t < \tau \rbrace$, we have $\tilde{\beta}_t = \beta(t, X_{\cdot \wedge t}, \tilde{\alpha}_t)$ and such that the $\tilde{\alpha}^n_t$ converge to $\tilde{\alpha}_t$.\par
    It remains to examine what happens when equivalence of the measures is not necessarily guaranteed. Let $\alpha^0$ be any bounded $A$ valued process. Of course, $\mathbb{P}^{\alpha^0} \sim \mathbb{P}$. Just like above, let us define $\iota^n_t = \frac{\mathcal{E}(\int_0^{\cdot \wedge \tau}\beta(s, X_{\cdot\wedge s}, \tilde{\alpha}^n_s)dW_s)_t}{\mathcal{E}(\int_0^{\cdot \wedge \tau}\beta(s, X_{\cdot\wedge s}, \tilde{\alpha}^n_s)dW_s)_t + \mathcal{E}(\int_0^{\cdot \wedge \tau}\beta(s, X_{\cdot\wedge s}, \alpha^0_s)dW_s)_t}$ and $\iota_t = \frac{\mathcal{E}(\int_0^{\cdot \wedge \tau} \tilde{\beta}_sdW_s)_t}{\mathcal{E}(\int_0^{\cdot \wedge \tau}\tilde{\beta}_sdW_s)_t + \mathcal{E}(\int_0^{\cdot \wedge \tau}\beta(s, X_{\cdot\wedge s}, \alpha^0_s)dW_s)_t}$, as well as $\tilde{\alpha}^{0, n}_t = \iota^n_t \tilde{\alpha}^n_t + (1 - \iota^n_t) \alpha^0_t$ and $\tilde{\beta}^0 = \iota_t \tilde{\beta}_t + (1 - \iota_t) \beta(t, X_{\cdot \wedge t}, \alpha^0_t)$. This way, $\mathbb{P}^{\tilde{\alpha}^{0, n}} = \frac{\mathbb{P}^{\tilde{\alpha}^n} + \mathbb{P}^{\alpha^0}}{2}$, and $\frac{d\frac{\tilde{\mathbb{P}} + \mathbb{P}^{\alpha^0}}{2}}{d\mathbb{P}} = \mathcal{E}(\int_0^{\cdot \wedge \tau}\tilde{\beta}^0_s dW_s)_T$. We still have that $\mathbb{P}^{\tilde{\alpha}^n}$ converges in total variation to $\tilde{\mathbb{P}}$, which implies $dt \times \mathbb{P}$ a.e.\ $\iota^n_t \rightarrow \iota_t$, as well as that the $\mathbb{P}^{\tilde{\alpha}^{0, n}}$ converge in total variation to $\frac{\tilde{\mathbb{P}} + \mathbb{P}^{\alpha^0}}{2}$. Since the $\frac{\mathbb{P}^{\tilde{\alpha}^n} + \mathbb{P}^{\alpha^0}}{2}$ and $\frac{\tilde{\mathbb{P}}+\mathbb{P}^{\alpha^0}}{2}$ are equivalent to $\mathbb{P}$ we can use the argument above to see that $dt \times \mathbb{P}$ a.e., we have $\iota^n_t \beta(t, X_{\cdot \wedge t}, \tilde{\alpha}^n_t) + (1- \iota^n_t)\beta(t, X_{\cdot \wedge t}, \alpha^0_t) \rightarrow \iota_t \tilde{\beta_t}+(1-\iota_t)\beta(t, X_{\cdot \wedge t}, \alpha^0_t)$. 
    This implies that on the event $\lbrace \iota_t > 0 \rbrace$, we have $\beta(t, X_{\cdot \wedge t}, \tilde{\alpha}^n_t)\to \tilde\beta_t$. 
    The event $\lbrace \iota_t > 0 \rbrace$ coincides with the event $\lbrace \frac{d\tilde{\mathbb{P}}_{\vert \mathcal{F}_t}}{d\mathbb{P}_{\vert \mathcal{F}_t}} > 0\rbrace$ and it is thus equivalent to say that this convergence happens $dt \times \tilde{\mathbb{P}}$ a.e. Again, as $\beta^1$ was assumed to be full rank, we can $\tilde{\mathbb{P}}$ a.e.\ define an $A$-valued process $\tilde{\alpha}$ such that $dt \times \tilde{\mathbb{P}}$ a.e., $\tilde{\beta}_t = \beta(t, X_{\cdot \wedge t}, \tilde{\alpha}_t)$. 
\end{proof}
\subsection{The optimal control}
Let us now come to the proof of the existence result. 
In this section we will also derive integrability properties of the optimal control.
\begin{proof}[Proof of Theorem \ref{thm: exs}]
    Uniqueness for strictly convex $f^1$ is an immediate consequence from Proposition \ref{prop: J cnv}.\par

    To prove existence, let us first consider the case where $A$ is bounded. 
    It is immediate $\inf_{\alpha \in \mathbb{A}} J(\alpha) > - \infty$, and we can consider any minimizing sequence $\alpha^n\in \mathbb{A}$ such that $J(\alpha^n) \leq \inf_{\alpha\in\mathbb{A}}J(\alpha) + \frac{1}{n}$. As $A$ is bounded, we can see that the $\mathcal{H}(\mathbb{P}^{\alpha^n}\parallel\mathbb{P})$ are bounded in $n$. By Lemma \ref{lem: drft cstr}, we can find $\hat{\mathbb{P}}$, such that up to a subsequence, the $\mathbb{P}^{\alpha^n}$ converge to $\hat{\mathbb{P}}$. As this $\hat{\mathbb{P}}$ can also be attained as the limit in total variation of convex combinations of the $\mathbb{P}^{\alpha^n}$, by Fatou's lemma $\mathbb{E}[\frac{d\mathbb{P}}{d\mathbb{\hat{\mathbb{P}}}}] \leq \liminf_{n \rightarrow\infty}\mathbb{E}[\frac{d\mathbb{P}}{d\mathbb{P}^{\alpha^n}}] < \infty$.
    Thus, $\frac{d\hat{\mathbb{P}}}{d\mathbb{P}}>0$ a.s., and hence $\hat{\mathbb{P}} \sim \mathbb{P}$. 
    Therefore, by Lemma \ref{lem: drft cstr}, we can find $\hat{\alpha} \in \mathbb{A}$ such that $\mathbb{P}^{\hat{\alpha}} = \hat{\mathbb{P}}$. Following the construction in Lemma \ref{lem: drft cstr}, the $\mathbb{P}^{\tilde{\alpha}^n}$ are convex combinations of the $\mathbb{P}^{\alpha^n}$ and by the convexity properties proven in Proposition \ref{prop: J cnv}, we have $J(\tilde{\alpha}^n)\leq J(\alpha^n) \leq \inf_{\alpha \in \mathbb{A}} J(\alpha) + \frac{1}{n}$. We further know that $d_{\text{TV}}(\mathbb{P}^{\hat{\alpha}}, \mathbb{P}^{\tilde{\alpha}^n}) \rightarrow 0$, from which it follows that for any $t \geq 0$, we have $\mathbb{P}^{\tilde{\alpha}^n}[t<\tau] \rightarrow \mathbb{P}^{\hat{\alpha}}[t<\tau]$ and $d_{\text{TV}}(\mathcal{L}_{\mathbb{P}^{\tilde{\alpha}^n}}(X_{\cdot \wedge t} \vert t<\tau) , \mathcal{L}_{\mathbb{P}^{\hat{\alpha}}}(X_{\cdot \wedge t}\vert t<\tau)) \rightarrow 0$. By boundedness of our cost coefficients, we have by the dominated convergence theorem that $J(\tilde{\alpha}^n) \rightarrow J(\hat{\alpha})$ and hence optimality.\par
    We now consider the case where $A$ is allowed to be unbounded. We will construct the optimal control in the general case as a limit of optimal controls for bounded control spaces. For each $N\geq1$ let us consider $A^N := \overline{B_0(N)} \cap A$ and $\mathbb{A}^N$ as the set of all $A^N$ valued progressively measurable $\alpha$. Clearly, $\mathbb{A}^N \subset \mathbb{A}_{\text{BMO}}$. 
    We let $\hat{\alpha}^N$ be the optimizer of $J$ over $\mathbb{A}^N$.\par
    Note that the $J(\hat{\alpha}^N)$ are decreasing in $N$ and in particular, thanks to Assumption \ref{asmp: cnv}, it must be bounded in $N$. 
    %{\color{blue}I think we should also argue that $\inf_{\alpha\in \mathbb A}J(\alpha)>-\infty$.} 
    In addition, using the coercivity property and strong convexity, this implies that the sequence $\mathcal{H}(\mathbb{P}^{\hat{\alpha}^N}\parallel \mathbb{P}) = \mathbb{E}^\mathbb{P}[\frac{d\mathbb{P}^{\hat{\alpha}^N}}{d\mathbb{P}} \log(\frac{d\mathbb{P}^{\hat{\alpha}^N}}{d\mathbb{P}})]$ is bounded. Let $\hat{\mathbb{P}} \ll \mathbb{P}$ be the limit probability measure constructed using Lemma \ref{lem: drft cstr}.
    It remains to show that $\hat{\mathbb{P}} = \mathbb{P}^{\hat\alpha}$ for some optimal $\hat\alpha \in \mathbb{A}$.\par

    Now, for some fixed $n$, consider any $\gamma \in \mathbb{A}^n$. As each $\hat{\alpha}^N$ is optimal for $J$ over $\mathbb{A}^N$, by Theorem \ref{thm: sfc}, for any $N \geq n$, we have $\mathcal{H}(\mathbb{P}^\gamma\parallel\mathbb{P}^{\hat{\alpha}^N}) \leq \frac{2L^2}{m}(J(\gamma) - J(\hat{\alpha}^N))$. As we have that $J$ is uniformly bounded from below thanks to Assumption \ref{asmp: cnv}, 
    %{\color{blue}(I don't see this in the assumptions.)}
     the right hand side is bounded in $N$. As the $\mathbb{P}^{\hat{\alpha}^N}$ converge to $\hat{\mathbb{P}}$ setwise up to a subsequence, by lower semicontinuity of the relative entropy, we have $\mathcal{H}(\mathbb{P}^\gamma\parallel\hat{\mathbb{P}})<\infty$. In particular, $\hat{\mathbb{P}}$ is equivalent to $\mathbb{P}$ since we already know that $\mathbb{P}^\gamma$ is equivalent to $\mathbb P$. 
     \par

    Following Lemma \ref{lem: drft cstr}, there exists a progressively measurable process $\hat{\alpha}$ valued in $A$ such that $\hat{\mathbb{P}} = \mathbb{P}^{\hat{\alpha}}$. 
    Further, let $\tilde{\alpha}^n$ be as in Lemma \ref{lem: drft cstr}. By convexity of $J$ as shown in Proposition \ref{prop: J cnv}, we can see that for these $\tilde{\alpha}^n$, we have $\liminf_{n\rightarrow \infty} J(\tilde{\alpha}^n) \leq \liminf_{n\rightarrow\infty}J(\hat{\alpha}^n)$. Now,
    \begin{align*}
        J(\tilde{\alpha}^n) &=  \mathbb{E}^{\hat{\alpha}}\bigg[\int_0^{T\wedge \tau} \frac{d\mathbb{P}^{\tilde{\alpha}^n}}{d\mathbb{P}^{\hat{\alpha}}}\Big( f^1(s, X_{\cdot \wedge s}, \tilde{\alpha}^n_s) + f^2(s, X_{\cdot \wedge s}, \mathcal{L}_{\mathbb{P}^{\tilde{\alpha}^n}}(X_{\cdot \wedge s} \vert s <\tau), \mathbb{P}^{\tilde{\alpha}^n}[s<\tau]) + \underline{C}\Big)ds\\
        \quad& + \frac{d\mathbb{P}^{\tilde{\alpha}^n}}{d\mathbb{P}^{\hat{\alpha}}} \Big( g(X, \mathcal{L}_{\mathbb{P}^{\tilde{\alpha}^n}}(X\vert T<\tau), \mathbb{P}^{\tilde{\alpha}^n}[T<\tau]) + \underline{C} \Big)\bigg] - \underline{C}(1 + \mathbb{E}^{\tilde{\alpha}^n}[T \wedge \tau])
    \end{align*}
    where we chose $\underline{C} > 0$ as some constant upper bounding $-(f_1 + f_2)$ and $-g$. As $d_{\text{TV}}(\mathbb{P}^{\tilde{\alpha}^n}, \mathbb{P}^{\hat{\alpha}}) \rightarrow 0$, we know for each $t\geq 0$ that $\mathbb{P}^{\tilde{\alpha}^n}[t<\tau] \rightarrow \mathbb{P}^{\hat{\alpha}}[t<\tau]$ and $d_{\text{LC}}(\mathcal{L}_{\mathbb{P}^{\tilde{\alpha}^n}}(X_{\cdot \wedge t}\vert t<\tau), \mathcal{L}_{\mathbb{P}^{\hat{\alpha}}}(X_{\cdot \wedge t}\vert t<\tau)) \rightarrow 0$. Since we already know that $dt \times \mathbb{P}$ a.e., $\tilde{\alpha}^n_t \rightarrow \hat{\alpha}_t$ and $\frac{d\mathbb{P}^{\tilde{\alpha}^n}}{d\mathbb{P}^{\hat{\alpha}}} \rightarrow 1$, by Fatou's lemma, we can thus see that $J(\hat{\alpha}) \leq \liminf_{n\rightarrow\infty} J(\tilde{\alpha}^n) \leq \liminf_{n\rightarrow\infty} J(\hat{\alpha}^n) < \infty$ and thus $\hat{\alpha} \in \mathbb{A}$.\par

    To argue that $J(\hat{\alpha}) \geq \limsup_{n\rightarrow \infty}J(\hat{\alpha}^n)$, relying on Lemma \ref{lem: BMO aprx}, it suffices to argue that for any $\alpha \in \mathbb{A}_{\text{BMO}}$, there exists a sequence of controls $(\alpha^n)_{n \geq 1}$ such that each $\alpha^n$ is bounded, and that $J(\alpha^n) \rightarrow J(\alpha)$. Define $\alpha^n_t = \alpha_t \mathbf{1}_{\lbrace\Vert \alpha_t \Vert \leq n \rbrace}$. It is immediate that $\mathcal{H}(\mathbb{P}^\alpha \parallel \mathbb{P}^{\alpha^n})\rightarrow 0$. Further, we have $\Vert \alpha^n \Vert_{\text{BMO}} \leq \Vert \alpha \Vert_{\text{BMO}}$ so that by \cite[Theorem 3.1]{Kazamaki}, there must exist some $r >1$ such that $\mathbb{E}[(\frac{d\mathbb{P}^\alpha}{d\mathbb{P}} - \frac{d\mathbb{P}^{\alpha^n}}{d\mathbb{P}})^r] \rightarrow 0$. At the same time, since $\int_0^{T \wedge \tau} \Vert \alpha_s \Vert^2 ds$ is $q$ integrable with respect to $\mathbb{P}$ for any $q>1$ by \cite[Theorem 8.2.21]{CE}, we can follow that $J(\alpha^n) \rightarrow J(\alpha)$.
    In conclusion, we have shown that $J(\hat{\alpha}^n) \searrow J(\hat{\alpha})$. Again relying on the approximation arguments we have shown, we have $J(\hat{\alpha}) = \inf_{\alpha \in \mathbb{A}}J(\alpha)$, and even $\mathcal{H}(\mathbb{P}^\alpha\parallel\mathbb{P}^{\hat{\alpha}}) \leq \frac{2L^2}{m}(J(\alpha) - J(\hat{\alpha}))$ for any $\alpha \in \mathbb{A}$.\par
    It remains to show that $\hat{\alpha} \in \mathbb{A}_{\text{BMO}}$. The proof of this statement requires a lengthy analysis of the adjoint equation provided below in Lemma \ref{lem: opt BMO}.
\end{proof}

\subsubsection{The BMO property}
To conclude the proof of Theorem \ref{thm: exs}, it remains to show that the optimal control satisfies $\hat\alpha\in \mathbb{A}_{\text{BMO}}$.
To this end, we need to use the necessary condition of optimality (note however that 
% we will check check that the necessary condition holds for $\hat{\alpha}$ although 
we cannot apply Theorem \ref{thm: ncs} directly as this theorem is proved for optimal controls $\hat{\alpha} \in \mathbb{A}_{BMO}$). 
Throughout this section, we are working in the framework of Theorem \ref{thm: exs} and its proof, with $A$ possibly unbounded.\par

First, observe that that since $\beta$ is independent of $\mu$ and $p$ (Assumption \ref{asmp: cnv}.$(i)$) we do not need to require the BMO property and Proposition \ref{prop: wp adj} to solve the adjoint equation \eqref{eq: adj eq}. 
In fact, for any $\alpha \in \mathbb{A}$, since the driver becomes independent of the adjoint variables, we can find $(Y^\alpha, Z^\alpha)$ as the solution of
\begin{align*}
    Y^\alpha_{t\wedge \tau} &= \mathbf{1}_{\lbrace T < \tau \rbrace}\Big(g(\theta^\alpha_T) + \tilde{\mathbb{E}}^\alpha\Big[\frac{\delta g}{\delta m}(\tilde{\theta}^\alpha_T, X)\vert T < \tilde{\tau}\Big] + \mathbb{E}^\alpha[\mathbf{1}_{\lbrace T < \tau \rbrace}g_p(\theta^\alpha_T)]\Big)\\
    &\quad + \int_{t\wedge\tau}^{T\wedge\tau}f^1(s, X_{\cdot \wedge s}, \alpha_s) + f^2(s, \theta^\alpha_s) + \tilde{\mathbb{E}}^\alpha\Big[\frac{\delta f^2}{\delta m}(s, \tilde{\theta}^\alpha_s, X_{\cdot \wedge s})\vert s<\tilde{\tau}\Big] + \mathbb{E}^\alpha[\mathbf{1}_{\lbrace s<\tau \rbrace}f^2_p(s, \theta^\alpha_s)]ds\\ 
    &\quad  -\int_{t\wedge\tau}^{T\wedge\tau} Z^\alpha_s dW^\alpha_s
\end{align*}
where $Y^\alpha$ is given by a conditional expectation and $Z^\alpha$ is found using the martingale representation theorem with respect to $\mathbb{P}^\alpha$ using \cite[Theorem III.5.24]{JS}.  
% That is, $Z^\alpha$  is given as the unique process such that $\int_0^\cdot Z^\alpha_s dW^\alpha_s$ is a true $\P^\alpha$-martingale, $Z^\alpha_t \mathbf{1}_{\lbrace t\geq\tau\rbrace} = 0$, and
% \begin{eqnarray*}
%     &&Y^\alpha_0 + \int_0^t Z^\alpha_s dW^\alpha_s = \mathbb{E}^\alpha[\mathbf{1}_{\lbrace T < \tau \rbrace}(g(\theta^\alpha_T) + \tilde{\mathbb{E}}^\alpha[\frac{\delta g}{\delta m}(\tilde{\theta}^\alpha_T, X)\vert T < \tilde{\tau}] + \mathbb{E}^\alpha[\mathbf{1}_{\lbrace T < \tau \rbrace}g_p(\theta^\alpha_T)]) \\
%     &+& \int_0^{T\wedge\tau}f^1(s, X_{\cdot \wedge s}, \alpha_s) + f^2(s, \theta^\alpha_s) + \tilde{\mathbb{E}}^\alpha[\frac{\delta f^2}{\delta m}(s, \tilde{\theta}^\alpha_s, X_{\cdot \wedge s})\vert s<\tilde{\tau}] + \mathbb{E}^\alpha[\mathbf{1}_{\lbrace s<\tau \rbrace}f^2_p(s, \theta^\alpha_s)]ds \vert \mathcal{F}_t].
% \end{eqnarray*}
Since for $\alpha \in \mathbb{A}$ we know that $\int_0^{T\wedge \tau}f^1(s, X_{\cdot \wedge s}, \alpha_s)ds$ is merely integrable with respect to $\mathbb{P}^\alpha$, with the above construction we cannot a priori guarantee any sufficient integrability properties for $Z^\alpha$, and we cannot rely on the usual $\mathbb{L}^p$-theory of BSDEs. 
Therefore, we first focus on representation properties of $\hat{\alpha}$ and $Z^{\hat{\alpha}}$.
\begin{lemma}\label{lem: opt ncs}
    Let $\hat\alpha$ be optimal and $Z^{\hat{\alpha}}$ be as above. 
    Then, $dt \times \mathbb{P}$ a.e. on $\lbrace t < \tau\rbrace$, we have $\hat{\alpha}_t = \arg\min_{a \in A}h^1(t, X_{\cdot \wedge t}, a, Z^{\hat{\alpha}}_t)$.
\end{lemma}
\begin{proof}
    Just like in the proof of Theorem \ref{thm: exs}, let $\hat{\alpha}^N$ be the optimal control in $\mathbb{A}^N$. Let us write
    \begin{align*}
        \mathcal{Y} &:= \mathbf{1}_{\lbrace T < \tau \rbrace}\Big( g(\theta^{\hat{\alpha}}_T) + \tilde{\mathbb{E}}^{\hat{\alpha}}\big[\frac{\delta g}{\delta m}(\tilde{\theta}^{\hat{\alpha}}_T, X)\vert T < \tilde{\tau}\big] + \mathbb{E}^{\hat{\alpha}}[\mathbf{1}_{\lbrace T < \tau \rbrace}g_p(\theta^{\hat{\alpha}}_T)]\Big)\\
        & \quad+ \int_0^{T\wedge\tau}f^1(s, X_{\cdot \wedge s}, \hat{\alpha}_s) + f^2(s, \theta^{\hat{\alpha}}_s) 
        + \tilde{\mathbb{E}}^{\hat{\alpha}}\Big[\frac{\delta f^2}{\delta m}(s, \tilde{\theta}^{\hat{\alpha}}_s, X_{\cdot \wedge s})\vert s<\tilde{\tau}\Big] + \mathbb{E}^{\hat{\alpha}}[\mathbf{1}_{\lbrace s<\tau \rbrace}f^2_p(s, \theta^{\hat{\alpha}}_s)]ds
    \end{align*}
    and
    \begin{align*}
        \mathcal{Y}^N &= \mathbf{1}_{\lbrace T < \tau \rbrace}\Big(g(\theta^{\hat{\alpha}^N}_T) + \tilde{\mathbb{E}}^{\hat{\alpha}^N}\Big[\frac{\delta g}{\delta m}(\tilde{\theta}^{\hat{\alpha}^N}_T, X)\vert T < \tilde{\tau} \Big] + \mathbb{E}^{\hat{\alpha}^N}[\mathbf{1}_{\lbrace T < \tau \rbrace}g_p(\theta^{\hat{\alpha}^N}_T)]\Big)\\
        &\quad+ \int_0^{T\wedge\tau}f^1(s, X_{\cdot \wedge s}, \hat{\alpha}^N_s) + f^2(s, \theta^{\hat{\alpha}^N}_s) + \tilde{\mathbb{E}}^{\hat{\alpha}^N}\Big[\frac{\delta f^2}{\delta m}(s, \tilde{\theta}^{\hat{\alpha}^N}_s, X_{\cdot \wedge s})\vert s<\tilde{\tau}\Big] + \mathbb{E}^{\hat{\alpha}^N}[\mathbf{1}_{\lbrace s<\tau \rbrace}f^2_p(s, \theta^{\hat{\alpha}^N}_s)]ds.
    \end{align*}
    As described above, $Z^{\hat{\alpha}}$ and $Z^{\hat{\alpha}^N}$ are characterized via $\int_0^\cdot Z^{\hat{\alpha}}_s dW^{\hat{\alpha}}_s = \mathbb{E}^{\hat{\alpha}}[\mathcal{Y}\vert\mathcal{F}_\cdot] - \mathbb{E}^{\hat{\alpha}}[\mathcal{Y}\vert \mathcal{F}_0]$ and $\int_0^\cdot Z^{\hat{\alpha}^N}_s dW^{\hat{\alpha}^N}_s = \mathbb{E}^{\hat{\alpha}^N}[\mathcal{Y}^N\vert\mathcal{F}_\cdot] - \mathbb{E}^{\hat{\alpha}^N}[\mathcal{Y}^N\vert\mathcal{F}_0]$. We will first show that $Z^{\hat{\alpha}^N}$ converges $dt \times\mathbb{P}$ a.e.\ to $Z^{\hat{\alpha}}$ up to a subsequence.\par
    Recall that we have proven that $\mathcal{H}(\mathbb{P}^{\hat{\alpha}^N} \parallel \mathbb{P}^{\hat{\alpha}}) \rightarrow 0$ and particularly, $d_{\text{TV}}(\mathbb{P}^{\hat{\alpha}^N}, \mathbb{P}^{\hat{\alpha}}) \rightarrow 0$. 
    This also implies that the sequence $\mathbb{P}^{\hat{\alpha}^N}[T< \tau]$ converges to $\mathbb{P}^{\hat{\alpha}}[T<\tau]$ and is thus bounded away from zero in $N$. Based on our prior discussion, and as we have assumed our coefficients and their derivatives to be continuous, we can see that $\mathcal{Y}$ converges a.s.\ to $\mathcal{Y}^N$. Further, by Assumption \ref{asmp: cnv}.$(v)$, we can see that $\mathcal{Y}$ and $\mathcal{Y}^N$ can be uniformly bounded from below. By adding a common constant, we can thus assume that $\mathcal{Y}^N$ and $\mathcal{Y}$ are non-negative (this constant does not affect the $Z^{\hat{\alpha}^N}$ and $Z^{\hat{\alpha}}$ as it gets cancelled in the construction of these processes).
    \par

    Now, since we have shown previously that $J(\hat{\alpha}^N) \rightarrow J(\hat{\alpha})$ and since the remaining derivatives terms can be uniformly bounded in $N$, we can see that $\mathbb{E}^{\hat{\alpha}^N}[\mathcal{Y}^N] \rightarrow \mathbb{E}^{\hat{\alpha}}[\mathcal{Y}]$. As we have assumed non-negativity, by Scheffé's lemma, the sequence $(\frac{d\mathbb{P}^{\hat{\alpha}^N}}{d\mathbb{P}^{\hat{\alpha}}}\mathcal{Y}^N)_N$ thus converges in $\mathbb{L}^1$ with respect to $\mathbb{P}^{\hat{\alpha}}$ to $\mathcal{Y}$. By Lemma \ref{lem: sem top}, the martingales $\mathbb{E}^{\hat{\alpha}}[\frac{d\mathbb{P}^{\hat{\alpha}^N}}{d\mathbb{P}^{\hat{\alpha}}}\mathcal{Y}^N\vert \mathcal{F}_\cdot]$ converge to $\mathbb{E}^{\hat{\alpha}}[\mathcal{Y}\vert \mathcal{F}_\cdot]$ in the ucp topology. Analogously, we can see that the martingales $\mathbb{E}^{\hat{\alpha}}[\frac{d\mathbb{P}^{\hat{\alpha}^N}}{d\mathbb{P}^{\hat{\alpha}}}\vert \mathcal{F}_\cdot]$ converge to $1$ (the constant process valued at $1$) in ucp. By Lemma \ref{lem: sem top}, we can see that the processes $\mathbb{E}^{\hat{\alpha}^N}[\mathcal{Y}^N\vert \mathcal{F}_\cdot] = \frac{\mathbb{E}^{\hat{\alpha}}[\frac{d\mathbb{P}^{\hat{\alpha}^N}}{d\mathbb{P}^{\hat{\alpha}}}\mathcal{Y}^N\vert\mathcal{F}_\cdot]}{\mathbb{E}^{\hat{\alpha}}[\frac{d\mathbb{P}^{\hat{\alpha}^N}}{d\mathbb{P}^{\hat{\alpha}}}|\mathcal{F}_\cdot]}$ thus converge in ucp to $\mathbb{E}^{\hat{\alpha}}[\mathcal{Y}\vert\mathcal{F}_\cdot]$.
    %{\color{blue}the denominator should be a conditional expectation right?}
    \par

    This also show that the $\int_0^\cdot Z^{\hat{\alpha}^N}_s dW^{\hat{\alpha}^N}_s$ converge in ucp to $\int_0^\cdot Z^{\hat{\alpha}}_s dW^{\hat{\alpha}}_s$. Note that $\int_0^\cdot Z^{\hat{\alpha}^N}_s dW^{\hat{\alpha}^N}_s = \int_0^{\cdot}Z^{\hat{\alpha}^N}_sdW^{\hat{\alpha}}_s + \int_0^\cdot (\beta(s, X_{\cdot \wedge s}, \hat{\alpha}_s) - \beta(s, X_{\cdot \wedge s}, \hat{\alpha}^N_s))^\top Z^{\hat{\alpha}^N}_sds$. We can estimate
    \begin{align}
    \notag
        &\mathbb{E}^{\hat{\alpha}^N}\bigg[\bigg(\int_0^{T\wedge\tau}\vert  (\beta(s, X_{\cdot \wedge s}, \hat{\alpha}_s) - \beta(s, X_{\cdot \wedge s}, \hat{\alpha}^N_s) )^\top Z^{\hat{\alpha}^N}_s\vert ds\bigg)^{\frac{1}{4}}\bigg] \\\notag 
        &\quad \leq \mathbb{E}^{\hat{\alpha}^N}\bigg[\bigg( \int_0^{T\wedge\tau}\vert Z^{\hat{\alpha}^N}_s \vert^2ds\bigg)^{\frac{1}{8}} \bigg(\int_0^{T\wedge\tau}\Vert \beta(s, X_{\cdot \wedge s}, \hat{\alpha}_s) - \beta(s, X_{\cdot \wedge s}, \hat{\alpha}^N_s) \Vert^2 ds\bigg)^{\frac{1}{8}}\bigg]\\\label{eq:factor.proof.rep.alpha.bmo}
        &\quad \leq\mathbb{E}^{\hat{\alpha}^N}\bigg[\bigg(\int_0^{T\wedge\tau}\vert Z^{\hat{\alpha}^N}_s \vert^2 ds\bigg)^{\frac{1}{4}}\bigg]^{\frac{1}{2}}\mathbb{E}^{\hat{\alpha}^N}\bigg[\int_0^{T\wedge\tau}\Vert \beta(s, X_{\cdot \wedge s}, \hat{\alpha}_s) - \beta(s, X_{\cdot \wedge s}, \hat{\alpha}^N_s)\Vert^2ds\bigg]^{\frac{1}{8}}.
    \end{align}
    Up to a constant coming from the Burkholder-Davis-Gundy inequality, the first term can be bounded by $\mathbb{E}^{\hat{\alpha}^N}[\sup_{t \in [0, T]} \vert \mathbb{E}^{\hat{\alpha}^N}[\mathcal{Y}^N\vert \mathcal{F}_t] - \mathbb{E}^{\hat{\alpha}^N}[\mathcal{Y}^N\vert \mathcal{F}_0] \vert^{\frac{1}{2}}] \leq 2 \mathbb{E}^{\hat{\alpha}^N}[\vert \mathcal{Y}^N] - \mathbb{E}^{\hat{\alpha}^N}[\mathcal{Y}^N\vert \mathcal{F}_0] \vert]^{\frac{1}{2}}$ where the latter inequality uses \cite[Lemma 6.1]{Briand03}. 
    As the sequence $\mathbb{E}^{\hat{\alpha}^N}[\mathcal{Y}^N]$ is bounded in $N$, the first factor in \eqref{eq:factor.proof.rep.alpha.bmo} is thus bounded in $N$. The second factor coincides with $(2\mathcal{H}(\mathbb{P}^{\hat{\alpha}^N} \parallel \mathbb{P}^{\hat{\alpha}}))^\frac{1}{8}$ and thus converges to zero. 
    Thus, $\int_0^{T\wedge\tau}\vert  (\beta(s, X_{\cdot \wedge s}, \hat{\alpha}_s) - \beta(s, X_{\cdot \wedge s}, \hat{\alpha}^N_s) )^\top Z^{\hat{\alpha}^N}_s\vert ds$ converges to zero in probability from which we can infer that $\int_0^\cdot (\beta(s, X_{\cdot \wedge s}, \hat{\alpha}_s) - \beta(s, X_{\cdot \wedge s}, \hat{\alpha}^N_s))^\top Z^{\hat{\alpha}^N}_sds$ converges to 0 in ucp. Together, this shows that $\int_0^\cdot Z^{\hat{\alpha}^N}_s - Z^{\hat{\alpha}}_s dW^{\hat{\alpha}}_s$ converges to 0 in ucp. Thus, by Lemma \ref{lem: sem top}, up to a subsequence, $Z^{\hat{\alpha}^N}_t$ converges $dt \times \mathbb{P}$ a.e.\ to $Z^{\hat{\alpha}}_t$.\par

    This convergence result allows to show that $\hat{\alpha}_t$ minimizes the Hamiltonian.
    In fact, Since $h^1$ is assumed to be convex and differentiable in $a$, the desired statement is equivalent to showing that for any $\mathbb{R}^k$-valued $\gamma$ that is uniformly bounded and $\hat{\alpha}_t + \gamma_t \in A$ $dt \times \mathbb{P}$ a.e.\ on $\lbrace t < \tau\rbrace$, we have $\gamma_t^\top h^1_a(t, X_{\cdot \wedge t}, \hat{\alpha}_t, Z^{\hat{\alpha}}_t)\geq 0$ $dt \times \mathbb{P}$ a.e.\ on $\lbrace t < \tau\rbrace$. In the following, all statements will only apply to the event $\lbrace t < \tau \rbrace$, even if not explicitly specified.\par
    Given such a $\gamma$, let us define $\gamma^N_t := \Pi_{A^N}(\hat{\alpha}^N_t + \gamma_t)- \hat{\alpha}^N_t$ where for any closed convex set $K$, $\Pi_K$ denotes the orthogonal projection onto $K$. We already know that $\hat{\alpha}^N_t$ converges a.e.\ to $\hat{\alpha}_t$, therefore, a.e., the $\Vert \alpha^N_t \Vert$ are pointwise bounded in $N$. Since $\Pi_{A}$ is Lipschitz, $\Vert \Pi_{A}(\hat{\alpha}^N_t + \gamma_t) - \hat{\alpha}^N_t\Vert \leq \Vert \gamma_t \Vert$ and for $N$ large enough, $\Vert \Pi_{A}(\hat{\alpha}^N_t + \gamma_t) \Vert \leq N$ which implies $\Pi_{A}(\hat{\alpha}^N_t + \gamma_t) = \Pi_{A^N}(\hat{\alpha}^N_t + \gamma_t)$. Therefore, the sequence $\Pi_{A^N}(\hat{\alpha}^N_t + \gamma_t)$ converges a.e.\ to $\Pi_A(\hat{\alpha}_t + \gamma_t) = \hat{\alpha}_t + \gamma_t$ and the sequence $\gamma^N_t$ converges a.e.\ to $\gamma_t$.\par
    By construction, $\hat{\alpha}^N_t + \gamma^N_t \in A^N$ a.e.\ and thus by Theorem \ref{thm: ncs}, we have$ (\gamma^N_t)^\top h^1_a(t, X_{\cdot \wedge t}, \hat{\alpha}^N_t, Z^{\hat{\alpha}^N}_t) \geq 0$ a.s. As $h^1_a$ is continuous in $a$ and $z$, and as the inner product is continuous, this thus shows that a.e., $\gamma_t^\top h^1_a(t, X_{\cdot \wedge t}, \hat{\alpha}_t, Z^{\hat{\alpha}}_t) \geq 0$. 
    %{\color{blue}how do I know that $\hat\alpha^N + \gamma^N$ is optimal for the problem with values in $A^N$?}
\end{proof}

This optimality property allows to show the desired integrability of $\hat{\alpha}$ and $(Y^{\hat{\alpha}}, Z^{\hat{\alpha}})$. To do so, let us introduce the functions $a^*:[0, T] \times \mathcal{C} \times \mathbb{R}^d \rightarrow A, (t, x, z) \mapsto \arg \min_{a \in A} h^1(t, x, a, z)$, and $a^*_N:[0, T] \times \mathcal{C} \times \mathbb{R}^d \rightarrow A^N, (t, x, z) \mapsto \arg \min_{a \in A^N} h^1(t, x, a, z)$. Using strong convexity, we can see that for any $t$, $x$, $z$ and $z'$,
\begin{eqnarray*}
    &&m\Vert a^*(t, x, z') - a^*(t, x, z)\Vert^2\\
    &\leq& h^1(t, x, a^*(t, x, z), z') - h^1(t, x, a^*(t, x, z'), z') + h^1(t, x, a^*(t, x, z'), z) - h^1(t, x, a^*(t, x, z), z)\\
    &=& (\beta^1(t, x)(a^*(t, x, z) - a^*(t, x, z')))^\top(z - z') \leq L \Vert a^*(t, x, z') - a^*(t, x, z)\Vert \Vert z-z'\Vert
\end{eqnarray*}
showing that $a^*$ is $\frac{L}{m}$ Lipschitz with respect to $z$. The same applies to the $a^*_N$.
\begin{lemma}\label{lem: opt BMO}
    We have $\Vert \sup_{t \in [0, T]} \vert Y^{\hat{\alpha}}_t \vert \Vert_\infty + \Vert \int_0^{\cdot} Z^{\hat{\alpha}}_s dW^{\hat{\alpha}}_s \Vert_{BMO} < \infty$. In particular, $\hat{\alpha} \in \mathbb{A}_{BMO}$.
\end{lemma}
\begin{proof}%[Proof of Corollary \ref{lem: opt BMO}]
    Let us first find a bound for $(Y^{\hat{\alpha}^N}, Z^{\hat{\alpha}^N})$ that is independent of $N$. By Proposition \ref{prop: wp adj}, we know $Y^{\hat{\alpha}^N}$ and $Z^{\hat{\alpha}^N}$ are square integrable with respect to $\mathbb{P}^{\hat{\alpha}^N}$, and by Theorem \ref{thm: ncs}, $\hat{\alpha}^N_t = a^*_N(t, X_{\cdot \wedge t}, Z^{\hat{\alpha}^N}_t)$.
     We can thus also consider $(Y^{\hat{\alpha}^N}, Z^{\hat{\alpha}^N})$ as the unique solution of the BSDE
    \begin{align}\notag
        Y^{\hat{\alpha}^N}_{t\wedge \tau} &= \mathbf{1}_{\lbrace T < \tau \rbrace}\Big( g(\theta^{\hat{\alpha}^N}_T) + \tilde{\mathbb{E}}^{\hat{\alpha}^N}\Big[\frac{\delta g}{\delta m}(\tilde{\theta}^{\hat{\alpha}^N}_T, X)\vert T < \tilde{\tau}\Big] + \mathbb{E}^{\hat{\alpha}^N}[\mathbf{1}_{\lbrace T < \tau \rbrace}g_p(\theta^{\hat{\alpha}^N}_T)]\Big)\\\notag
        &\quad+ \int_{t\wedge\tau}^{T\wedge\tau}f^1(s, X_{\cdot \wedge s}, a^*_N(s, X_{\cdot \wedge s}, Z^{\hat{\alpha}^N}_s)) + f^2(s, \theta^{\hat{\alpha}^N}_s) + \tilde{\mathbb{E}}^{\hat{\alpha}^N}\Big[\frac{\delta f^2}{\delta m}(s, \tilde{\theta}^{\hat{\alpha}^N}_s, X_{\cdot \wedge s})\vert s<\tilde{\tau}\Big]\\\label{eq: Lips no MKV BSDE}
    &\quad\qquad + \mathbb{E}^{\hat{\alpha}^N}\big[\mathbf{1}_{\lbrace s<\tau \rbrace}f^2_p(s, \theta^{\hat{\alpha}^N}_s)\big]ds
    -\int_{t\wedge\tau}^{T\wedge\tau} Z^{\hat{\alpha}^N}_s dW^{\hat{\alpha}^N}_s\quad \mathbb{P}^{\hat{\alpha}^N}\text{-a.s.}
\end{align}
    which is a standard non McKean-Vlasov BSDE with $\mathbb{P}^{\hat{\alpha}^N}$ and $\theta^{\hat{\alpha}^N}$ kept fixed and only $Y^{\hat{\alpha}^N}$ and $Z^{\hat{\alpha}^N}$ as part of the solution. Since $a^*_N$ is Lipschitz but also valued in a bounded set, we can see that the driver of the above BSDE is Lipschitz in $Z$. Therefore, the process $\kappa_t := \mathbf{1}_{\lbrace T < \tau \rbrace}\mathbf{1}_{\lbrace Z^{\hat{\alpha}^N}_t \neq 0 \rbrace} \frac{f^1(s, X_{\cdot}, a^*_N(s, X_{\cdot \wedge s}, Z^{\hat{\alpha}}_s)) - f^1(s, X_{\cdot}, a^*_N(s, X_{\cdot \wedge s}, 0))}{\Vert Z^{\hat{\alpha}^N}_s\Vert^2} Z^{\hat{\alpha}^N}_s$, is uniformly bounded. In particular, we can consider the equivalent probability measure $ \tilde{\mathbb{P}}^{\hat\alpha^N} :=  \mathbb{P}^{\hat{\alpha}^N+\kappa}$ that also satisfies $\frac{d\tilde{\mathbb{P}}^{\hat{\alpha}^N}}{d\mathbb{P}^{\hat{\alpha}^N}} = \mathcal{E}(\int_0^{\cdot \wedge \tau} \kappa_s dW^{\hat{\alpha}^N}_s)_T$. Since $Z^{\hat{\alpha}^N}$ is square integrable with respect to $dt \times \mathbb{P}^{\hat{\alpha}^N}$, we have that $\int_0^\cdot Z^{\hat{\alpha}^N}_s dW^{\hat{\alpha}^N + \kappa}_s$ is a true $\mathbb{P}^{\hat{\alpha}^N+\kappa}$ martingale, and
    \begin{align*}
        Y^{\hat{\alpha}^N}_{t\wedge\tau} &= \mathbb{E}^{\hat{\alpha}^N + \kappa}\bigg[ \mathbf{1}_{\lbrace T < \tau \rbrace}\Big( g(\theta^{\hat{\alpha}^N}_T) + \tilde{\mathbb{E}}^{\hat{\alpha}^N}\Big[\frac{\delta g}{\delta m}(\tilde{\theta}^{\hat{\alpha}^N}_T, X)\vert T < \tilde{\tau}\Big] + \mathbb{E}^{\hat{\alpha}^N}[\mathbf{1}_{\lbrace T < \tau \rbrace}g_p(\theta^{\hat{\alpha}^N}_T )]\Big)\\
        &\quad +\int_{t\wedge\tau}^{T\wedge\tau}f^1(s, X_{\cdot \wedge s}, a^*_N(s, X_{\cdot \wedge s}, 0)) + f^2(s, \theta^{\hat{\alpha}^N}_s) + \tilde{\mathbb{E}}^{\hat{\alpha}^N}\Big[\frac{\delta f^2}{\delta m}(s, \tilde{\theta}^{\hat{\alpha}^N}_s, X_{\cdot \wedge s})\vert s<\tilde{\tau}\Big] \\
        &\quad \qquad + \mathbb{E}^{\hat{\alpha}^N}[\mathbf{1}_{\lbrace s<\tau \rbrace}f^2_p(s, \theta^{\hat{\alpha}^N}_s)]ds \vert \mathcal{F}_{t\wedge \tau}\bigg].
    \end{align*}
    Recall that $f^1$ is bounded from below, and further, $f^1(s, X_{\cdot \wedge s}, a^*_N(s, X_{\cdot\wedge s}, 0)) \leq f^1(s, X_{\cdot \wedge s}, 0)$ showing that $\vert f^1(s, X_{\cdot \wedge s}, a^*_N(s, X_{\cdot\wedge s}, 0))\vert$ is uniformly bounded independent of $N$. Further since the sequence $\mathbb{P}^{\hat{\alpha}^N}[T < \tau]$ converges to $ \mathbb{P}^{\hat{\alpha}}[T<\tau] > 0$ and are thus bounded away from $0$, the remaining terms are also bounded in $N$ by Assumption \ref{asmp: cnv}. Therefore, we can also see that $\vert Y^{\hat{\alpha}^N}\vert$ is $dt\times \mathbb{P}$ a.e.\ bounded by a constant independent of $N$.\par
    For a uniform bound of the $Z^{\hat{\alpha}}$, note that we can also interpret \eqref{eq: Lips no MKV BSDE} as a quadratic BSDE with $\vert f^1(t, x, a^*_N(t, x, z)) - f^1(s, x, a^*_N(s, x, z'))\vert \leq \frac{M^1 L}{m}\Vert z - z' \Vert(1 + \Vert a^*_N(t, x, 0)\Vert + \frac{L}{2m}(\Vert z \Vert + \Vert z'\Vert))$ with $\Vert a^*_N(t, x, 0) \Vert^2 \leq \frac{4}{m}(f^1(s, x, 0) - f^1(s, x, a^*_N(s, x, 0)))$ being bounded over all $N$ using \cite[Lemma 2.4]{Z24}. Therefore, by \cite[Theorem 7.2.1]{ZhangJianfeng}, there exists a constant $C$ independent of $N$ such that $\mathbb{E}^{\hat{\alpha}^N}[\int_0^{T\wedge \tau}\Vert Z^{\hat{\alpha}^N}_s \Vert^2 ds] \leq C$.\par    
    From the last proof, recall that for each $t \geq 0$, up to a subsequence, we have that $\int_0^{t\wedge \tau} Z^{\hat{\alpha}^N}_sdW^{\hat{\alpha}^N}_s$ converges a.e.\ to $\int_0^{t\wedge \tau} Z^{\hat{\alpha}}_sdW^{\hat{\alpha}}_s$, as well as $Y^{\hat{\alpha}^N}_0 \rightarrow Y^{\hat{\alpha}}_0$ a.s.\  Further, as established in the proof of Lemma \ref{lem: opt ncs} (using notation therein) we have that $\frac{\mathbb{P}^{\hat{\alpha}^N}}{d\mathbb{P}^{\hat{\alpha}}} \mathcal{Y}^N$ converges in $\mathbb{L}^1(\mathbb{P}^{\hat{\alpha}})$ to $\mathcal{Y}$. Using that our coefficients are all bounded from below, we have by Vitali's theorem that
    $$\frac{d\mathbb{P}^{\hat{\alpha}^N}}{d\mathbb{P}^{\hat{\alpha}}} \int_0^{t\wedge \tau}f^1(s, X_{\cdot \wedge s}, \hat{\alpha}^N_s) +  f^2(s, \theta^{\hat{\alpha}^N}_s) + \tilde{\mathbb{E}}^{\hat{\alpha}^N}\Big[\frac{\delta f^2}{\delta m}(s, \tilde{\theta}^{\hat{\alpha}^N}_s, X_{\cdot \wedge s})\vert s<\tilde{\tau} \Big] + \mathbb{E}^{\hat{\alpha}^N}[\mathbf{1}_{\lbrace s<\tau \rbrace}f^2_p(s, \theta^{\hat{\alpha}^N}_s)]ds$$ converges in $\mathbb{L}^1(\mathbb{P}^{\hat{\alpha}})$ to
    $$\int_0^{t\wedge\tau}f^1(s, X_{\cdot \wedge s}, \hat{\alpha}_s) + f^2(s, \theta^{\hat{\alpha}}_s) + \tilde{\mathbb{E}}^{\hat{\alpha}}\Big[\frac{\delta f^2}{\delta m}(s, \tilde{\theta}^{\hat{\alpha}}_s, X_{\cdot \wedge s})\vert s<\tilde{\tau} \Big] + \mathbb{E}^{\hat{\alpha}}[\mathbf{1}_{\lbrace s<\tau \rbrace}f^2_p(s, \theta^{\hat{\alpha}}_s)]ds.$$ 
    This further implies that the sequence
    $$\int_0^{t\wedge \tau}f^1(s, X_{\cdot \wedge s}, \hat{\alpha}^N_s) +  f^2(s, \theta^{\hat{\alpha}^N}_s) + \tilde{\mathbb{E}}^{\hat{\alpha}^N}\Big[\frac{\delta f^2}{\delta m}(s, \tilde{\theta}^{\hat{\alpha}^N}_s, X_{\cdot \wedge s})\vert s<\tilde{\tau} \Big] + \mathbb{E}^{\hat{\alpha}^N}[\mathbf{1}_{\lbrace s<\tau \rbrace}f^2_p(s, \theta^{\hat{\alpha}^N}_s)]ds$$
    converges up to a subsequence a.s.\ to
    $$\int_0^{t\wedge\tau}f^1(s, X_{\cdot \wedge s}, \hat{\alpha}_s) + f^2(s, \theta^{\hat{\alpha}}_s) + \tilde{\mathbb{E}}^{\hat{\alpha}}\Big[\frac{\delta f^2}{\delta m}(s, \tilde{\theta}^{\hat{\alpha}}_s, X_{\cdot \wedge s})\vert s<\tilde{\tau} \Big] + \mathbb{E}^{\hat{\alpha}}[\mathbf{1}_{\lbrace s<\tau \rbrace}f^2_p(s, \theta^{\hat{\alpha}}_s)]ds$$
    so that together, we can see that $Y^{\hat{\alpha}^N}_{t\wedge \tau}$ converges a.s.\ to $Y^{\hat{\alpha}}_{t\wedge\tau}$ up to a subsequence. By our previously shown bound, this implies that $\Vert \sup_{0\leq t \leq \tau}\vert Y^{\hat{\alpha}}_t \vert \Vert_\infty<\infty$.\par
    Again, recall from the proof of Lemma \ref{lem: opt ncs} that $Z^{\hat{\alpha}^N}_t$ converges $dt \times \mathbb{P}$ a.e. to $Z^{\hat{\alpha}}$ and also that $d_{\mathrm{TV}}(\mathbb{P}^{\hat{\alpha}}, \mathbb{P}^{\hat{\alpha}^N}) \rightarrow 0$. Thus, by Fatou's lemma, we have $\mathbb{E}^{\hat{\alpha}}[\int_0^{T\wedge\tau}\Vert Z^{\hat{\alpha}}_s \Vert^2 ds] \leq C$. Now, $(Y^{\hat{\alpha}}, Z^{\hat{\alpha}})$ can also be considered as a solution of the BSDE
    \begin{align*}
        Y^{\hat{\alpha}}_{t\wedge \tau} &= \mathbf{1}_{\lbrace T < \tau \rbrace}\Big( g(\theta^{\hat{\alpha}}_T) + \tilde{\mathbb{E}}^{\hat{\alpha}}\Big[ \frac{\delta g}{\delta m}(\tilde{\theta}^{\hat{\alpha}}_T, X)\vert T < \tilde{\tau}\Big] + \mathbb{E}^{\hat{\alpha}}[\mathbf{1}_{\lbrace T < \tau \rbrace}g_p(\theta^{\hat{\alpha}}_T)]\Big)\\
        &\quad + \int_{t\wedge\tau}^{T\wedge\tau}f^1(s, X_{\cdot \wedge s}, a^*(s, X_{\cdot \wedge s}, Z^{\hat{\alpha}}_s)) + f^2(s, \theta^{\hat{\alpha}}_s) + \tilde{\mathbb{E}}^{\hat{\alpha}}\Big[\frac{\delta f^2}{\delta m}(s, \tilde{\theta}^{\hat{\alpha}}_s, X_{\cdot \wedge s})\vert s<\tilde{\tau}\Big]\nonumber\\
        &\quad + \mathbb{E}^{\hat{\alpha}}[\mathbf{1}_{\lbrace s<\tau \rbrace}f^2_p(s, \theta^{\hat{\alpha}}_s)]ds
        -\int_{t\wedge\tau}^{T\wedge\tau} Z^{\hat{\alpha}}_s dW^{\hat{\alpha}}_s\nonumber
    \end{align*}
    again considered as a standard non McKean-Vlasov BSDE with $\mathbb{P}^{\hat{\alpha}}$ and $\theta^{\hat{\alpha}}$ kept fixed. Using the same argument as above, this is again a quadratic BSDE and by \cite[Theorem 7.2.1]{ZhangJianfeng}, we have $\Vert \int_0^\cdot Z_s dW^{\hat{\alpha}_s}_s \Vert_{\mathbb{P}^\alpha\text{-BMO}} < \infty$. 
    As $a^*$ is Lipschitz, this shows $\Vert \hat{\alpha} \Vert_{\mathbb{P}^\alpha \text{- BMO}} < \infty$ and consequently using \cite[Theorem 3.6.]{Kazamaki}, also finiteness of the BMO norms with respect to $\mathbb{P}$.
\end{proof}

\subsection{Approximation of non equivalent changes of measures}

\begin{proof}[Proof of Theorem \ref{thm: opt non eq}]
    Let $\alpha \in \tilde{\mathbb{A}}$ be arbitrary. For any $n \geq 1$, we can consider $\mathbb{P}^n := (1 - \frac{1}{n})  \mathbb{P}^\alpha + \frac{1}{n}\mathbb{P}$. Like in our previous arguments, by defining $\alpha^n_t = (1-\frac{1}{n})\frac{d\mathbb{P}^\alpha_{\vert \mathcal{F}_t}}{d\mathbb{P}^n_{\vert\mathcal{F}_t}}\alpha_t$ as a pointwise convex combination, we see that $\mathbb{P}^{\alpha^n} = \mathbb{P}^n$. 
    Further, 
    \begin{equation}\label{eq: no eqv cnv}
        \mathbb{E}^{n}\bigg[\int_0^{T\wedge\tau}f^1(s, X_{\cdot \wedge s}, \alpha^n_s)ds \bigg] \leq (1-\frac{1}{n}) \mathbb{E}^\alpha\bigg[\int_0^{T\wedge\tau}f^1(s, X_{\cdot \wedge s}, \alpha_s)ds\bigg] + \frac{1}{n}\mathbb{E}\bigg[ \int_0^{T\wedge\tau}f^1(s, X_{\cdot\wedge \tau}, 0)ds \bigg]< \infty.
    \end{equation}
    One can thus see that $\alpha^n \in \mathbb{A}$.\par
    Following \cite[(30)]{Leonard}, we have $\mathcal{H}(\mathbb{P}^\alpha \parallel \mathbb{P}^n) \rightarrow 0$. Thus,
    \begin{eqnarray*}
        &&\int_0^T F(s, \mathcal{L}_{\mathbb{P}^{\alpha^n}}(X_{\cdot \wedge s}\vert s<\tau), \mathbb{P}^{\alpha^n}[s<\tau])ds + G(\mathcal{L}_{\mathbb{P}^{\alpha^n}}(X\vert T<\tau), \mathbb{P}^{\alpha^n}[T<\tau])ds\\
        &\rightarrow& \int_0^T F(s, \mathcal{L}_{\mathbb{P}^\alpha}(X_{\cdot \wedge s}\vert s<\tau), \mathbb{P}^\alpha[s<\tau])ds + G(\mathcal{L}_{\mathbb{P}^\alpha}(X\vert T<\tau), \mathbb{P}^\alpha[T<\tau])ds.
    \end{eqnarray*}
    Further, \eqref{eq: no eqv cnv} implies that $\mathbb{E}^\alpha[\int_0^{T\wedge\tau}f^1(s, X_{\cdot\wedge s}, \alpha_s)ds] \geq \limsup_{n \rightarrow \infty}\mathbb{E}^n[\int_0^{T\wedge \tau} f^1(s, X_{\cdot \wedge s}, \alpha^n_s)ds]$ and combined with Fatou's lemma thus $\mathbb{E}^\alpha[\int_0^{T\wedge \tau} f^1(s, X_{\cdot \wedge s}, \alpha_s)ds] = \lim_{n\rightarrow\infty} \mathbb{E}^n[\int_0^{T \wedge \tau} f^1(s, X_{\cdot \wedge s}, \alpha^n_s)ds]$. In particular, this shows $\tilde{J}(\alpha) \geq \tilde{J}(\hat{\alpha})$.\par
    To see uniqueness of the optimizer, we see that $\tilde{J}(\alpha) = \tilde{J}(\hat{\alpha})$ would imply that each $\alpha^n$ is an optimizer as well. As each $\alpha^n$ is already in $\mathbb{A}$, this shows $\alpha^n = \hat{\alpha}$ and thus $\alpha = \hat{\alpha}$.
\end{proof}

%\section{Control with target constraints in the law via penalization}

\section{Applications: Schrödinger problem and mean field games}\label{sct: aplc}

%We only provided the proof in the case where $A$ is unbounded. The case where $A$ is bounded follows from the same steps as the proof below with similar arguments.

Let us first discuss the application to Schrödinger problems with hard killing.
\subsection{Proof of Theorem  \ref{thm: trgt}}
    For each $l\ge1$, the problem \eqref{eq:def.schro.penal} is a standard (conditional) McKean-Vlasov optimal control with running cost $f(t,x,\mu,p) = \frac12|a|^2$ and terminal cost $g^l(t,x,\mu,p) = \frac{l}{2}\Vert \hat p\hat \mu - p\mu\Vert^{2}_{-s}$.
    By the conditions on $b$, the assumptions \ref{asmp: beta}, \ref{asmp: ncs} and \ref{asmp: cnv} hold. 
    In particular, as discussed in Subsection \ref{sct: FW mtrc}, $g^l$ is $p$-convex.
    Thus, by Theorem \ref{thm: exs}, the problem \eqref{eq:def.schro.penal} admits a unique minimizer $\hat\alpha^l \in \mathbb{A}_{\mathrm{BMO}}$.
    Moreover, if $\sup_lV^l_{\hat p,\hat\mu}<\infty$,
    then since $g^l\ge0$, we have $\sup_l \mathcal{H}(\mathbb{P}^{\hat{\alpha}^l} \parallel \mathbb{P}) < \infty$.
    By Lemma \ref{lem: drft cstr}, there thus exists an $A$-valued process $\hat{\alpha}$ such that up to a subsequence, the $\mathbb{P}^{\hat{\alpha}^l}$ converge setwise to $\mathbb{P}^{\hat{\alpha}}$.
    We will keep these notation throughout the proof.

\begin{proof}[Proof of (i)]
    If\ $V_{\hat{p}, \hat{q}} < \infty$, then for any $l$, we have $V^l_{\hat{p}, \hat{\mu}} \leq V_{\hat{p}, \hat{\mu}}$ and it is easy to check that $V^l_{\hat{p}, \hat{\mu}} \nearrow V_{\hat{p}, \hat{\mu}}$.

    Assume \ $V_{\hat{p}, \hat{\mu}} = \infty$. If we did not also have $V^l_{\hat{p}, \hat{\mu}}\nearrow \infty$, then we would have $\sup_l V^l_{\hat{p}, \hat{\mu}} < \infty$ and thus $\sup_l \mathcal{H}(\mathbb{P}^{\hat{\alpha}^l} \parallel \mathbb{P}) < \infty$. 
    Thus the construction above gives a sequence of optimal controls $\hat\alpha^l$ for \eqref{eq:def.schro.penal} such that $\mathbb{P}^{\hat{\alpha}^l}$ converge setwise to $\mathbb{P}^{\hat{\alpha}}$.
    Using continuity of the Fourier Wasserstein distance and $\sup_l V^l_{\hat{p}, \hat{\mu}} < \infty$ gives that $\big\Vert \hat{p}\hat{\mu} - \mathbb{P}^{\hat\alpha}[T<\tau]\mathcal{L}_{\mathbb{P}^{\hat\alpha}}(X_T\vert s<\tau)\big\Vert^2_{-s} = 0$.
    In particular, $\hat\alpha$ is feasible, which contradicts $V_{\hat p,\hat\mu}=\infty$.
%    Thus, by Lemma \ref{lem: drft cstr} there is a limit probability $\mathbb{P}^{\hat{\alpha}}$ and it is then straightforward to show that 
% implies that $\hat{\alpha}$ is feasible 
    \end{proof}

%  As $f$ and $g$ are bounded from below and the terminal condition is always non negative, this thus implies $\sup_l \mathcal{H}(\mathbb{P}^{\hat{\alpha}^l} \parallel \mathbb{P}) < \infty$. By Lemma \ref{lem: drft cstr}, there exists an $A$ valued process $\hat{\alpha}$ such that up to a subsequence, the $\mathbb{P}^{\hat{\alpha}^l}$ converge setwise to $\mathbb{P}^{\hat{\alpha}}$.\par
\begin{proof}[Proof of (ii)]
    Let us show that $\hat\alpha$ is the unique minimizer in \eqref{eq:def.schro.hard.kill}.
    Let $\alpha^0$ be an arbitrary feasible control in $\tilde{\mathbb{A}}$.
    Although $\alpha^0$ is not assumed to necessarily lie in $\mathbb{A}$, since $\mathbb{P}^{\alpha^0}[T<\tau] = \hat{p} > 0$ by assumption, $J^j(\alpha^0)$ is still well defined and we can still apply Theorem \ref{thm: opt non eq} to see that $J(\alpha^0) = J^j(\alpha^0) \geq J^j(\hat{\alpha}^j)$. In particular, this shows 
    \begin{equation}
    \label{eq:chain.ineq.schro}
        J^l(\hat{\alpha}) \leq \lim_{i\rightarrow\infty} J^i(\hat{\alpha}^i) \leq J(\alpha^0) < \infty.
    \end{equation}
    Let $\tilde{\alpha}^l$ be  constructed as in \eqref{eq:def.tilde.alpha.lemma} in the proof of Lemma \ref{lem: drft cstr}. Since the sequence $(\tilde{\alpha}^l)_{l\ge0}$ $dt \times \mathbb{P}$ converges a.e.\ to $\hat{\alpha}$, and as $\mathbb{P}^{\tilde{\alpha}^l}$ converges to $\mathbb{P}^{\hat{\alpha}}$ in total variation, for any $l \geq 0$, by Fatou's lemma, we have $J^l(\hat{\alpha}) \leq \liminf_{i \rightarrow \infty} J^l(\tilde{\alpha}^i)$. Since $\mathbb{P}^{\tilde{\alpha}^i} = \sum_{j \geq i} \lambda^{i, j} \mathbb{P}^{\hat{\alpha}^j}$ is given by a finite convex combination for suitable $\lambda^{i, j} \in [0, 1]$ (which, just as in Lemma \ref{lem: drft cstr} are found using Mazur's lemma \cite[Lemma 10.19]{Renardy}), we can see that $\liminf_{i \rightarrow \infty} J^l(\tilde{\alpha}^i) \leq \liminf_{i \rightarrow \infty} \sum_{j \geq i} \lambda^{i, j} J^l(\hat{\alpha}^j) \leq \liminf_{i \rightarrow \infty} \sum_{j \geq i} \lambda^{i, j} J^j(\hat{\alpha}^j) \leq \lim_{i \rightarrow \infty} J^i(\hat{\alpha}^i)$. The first inequality is a consequence of the convexity of $J^l$ as in Proposition \ref{prop: J cnv} and Remark \ref{rmk: gnr cnv}, and the latter two follow as $J^l \leq J^j$ as soon as $j \leq l$, as well as that $J^i(\hat{\alpha}^i)$ is increasing in $i$.\par
    As the sequence $(J^l(\hat{\alpha}))_{l\ge 1}$ bounded and $g^l$ is lower semicontinuous, it follows that $\Vert \hat{p}\hat{\mu} - \mathbb{P}^{\hat\alpha}[T<\tau]\mathcal{L}_{\mathbb{P}^{\hat\alpha}}(X_T\vert s<\tau)\Vert^2_{-s}=0$ and thus, $\hat{\alpha}$ is feasible. Consequently, $J(\hat{\alpha}) = J^l(\hat{\alpha}) \leq J(\alpha^0)$, and since the admissible control $\alpha^0$ was taken arbitrary, this shows that $\hat{\alpha}$ is optimal for $V_{\hat{p}, \hat{q}}$.\par
    Uniqueness of $\hat{\alpha}$ is an immediate consequence of the strong convexity of $J$ proven in Proposition \ref{prop: J cnv} and Remark \ref{rmk: gnr cnv}, since for any two optimal $\hat{\alpha}^1$ and $\hat{\alpha}^2$, there will be a control $\hat{\alpha}^*$ such that $\mathbb{P}^{\hat{\alpha}^*} = \frac{\mathbb{P}^{\hat{\alpha}^1} + \mathbb{P}^{\hat{\alpha}^2}}{2}$ so that $\hat{\alpha}^*$ must also be feasible, but will actually admit a strictly smaller value for $J$ if $\hat{\alpha}^1 \neq \hat{\alpha}^2$.\par
    
    Let us now show convergence of $(\hat{\alpha}^l)_{l\ge1}$. From \eqref{eq:chain.ineq.schro},  choosing $\alpha^0 = \hat{\alpha}$ shows that $J(\hat{\alpha}) = \lim_{l \rightarrow \infty} J^l(\hat{\alpha}^l)$. Now, by Theorem \ref{thm: sfc}, the approximation in Theorem \ref{thm: opt non eq}, and Fatou's lemma, we have 
    \begin{equation}
    \label{eq:conv.control.Schroed}
        \mathcal{H}(\mathbb{P}^{\hat{\alpha}} \parallel \mathbb{P}^{\hat{\alpha}^l}) \leq \frac{L^2}{2}\mathbb{E}^{\hat{\alpha}}\bigg[\int_0^{T\wedge \tau}\Vert \hat{\alpha}^l_s - \hat{\alpha}_s \Vert^2 ds\bigg] \leq L^2(J(\hat{\alpha}) - J^l(\hat{\alpha}^l)) \rightarrow 0.
    \end{equation}
 \end{proof}
\begin{proof}[Proof of (iii)]
    When there is a feasible $\alpha^0 \in \mathbb{A}$, we can use the approximation argument from Theorem \ref{thm: opt non eq} and Theorem \ref{thm: sfc} to see that $\mathcal{H}(\mathbb{P}^{\alpha^0} \parallel \mathbb{P}^{\hat{\alpha}^l}) \leq L^2(J^l(\alpha^0) - J^l(\hat{\alpha}^l)) \leq L^2(J(\alpha^0) - J^1(\hat{\alpha}^1))$ is uniformly bounded from above. Therefore $\mathcal{H}(\mathbb{P}^{\alpha^0} \Vert \mathbb{P}^{\hat{\alpha}}) < \infty$, and $\mathbb{P}^{\hat{\alpha}}$ is equivalent to $\mathbb{P}$ so that $\hat{\alpha} \in \mathbb{A}$.
\end{proof}
\begin{proof}[Proof of (iv)]
        By part (ii) of the statement, 
        %Theorem \ref{thm: trgt} and Remark \ref{rmk: trgt}, 
        there exists a unique feasible $\hat{\alpha} \in \tilde{\mathbb{A}}$ that minimizes $J$ 
        %as in \eqref{eq: schrcnt} 
         over all controls in $\tilde{\mathbb{A}}$ that satisfy the target constraints.
        Using \cite[Theorem 2]{Leonard}, we can see that $\hat{\mathbb{P}}:= \mathbb{P}^{\hat{\alpha}}$ is the unique minimizer of \eqref{eq:def.schro.hard.kill} on $\mathcal{F}_{T\wedge\tau}$.
         To represent the density of $\hat{\mathbb{P}}$, we use the fact that $\hat{\alpha}$ satisfies \eqref{eq:conv.control.Schroed} where $\hat\alpha^l$ is optimal for \eqref{eq:def.schro.penal}.
        Thus, by
        % constructed as the limit as follows. For $k\geq 1$, let $\alpha^k \in \mathbb{A}_{BMO}$ be the unique optimal control for the penalized cost
     %$J^k(\alpha) = J(\alpha) + kG(\mathcal{L}_{\mathbb{P}^\alpha}(X_T \vert T<\tau), \mathbb{P}^{\alpha}[T<\tau])$ with $G(\mu, p):=\frac{1}{2} \Vert \mathbb{P}^\alpha[T<\tau] p\mu- \hat{p} \mu_{fin} \Vert^2_s$ where $\Vert \cdot \Vert_{-s}$ denotes the Fourier-Wasserstein metric of order $s > \frac{d}{2}$ discussed in Section \ref{sct: FW mtrc}. To adapt to our previous notation, we also write $g(\mu, p) = \frac{1}{p}G(\mu, p)$. By Theorem \ref{thm: exs} and \ref{thm: opt non eq}, such $\alpha^k$ exists and is unique within $\tilde{\mathbb{A}}$. These $\alpha^k$ satisfy
     Theorem \ref{thm: ncs}, $\hat\alpha^l_t = -Z^l_t$ a.e. on $\{t<\tau\}$ where $(Y^l, Z^l)$ satisfies
 %    $$\mathcal{H}(\mathbb{P}^{\hat{\alpha}} \parallel \mathbb{P}^{\alpha^k}) = \frac{1}{2} \mathbb{E}^{\hat{\alpha}}[\int_0^{T\wedge \tau} \Vert \hat{\alpha}_s - \alpha^k_s \Vert^2 ds] \leq J(\hat{\alpha}) - J^k(\alpha^k) \searrow 0$$
  %   for $k \rightarrow \infty$ and can further be characterized via the adjoint equation
    \begin{eqnarray}\label{eq: schradj}
        &&Y^l_{t\wedge \tau} = \mathbf{1}_{\lbrace T<\tau \rbrace} (g^l(\mathcal{L}_{\mathbb{P}^{\hat{\alpha}^l}}(X_T\vert T<\tau), \mathbb{P}^{\hat{\alpha}^l}[T<\tau]) + \frac{\delta g^l}{\delta m}(\mathcal{L}_{\mathbb{P}^{\hat{\alpha}^l}}(X_T\vert T<\tau), \mathbb{P}^{\hat{\alpha}^l}[T<\tau], X_T)\nonumber\\
        &+& \mathbb{P}^{\hat{\alpha}^l}[T<\tau] g^l_p(\mathcal{L}_{\mathbb{P}^{\alpha ^l}}(X_T \vert T < \tau), \mathbb{P}^{\hat{\alpha}^l}[T<\tau]))
        +\int_{t\wedge \tau}^{T\wedge\tau} \frac{1}{2} \Vert Z^l_s \Vert^2 ds - \int_{t\wedge \tau}^{T\wedge \tau} Z^l_s dW^{\hat{\alpha}^l}_s.
        %&&\hat{\alpha}^l_t =  -Z^l_t \; \text{ a.e. on } \lbrace t<\tau\rbrace.
    \end{eqnarray}
    % For the representation of the density, first 
    Note that since $Y_0$ is $\mathcal{F}_0$-measurable, there must exists a measurable map $\phi^l:D \rightarrow \mathbb{R}$, such that $-Y^l_0 = \phi^l(X_0)$. Let us also define the mapping $\psi^l:D \rightarrow \mathbb{R}$ as
    \begin{align*}
        \psi^l(x) &:= g^l(\mathcal{L}_{\mathbb{P}^{\hat{\alpha}^l}}(X_T\vert T<\tau), \mathbb{P}^{\hat{\alpha}^l}[T<\tau]) + \frac{\delta g^l}{\delta m}(\mathcal{L}_{\mathbb{P}^{\hat{\alpha}^l}}(X_T\vert T<\tau), \mathbb{P}^{\hat{\alpha}^l}[T<\tau],x)\\
        &\qquad + \mathbb{P}^{\hat{\alpha}^l}[T<\tau] g^l_p(\mathcal{L}_{\mathbb{P}^{\alpha ^l}}(X_T \vert T < \tau), \mathbb{P}^{\hat{\alpha}^l}[T<\tau]).
    \end{align*}
    It will be useful to rewrite $\psi^l$ slightly differently. 
    Let $D^* = D \cup \lbrace \zeta \rbrace$ with $\zeta$ denoting a cemetery state and $D^*$ equipped with the $\sigma$ algebra generated by the Borel $\sigma$ algebra on $D$ and $\lbrace \zeta \rbrace$. We then extend $\psi^l$ as $\tilde{\psi}^l: D^* \rightarrow \mathbb{R}, \mathbf{1}_{\lbrace x \neq \zeta \rbrace} \psi^l(x)$. Then, if we define the $D^*$-valued process $\tilde{X}_t = \mathbf{1}_{t < \tau} X_t + \mathbf{1}_{t \geq \tau} \zeta$, we can write $\frac{d\mathbb{P}^{\hat{\alpha}^l}}{d\mathbb{P}} = e^{\phi^l(X_0) + \tilde{\psi}^l(\tilde{X}_T)}$.\par
    Recall that $\mathcal{H}(\mathbb{P}^{\hat{\alpha}} \parallel \mathbb{P}^{\hat{\alpha}^l}) \rightarrow 0$. By Pinsker's inequality this implies $d_{TV}(\mathbb{P}^{\hat{\alpha}}, \mathbb{P}^{\hat{\alpha}^l}) \rightarrow 0$ and thus $\mathbb{P}^{\hat{\alpha}^l}[ \frac{d\mathbb{P}^{\hat{\alpha}}}{d\mathbb{P}^{\hat{\alpha}^l}} > 0] = \mathbb{P}^{\hat{\alpha}^l}[\frac{d\mathbb{P}^{\hat{\alpha}}}{d\mathbb{P}} > 0] \rightarrow \mathbb{P}^{\hat{\alpha}}[\frac{d\mathbb{P}^{\hat{\alpha}}}{d\mathbb{P}} > 0] = 1$. Consequently, with the usual convention that $0 \cdot \infty = 0$,
    $$\mathbb{E}^{\hat{\alpha}}\Big[\Big\vert\Big(\frac{d\mathbb{P}^{\hat{\alpha}}}{d\mathbb{P}^{\hat{\alpha}^l}}\Big)^{-1} - 1\Big\vert\Big] = \mathbb{E}^{\hat{\alpha}^l}\Big[\mathbf{1}_{\big\lbrace\frac{d\mathbb{P}^{\hat{\alpha}}}{d\mathbb{P}^{\hat{\alpha}^l}} > 0 \big\rbrace} \Big\vert 1 - \frac{d\mathbb{P}^{\hat{\alpha}}}{d\mathbb{P}^{\hat{\alpha}^l}}\Big\vert\Big] = 2d_{TV}(\mathbb{P}^{\hat{\alpha}}, \mathbb{P}^{\hat{\alpha}^l}) - \mathbb{E}^{\hat{\alpha}^l}\Big[\mathbf{1}_{\big\lbrace \frac{d\mathbb{P}^{\hat{\alpha}}}{d\mathbb{P}^{\hat{\alpha}^l}} = 0 \big\rbrace}\Big] \rightarrow 0$$
    and further, since $-\log(x) + x - 1 \geq 0$ for any $x >0$,
    \begin{eqnarray*}
        &&\mathbb{E}^{\hat{\alpha}}\Big[\Big\vert \log\Big(\frac{d\mathbb{P}^{\hat{\alpha}^l}}{d\mathbb{P}}\Big) - \log\Big(\frac{d\mathbb{P}^{\hat{\alpha}}}{d\mathbb{P}}\Big) \Big\vert\Big] \leq \mathbb{E}^{\hat{\alpha}}\Big[\Big\vert - \log\Big(\Big(\frac{d\mathbb{P}^{\hat{\alpha}^l}}{d\mathbb{P}^{\hat{\alpha}}}\Big)^{-1}\Big) + \Big(\frac{d\mathbb{P}^{\hat{\alpha}}}{d\mathbb{P}^{\hat{\alpha}^l}}\Big)^{-1} - 1 \Big\vert\Big] \\
        &+& \mathbb{E}^{\hat{\alpha}}\Big[\Big\vert \Big(\frac{d\mathbb{P}^{\hat{\alpha}}}{d\mathbb{P}^{\hat{\alpha}^l}}\Big)^{-1} - 1 \Big\vert\Big] = \mathcal{H}(\mathbb{P}^{\hat{\alpha}}\parallel\mathbb{P}^{\hat{\alpha}^l}) + \mathbb{P}^{\hat{\alpha}^l}\Big[\frac{d\mathbb{P}^{\hat{\alpha}}}{d\mathbb{P}^{\hat{\alpha}^l}} > 0\Big] - 1 + \mathbb{E}^{\hat{\alpha}}\Big[\Big\vert \Big(\frac{d\mathbb{P}^{\hat{\alpha}}}{d\mathbb{P}^{\hat{\alpha}^l}}\Big)^{-1} - 1 \Big\vert\Big] \rightarrow 0.
    \end{eqnarray*}
Therefore, $\phi^l(X_0) + \tilde{\psi}^l(\tilde{X}_T)$ must admit an $\mathbb{L}^1$ limit with respect to $\mathbb{P}^{\hat{\alpha}}$.\par
     To characterize the limit let us first discuss the joint law of $(X_0, \tilde{X}_T)$ under $\mathbb{P}$. Note that $\mathcal{L}_{\mathbb{P}}(\tilde{X}_T) = \mathbb{P}[T \geq \tau] \delta_{\zeta} + \mathbb{P}[T<\tau]\mathcal{L}_{\mathbb{P}}(X_T\vert T<\tau)$. Note $\mathbb{P}[T < \tau]$, as well as $\mathbb{P}[T <\tau \vert X_0 = x_0]$ for $\nu$ a.e. $x_0 \in D$, are both strictly in $(0, 1)$. Further, using the characterization of $\mathcal{L}_\mathbb{P}(X_T\vert T<\tau)$ in \cite[Chapter 2.VII.9]{Doob84}, $\mathcal{L}_\mathbb{P}(X_T\vert T<\tau)$, as well as $\mathcal{L}_\mathbb{P}(X_T \vert T<\tau, X_0 = x_0)$ for $\nu$ a.e. $x_0$, are equivalent to the Lebesgue measure on $D$. Therefore, we have $\mathcal{L}_\mathbb{P}(X_0, \tilde{X}_T) \ll \nu \times (\mathbb{P}[T \geq \tau]\delta_\zeta + \mathbb{P}[T<\tau] \mathcal{L}_\mathbb{P}(X_T\vert T<\tau))$. As $\mathbb{P}^{\hat{\alpha}} \ll \mathbb{P}$, this allows us to argue that $\mathcal{L}_\mathbb{P}^{\hat{\alpha}}(X_0, \tilde{X}_T) \ll \nu \times (\mathbb{P}^{\hat{\alpha}}[T \geq \tau]\delta_\zeta + \mathbb{P}^{\hat{\alpha}}[T<\tau] \hat{\mu})$. Consequently, we can apply \cite[Proposition 2]{Ruschendorf} to see that there must exist measurable $\hat{\phi}: D \rightarrow R$ and $\tilde{\psi}: D^* \rightarrow \mathbb{R}$ such that $\mathbb{P}^{\hat{\alpha}}$ a.s., we have $\log(\frac{d\mathbb{P}^{\hat{\alpha}}}{d\mathbb{P}}) = \hat{\phi}(X_0) + \tilde{\psi}(\tilde{X}_T)$. By shifting $\hat{\phi}$ and $\tilde{\psi}$ by constants, we can w.l.o.g.\ assume $\tilde{\psi}(\zeta) = 0$. Thus, defining $\hat{\psi} = \tilde{\psi}_{\vert D}$, we have $\tilde{\psi}(\tilde{X}_T) = \mathbf{1}_{\lbrace T<\tau \rbrace}\hat{\psi}(X_T)$ which concludes the proof.    
\end{proof}

\subsection{Standard McKean-Vlasov control problems}\label{sct: uncnd}
We developed the weak formulation for McKean-Vlasov control to handle the missing regularity coming from conditioned mean field interactions. Although our focus has been on the conditioned case, our results also provide an interesting new approach to the standard unconditioned problem. To recover the unconditioned case from our framework, let the domain be $D = \mathbb{R}^d$. 
In this case, $\tau = \infty$, $\mathcal{L}_{\mathbb{P}^\alpha}(X_{\cdot \wedge t}\vert t < \tau) = \mathcal{L}_{\mathbb{P}^\alpha}(X_{\cdot \wedge t})$ and $\mathbb{P}^\alpha[t < \tau] = 1$ for any $t$. The dependence of the cost in $\mathbb{P}^\alpha[t < \tau]$ becomes superfluous.\par
Our state process is then a weak solution of
$$X_t = \xi + \int_0^t b(s, X_{\cdot \wedge s}, \alpha_s, \mathcal{L}_{\mathbb{P}^\alpha}(X_{\cdot \wedge s}))ds + \int_0^t \sigma(s, X_{\cdot \wedge s})dW^\alpha_s$$
and the cost functional becomes
$$J(\alpha) := \mathbb{E}^\alpha\bigg[\int_0^{T \wedge\tau} f(s, X_{\cdot \wedge s}, \alpha_s, \mathcal{L}_{\mathbb{P}^\alpha}(X_{\cdot \wedge s})) ds + g(X, \mathcal{L}_{\mathbb{P}^\alpha}(X))\bigg].$$
In this case, the adjoint equation takes the easier form
\begin{align}
\notag
    Y^\alpha_t &= g(\Theta^\alpha_T) + \tilde{\mathbb{E}}^\alpha\Big[\frac{\delta g}{\delta m}(\tilde{\Theta}^\alpha_T, X) \Big] + \int_t^T f(\Theta^\alpha_s) + \tilde{\mathbb{E}}^\alpha\Big[\frac{\delta f}{\delta m}(\tilde{\Theta}^\alpha_s, X_{\cdot \wedge s}) + \frac{\delta \beta}{\delta m}(\tilde{\Theta}^\alpha_s, X_{\cdot \wedge s})^\top\tilde{Z}^\alpha_s\Big]ds\\ \notag 
    &\quad- \int_t^T Z^\alpha_s dW^\alpha_s\\\label{eq: unc adj}
    &=g(\Theta^\alpha_T) + \tilde{\mathbb{E}}^\alpha\Big[\frac{\delta g}{\delta m}(\tilde{\Theta}^\alpha_T, X)\Big] + \int_t^T h(\Theta^\alpha_s, Z^\alpha_s) + \tilde{\mathbb{E}}^\alpha\Big[\frac{\delta h}{\delta m}(\tilde{\Theta}^\alpha_s, \tilde{Z}^\alpha_s, X_{\cdot \wedge s})\Big]ds - \int_t^T Z^\alpha_s dW_s.
\end{align}
Moreover, the Pontryagin maximum principle we established is also applicable to unconditioned McKean-Vlasov control, and our discussion on the optimal control remains true as well. In the unconditioned setting, the $p$-convexity required in Assumption \ref{asmp: cnv} for the sufficient condition, as well as our existence results, reduces to requiring $\int f^2(t, x, \mu)\mu(dx)$ and $\int g(x, \mu)\mu(dx)$ to be linearly convex as functions in $\mu$. Furthermore, as we no longer worry about the possibility that the state a.s.\ leaves the domain, one can directly formulate the problem in the setting of theorem \ref{thm: opt non eq}, allowing for non-equivalent changes of measures from the beginning without having to additionally assume extendability.

%{\color{blue}add a remark that the mean of $Y_0$ is the value function.}

\subsection{Potential mean field games}\label{sct: potgms}
We conclude by presenting the link with mean field games.
In fact, the Pontryagin maximum principle derived in this paper allows to derive solutions of a class of mean field games called potential mean field games.
Although this result may seem folklore, the first general construction was only recently obtained by \cite{Hofer2}.
We will extend and deduce these results as a byproduct of our necessary condition for optimality together with results in \cite{PossamaiTangpi}.

  % given a McKean-Vlasov control problem, we consider its potential game as a mean field game constructed out of the McKean-Vlasov control problem such that any optimal control for the McKean-Vlasov control problem provides a solution to the mean field game. In \cite{Hofer2}, they provide such a construction that happens to reappear within the newly found Pontryagin maximum principle. In fact, the necessary condition implies that any optimal control corresponds to a solution to of a generalized McKean-Vlasov BSDE such as it was introduced in \cite{PossamaiTangpi} to characterize solutions of mean field games, allowing us to rediscover this connection between control problems and games. Such a connection can of course also be formulated in the constrained setting,  but let us present the connection here in the unconstrained setting to compare with existing results.\par

\begin{proof}[Proof of Theorem \ref{thm: pot mfg}]
    Let us start by proving the first statement. We are only going to present the proof for the case where $a^*$ exists and is invertible, the other case follows by the same steps. Let $\alpha$ be any control that is optimal for $J$ over $\mathbb{A}_{BMO}$ and write $\nu = (\mathcal{L}_{\mathbb{P}^\alpha}(X_{\cdot \wedge t}, \alpha_t))_{t \in [0, T]}$. For any other $\alpha' \in \mathbb{A}_{BMO}$, we define $(\overline{Y}^{\nu, \alpha'}, \overline{Z}^{\nu, \alpha'})$ as the solution of
    \begin{align*}
        \overline{Y}^{\nu, \alpha'}_t &= g(X, \nu^x_T) + \int_{\mathcal{C}_D}\frac{\delta g}{\delta m}(\tilde{x}, \nu^x_T, X)\nu^x_T(d\tilde{x}) + \int_t^T f(s, X_{\cdot \wedge s}, \alpha'_s, \nu^x_s) + \int_{\mathcal{C}_D \times A}\frac{\delta f}{\delta m}(s, \tilde{x}, \tilde{a}, \nu^x_s, X_{\cdot \wedge s}) \\
        &\quad + \frac{\delta \beta}{\delta m}(s, \tilde{x},\tilde{a}, \nu^x_s, X_{\cdot \wedge s})^\top (a^*)^{-1}(s, \tilde{x}, \tilde{a}, \nu^x_s)\nu_s(d\tilde{x}, d\tilde{a})ds - \int_t^T \overline{Z}^{\nu, \alpha'}_s d\overline{W}^{\nu, \alpha'}_s\\
        &= g(X, \nu^x_T) + \int_{\mathcal{C}_D}\frac{\delta g}{\delta m}(\tilde{x}, \nu^x_T, X)\nu^x_T(d\tilde{x}) + \int_t^T h(s, X_{\cdot \wedge s}, \alpha'_s, \nu^x_s, \overline{Z}^{\nu, \alpha'}_s) + \int_{\mathcal{C}_D \times A}\frac{\delta f}{\delta m}(s, \tilde{x}, \tilde{a}, \nu^x_s, X_{\cdot \wedge s}) \\
        &+ \frac{\delta \beta}{\delta m}(s, \tilde{x},\tilde{a}, \nu^x_s, X_{\cdot \wedge s})^\top (a^*)^{-1}(s, \tilde{x}, \tilde{a}, \nu^x_s)\nu_s(d\tilde{x}, d\tilde{a})ds - \int_t^T \overline{Z}^{\nu, \alpha'}_s dW_s
    \end{align*}
    which is immediately found using the martingale representation theorem applied with respect to $\mathbb{P}^{\alpha'}$ and $W^{\alpha'}$. Note that $\mathbb{E}[\overline{Y}^{\nu, \alpha'}_0] = \mathbb{E}^{\overline{\mathbb{P}}^{\nu, \alpha'}}[\overline{Y}^{\nu, \alpha'}_0] = J^{\mathrm{MFG}}(\alpha', \nu)$, as well as $(\overline{Y}^{\nu, \alpha}, \overline{Z}^{\nu, \alpha}) = (Y^\alpha, Z^\alpha)$ with $(Y^\alpha, Z^\alpha)$ being the solution to \eqref{eq: adj eq} as found in Proposition \ref{prop: wp adj}. For any $\alpha' \in \mathbb{A}_{BMO}$ it is easy to check that under our growth conditions, $\overline{Y}^{\nu, \alpha'}$ must be uniformly bounded. Further,  applying Itô's lemma to $(\overline{Y}^{\nu, \alpha'})^2$ shows that $\mathbb{E}[\int_0^T \Vert \overline{Z}^{\nu, \alpha'}_s \Vert^2 ds] < \infty$.\par
    By Theorem \ref{thm: ncs}, $\alpha_t \in \arg\max_{a \in A} h(t, X_{\cdot \wedge t}, a, \mathcal{L}_{\mathbb{P}^\alpha}(X_{\cdot \wedge t}), Z^\alpha_t)$. 
    Let us define $H(t, x, \mu, z) = \min_{a \in A} h(t, x, a, \mu, z)$. Using strong convexity of $f$ and Lipschitzness of $\beta$ in $a$, one can check that $a^*$ must be Lipschitz in $z$. which shows that $H$ admits quadratic growth in $z$. Therefore, $(Y^\alpha, Z^\alpha)$ is also a solution to the (non McKean-Vlasov) quadratic BSDE
    \begin{align*}
        Y^\alpha_t =& g(X, \nu^x_T) + \int_{\mathcal{C}_D} \frac{\delta g}{\delta m}(\tilde{x}, \nu^x_T, X)\nu^x_T(d\tilde{x}) + \int_t^T  H(s, X_{\cdot \wedge s}, \nu^x_s, Z^\alpha_s)  + \int_{\mathcal{C}_D \times A}\frac{\delta f}{\delta m}(s, \tilde{x}, \tilde{a}, \nu^x_s, X_{\cdot \wedge s}) \\
        &\quad + \frac{\delta \beta}{\delta m}(s, \tilde{x},\tilde{a}, \nu^x_s, X_{\cdot \wedge s})^\top (a^*)^{-1}(s, \tilde{x}, \tilde{a}, \nu^x_s)\nu_s(d\tilde{x}, d\tilde{a})ds - \int_t^T Z^\alpha_s dW_s.
    \end{align*}
    Since we for any $t \geq 0$ we a.s. have $H(t, X_{\cdot \wedge t}, \nu^x_t, \overline{Z}^{\nu, \alpha'}_t) \leq h(t, X_{\cdot \wedge t}, \alpha'_t, \nu^x_t, \overline{Z}^{\nu, \alpha'}_s)$, by the comparison principle for quadratic BSDEs \cite[Theorem 2.6]{Kobylanski00}, we have that a.s. $Y^\alpha_0 \leq \overline{Y}^{\nu, \alpha'}_0$ and thus $J^{\mathrm{MFG}}(\alpha, \nu) \leq J^{\mathrm{MFG}}(\alpha', \nu)$, showing that $(\alpha, \nu)$ is a solution to the potential mean field game.\par
    Now, in case we additionally assumed \ref{asmp: cnv}, applying Proposition \ref{prop: pcnv diff} to functions independent of $p$ shows that linear convexity of $\int_{\mathcal{C}_D} g(x, \mu) d\mu(dx)$ implies for any $\mu$ and $\mu'$ that
    % $$\int_{\mathcal{C}_D} g(x, \mu') \mu'(dx) - \int_{\mathcal{C}_D} g(x, \mu)\mu(dx) \geq \int_{\mathcal{C}_D} g(x, \mu) + \int_{\mathcal{C}_D} \frac{\delta g}{\delta m}(\tilde{x}, \mu, x)\mu(d\tilde{x})(\mu' - \mu)(dx)$$
    Summing up this inequality with the one derived by swapping $\mu$ and $\mu'$ shows that the function $g(x, \mu) + \int_{\mathcal{C}_D}\frac{\delta g}{\delta m}(\tilde{x}, \mu, x) \mu(d\tilde{x})$ is Lasry-Lions monotone in the sense that
    $$\int_{\mathcal{C}_D}\Big\{ g(x, \mu') + \int_{\mathcal{C}_D}\frac{\delta g}{\delta m}(\tilde{x}, \mu', x) \mu'(d\tilde{x}) - g(x, \mu) - \int_{\mathcal{C}_D}\frac{\delta g}{\delta m}(\tilde{x}, \mu, x) \mu(d\tilde{x})\Big\}(\mu'-\mu)(dx) \geq 0$$
    for any $\mu$, $\mu'$. The same discussion applies to $f^2$ as well.\par
    Then, following the argument in \cite[Theorem 3.8]{Carmona15} implies that the potential mean field game admits at most one solution. (Although \cite[Theorem 3.8]{Carmona15} assumes bounded controls, working with controls in $\mathbb{A}_{BMO}$ still ensures enough integrability to use the argument.)
\end{proof}
\begin{remark}
    When $A$ is bounded and accordingly $b$ is independent of $\mu$, some conditions can be weakened. For the first statement, $f$ only needs to be convex, not strongly, and the additional bounds on $f$, $g$, $\frac{\delta f}{\delta m}$, $\frac{\delta g}{\delta m}$ and $\frac{\delta \beta}{\delta m}$ beyond the ones in assumption \ref{asmp: beta} and \ref{asmp: ncs} are not required. This is since in this case there is no need to work with quadratic BSDEs but with Lipschitz BSDEs, in a way that is essentially covered in \cite[Proposition 2.8]{PossamaiTangpi}. The second statement also holds without the extra bounds, and will only require $f$ to be strictly convex to ensure that the Hamiltonian is uniquely minimizable.
\end{remark}
This result is consistent with the recent findings on potential mean field games in \cite{Hofer2}. In this paper, the authors consider a more general framework allowing among other for interaction through the law of control, as well as common noise. Notably, convexity of the running cost in $a$ is not required. %Still, we can see that our results lead to the same construction of the potential game. 
Notwithstanding, our derivation of the result differ fundamentally. \cite{Hofer2} relies on first principle arguments differentiating an additionally introduced randomization at initial time whereas we derive the result from Pontryagin's maximum principle.\par

\bibliographystyle{plain}
\bibliography{refs.bib}
\end{document}